\documentclass[twoside,12pt,reqno]{amsart}
\usepackage{amsmath,amscd,amssymb,epsfig,subfig}
\usepackage[all]{xy}

\usepackage{verbatim}
\usepackage{graphics,graphicx}
\usepackage{tikz}
\usetikzlibrary{positioning}
\newtheorem{definition}{Definition}
\newtheorem{theorem}[definition]{Theorem}

\newtheorem{proposition}[definition]{Proposition}
\newtheorem{lemma}[definition]{Lemma}
\newtheorem{remark}[definition]{Remark}

\newtheorem{corollary}[definition]{Corollary}
\newtheorem{conjecture}[definition]{Conjecture}

\newcommand{\C}{\ensuremath{\mathbb{C}} }
\newcommand{\Z}{\ensuremath{\mathbb{Z}} }

\newcommand{\cat}{\mathcal{C}}

\newcommand{\End}{\operatorname{End}}
\newcommand{\Hom}{\operatorname{Hom}}

\newcommand{\Aut}{\operatorname{Aut}}
\newcommand{\Ker}{\operatorname{Ker}}
\newcommand{\unit}{\ensuremath{\mathbb{I}}}

\newcommand{\Id}{\operatorname{Id}}

\newcommand{\FK}{\mathsf{k}}

\newcommand{\sll}{\mathfrak{sl}}

\newcommand{\epsh}[2]
         {\begin{array}{c} \hspace{-1.3mm}
        \raisebox{-4pt}{\epsfig{figure=#1,height=#2}}
        \hspace{-1.9mm}\end{array}}

\newcommand{\Aop}{A}
\newcommand{\Dop}{D}
\newcommand{\Aops}{\ensuremath{\mathsf{A}}}
\newcommand{\Bop}{B}
\newcommand{\Bops}{\ensuremath{\mathsf{B}}}
\newcommand{\Cop}{C}
\newcommand{\Lop}{L}
\newcommand{\Rop}{R}

\newcommand{\Qop}{Q}

\newcommand{\act}[2]{#1^{#2}}
\newcommand{\teq}{\stackrel{T}{=}}
\newcommand{\steq}{\stackrel{sT}{=}}
\newcommand{\spsi}{\widehat{\Psi}}
\newcommand{\slop}[1]{\Lop^{#1\frac12}}
\newcommand{\srop}[1]{\Rop^{#1\frac12}}
\newcommand{\scop}[1]{C^{#1\frac12}}
\newcommand{\sqop}[2]{q^{#1}}
\newcommand{\wsqop}{\widetilde q}
\newcommand{\T}{{\mathcal{T}}}
\newcommand{\LL}{\mathcal{L}}
\newcommand{\bubble}{$H$-bubble}
\newcommand{\Pachner}{$H$-Pachner}

\newcommand{\Gr}{G}

\newcommand{\p}{\varphi}
\newcommand{\vect}{\overrightarrow}
\newcommand{\cntr}{{\operatorname {cntr}}}
\newcommand{\bb}{\operatorname{\mathsf{b}}}
\newcommand{\states}{\operatorname{St}}
\renewcommand{\wp}{{\Phi}}
\newcommand{\MG}[1]{{\mathcal M}(#1,\Gr)}
\newcommand{\CH}{\mathcal{R}}
\newcommand{\bv}{w}
\newcommand{\Kas}{\mathsf{K}}
\newcommand{\rofo}{\varpi}

\newcommand{\sj}[6]{\left\{\begin{array}{ccc}#1 & #2 & #3 \\#4 & #5 &
      #6\end{array}\right\}}
      \newcommand{\sjn}[6]{\left\{\begin{array}{ccc}#1 & #2 & #3 \\#4 & #5 &
      #6\end{array}\right\}^{-}}

\begin{document}
\let\co=\comment \let\endco=\endcomment
\title[]{Tetrahedral forms in monoidal categories and 3-manifold invariants}
\author{Nathan Geer}
\address{Mathematics \& Statistics\\
  Utah State University \\
  Logan, Utah 84322, USA and\\
Max-Planck-Institut f\"{u}r Mathematik\\
Vivatsgasse 7\\
53111 Bonn, Germany}
 \email{nathan.geer@usu.edu}
  \author{Rinat Kashaev}
\address{Section de Math\'ematiques\\
Universit\'e de Gen\`eve\\
2-4, rue du Li\`evre\\
Case postale 64\\
1211 Gen\`eve 4, Suisse }
\email{Rinat.Kashaev@unige.ch}
\author{Vladimir Turaev}
\address{Department of Mathematics \\
Indiana University \\
Rawles Hall, 831 East 3rd St \\
Bloomington, IN 47405, USA}
\email{vtouraev@indiana.edu}

\begin{abstract}
 We introduce systems of objects and operators  in linear monoidal categories   called  $\widehat \Psi$-systems. A $\widehat \Psi$-system satisfying several additional assumptions gives rise to a topological invariant  of triples (a closed oriented 3-manifold $M$, a principal bundle over $M$, a link in $M$).  This construction generalizes the quantum dilogarithmic invariant of links appearing in the original formulation of the volume conjecture.
  We conjecture that all quantum groups at odd roots of unity give rise to  $\widehat \Psi$-systems and we verify this conjecture in the case of the Borel subalgebra of quantum  $\sll_2$. \end{abstract}

  \date{\today}

\maketitle
\setcounter{tocdepth}{1}

\section*{Introduction}

One of the fundamental  achievements of quantum topology was   a
discovery of deep connections between   monoidal categories and
  3-dimensional manifolds. It was first observed by O.~Viro and V.~Turaev
that  the   category of representations of the quantum
  group $U_q(\sll_2)$ gives rise to a topological invariant of
  3-manifolds.  The invariant is obtained as a    state sum  on  a triangulation  of a 3-manifold; the key ingredients of the  state sum   are
   the
  $6j$-symbols. This construction  was generalized to other categories by several authors
  including  J.~Barrett, B.~Westbury, A.~Ocneanu, S.~Gelfand,
  D.~Kazhdan and others. Their results may be summarized by saying
  that every spherical fusion category gives rise to a state sum  3-manifold
  invariant. Similar methods apply
  to links in 3-manifolds and to 3-manifolds endowed with
  principal fiber bundles. A related but somewhat different line of
  development was initiated by  Kashaev \cite{K1}. He  defined a state sum invariant of links in
  3-manifolds using
  \lq\lq charged" versions of the  $6j$-symbols   associated with certain representations of the
  Borel subalgebra of
$U_q(\sll_2)$. The work of Kashaev was further extended by
S.~Baseilhac and R.~Benedetti, see \cite{Bas}, \cite{BB}.

The aim of this paper is to analyze categorical foundations of the
Kashaev-Baseilhac-Benedetti theory. The key new notions in our
approach are the ones of $\Psi$-systems and $\widehat \Psi$-systems
in linear monoidal categories. The $\Psi$-systems provide a
general framework for  $6j$-symbols. Roughly speaking, a
$\Psi$-system is a family of simple objects of the category
$\{V_i\}_{i\in I}$ closed under duality and such that for \lq\lq
almost all" $i,j\in I$, the identity endomorphism of $V_i\otimes
V_j$ splits as a sum of certain compositions $\{V_i\otimes V_j \to
V_k \to V_i\otimes V_j\}_{k\in I}$  (see Section \ref{section1}).
Examples of $\Psi$-systems can be derived from quantum groups at
roots of unity or, more generally, from Cayley-Hamilton Hopf
algebras, see Section \ref{section12}. Every fusion category has a
$\Psi$-system formed by arbitrary representatives of the isomorphism
classes of all simple objects. A $\Psi$-system in a linear monoidal
category gives rise to a vector space $H$ (the space of
multiplicities), a linear form $T$ on $H^{\otimes 4}$ (the
tetrahedral evaluation form), and two automorphisms $A, B$ (obtained
by taking  adjoints of morphisms). The vector space $H$ has a
natural symmetric bilinear form which allows us to consider the
transposes $A^*$, $B^*$ of $A$, $B$. We use   $T$ to define
$6j$-symbols and we use  $A,B, A^*, B^* $ to formulate the
tetrahedral symmetry of the $6j$-symbols. We also develop a
$T$-calculus for endomorphisms of $H$ which allows us to speak of
equality/commutation of operators \lq\lq up to composition with
$T$". These definitions and results occupy Sections
\ref{section1}--\ref{section6}.

 To define 3-manifold invariants we need to fix  square roots of the operators $$L=A^*A , \,\,
R=B^*B, \,\, C=(AB)^3 \in \End(H) .$$  A $\Psi$-system endowed with
such square roots $\slop{}$, $\srop{}$,   $\scop{}$ satisfying
appropriate  relations is said to be a $\widehat \Psi$-system.   The
$\widehat \Psi$-systems provide a general framework for so-called
\lq\lq charged"  $6j$-symbols depending on two additional integer of
half-integer parameters. The advantage of the charged $6j$-symbols
lies in the simpler tetrahedral symmetry. This material occupies
Sections \ref{section7}--\ref{section9}.

We need two  assumptions on a $\widehat \Psi$-system to produce a
3-manifold invariant. The first assumption says essentially that the
operators $\slop{}$ and $\srop{}$  commute   up to composition with
$T$  and multiplication by a certain scalar $\widetilde q$. The
second assumption introduces  additional data: a group $G$  and a
family of finite subsets $\{I_g\}_{g\in G}$ of $I$ satisfying
certain conditions. We use this data to define a numerical
topological invariant of any tuple (a closed connected oriented
3-manifold $M$, a non-empty link $L\subset M$, a conjugacy class of
homomorphisms $\pi_1(M)\to G$, an element of $H^1(M;\Z/2\Z)$), see
Sections \ref{section10}, \ref{section11}. The invariant in question
is defined as a state sum on a Hamiltonian triangulation of $(M,L)$.
To encode the Hamiltonian path $L$ into the state sum, we use the  charges on $H$-triangulations first introduced in \cite{K1}.
The theory of charges subsequently has been developed in \cite{BB}.  It is a natural extension of the theory of angle structures due to
W.~Neumann, see, for example, \cite{Neu90,Neu98}. The key ingredients of our
state sum are the charged $6j$-symbols. The resulting invariant is
well-defined up to multiplication by integer powers of $\widetilde
q$.

We conjecture that the $\Psi$-systems associated with quantum groups and their Borel subalgebras
at odd roots of unity   extend to $\widehat \Psi$-systems satisfying all
our requirements. We verify this conjecture in the case of the Borel subalgebra of  ${U}_q(\sll_2)$, see
Sections~\ref{section12}, \ref{section13}.  Geer and Patureau-Mirand \cite{GP} verify the conjecture for all quantum groups associated to simple Lie algebras and prove that the usual modular categories arising from quantum groups have $\widehat \Psi$-systems satisfying all our requirements.  The conjecture is open for Borel subalgebras of quantum groups other than ${U}_q(\sll_2)$.
  We expect that the associated invariants are closely related with the invariants constructed in \cite{GPT2,KR}.
  In the case of the example of Section~\ref{section13} with the trivial homomorphism $\pi_1(M)\to G$,
  this construction coincides with the one of Kashaev \cite{K1}. The latter invariant enters the volume conjecture \cite{K3},
  and for links in $S^3$, it is a specialization of the colored Jones polynomial \cite{MM}.
Precise relationships of our invariants with the Baseilhac--Benedetti 3-manifold
invariants are yet unclear.

\begin{co}
I have used the environment {co} to automaticly
show/hide some comments writen in this block form.
To hind such comments just change the lines
\let\co=\comment \let\endco=\endcomment
Note that there are other comments labled by 
\end{co}
This work is partially supported by the Swiss National Science
Foundation. The work of N.\ Geer was partially supported by the NSF
  grant DMS-0706725 and the work of V.\ Turaev was partially supported by the NSF
  grants DMS-0707078 and DMS-0904262.  NG and VT would like to thank the University of Geneva
  and RK would like to thank the Indiana University for hospitality.

\section{$\Psi$-systems   in monoidal categories}\label{section1}
\subsection{Monoidal Ab-categories} A \emph{monoidal (tensor) category} $\cat$ is a category equipped with a
covariant bifunctor $\boxtimes :\cat \times \cat\rightarrow \cat$ called the
tensor product, an associativity constraint, a unit object $\unit$, and left
and right unit constraints such that the Triangle and Pentagon Axioms hold.
When the associativity constraint and the left and right unit
constraints are all identities,  the category $\cat$ is a
\emph{strict} monoidal (tensor) category. By MacLane's coherence
theorem, any monoidal category is equivalent to a strict monoidal
category. To simplify the exposition, we formulate further
definitions only for strict monoidal categories; the reader will
easily extend them to arbitrary monoidal categories.

A monoidal category $\cat$ is said to be an \emph{Ab-category} if
for any objects $V,W$ of $\cat$, the set of morphisms $\Hom(V,W)$ is
an additive abelian group and the composition and tensor product of
morphisms are bilinear.  Composition of morphisms induces a
commutative ring structure on the abelian group $\FK=\End(\unit)$.
The resulting ring is called the \emph{ground ring} of $\cat$. For
any objects $V,W$ of $\cat$ the abelian group $\Hom(V,W)$ becomes a
left $\FK$-module via $kf=k\boxtimes f$ for $k\in \FK$ and $f\in
\Hom(V,W)$. We assume that the tensor multiplication of morphisms in $\cat$ is $\FK$-bilinear.

An object $V$ of $\cat$ is   \emph{simple} if $\End(V)=
\FK \Id_V$. For any simple object $V$ and $f\in\End(V)$, there is a
unique $k\in \FK$ such that $f=k\Id_{V}$. This $k$ is denoted $\langle
f\rangle$.

Fix from now on a monoidal Ab-category $\cat$ whose ground ring $\FK$
is a field. We shall use the symbol $\otimes$ for the tensor product
of $\FK$-vector spaces over   $\FK$ and the symbol $\boxtimes$ for the
tensor product in $\cat$.

\subsection{$\Psi$-systems} A
\emph{$\Psi$-system}\footnote{This name is inspired by the logo of
the Indiana University, where the definition has been finalized.}
 in $\cat$ consists  of
 \begin{itemize}
  \item[(1)] a distinguished set of simple objects $\{ V_i\}_{ i\in I}$ such that  $\Hom(V_i,V_j)=0$ for all $i\ne j$;
  \item[(2)] an involution $I\to I$, $i\mapsto i^*$;
  \item[(3)] two families of morphisms
\(\{ b_i\colon \unit\to V_i\boxtimes V_{i^*}\}_{i\in I}\) and \(
\{d_i\colon V_{i}\boxtimes V_{i^*}\to\unit \}_{i\in I}\),   such
that for all $i\in I$,
\begin{equation}\label{E:dualitybd}
(\Id_{V_i}\boxtimes \,d_{i^*})(b_i\boxtimes \Id_{V_i})=\Id_{V_i} \quad {\rm {and}} \quad (d_i\boxtimes \Id_{V_i})(\Id_{V_i}
\boxtimes \,b_{i^*})=\Id_{V_i}.
\end{equation}
\end{itemize}
To formulate the fourth (and the last) requirement on the
$\Psi$-systems,   set
\[ H^{ij}_k=\Hom(V_k,V_i \boxtimes V_j) \quad {\rm {and}} \quad
H_{ij}^k = \Hom(V_i\boxtimes V_j,V_k).
\] for any $i,j,k\in I$. We require that

\begin{enumerate}
  \item[(4)]  for any $i,j\in I$  such that $H^{ij}_k\ne0$ for some $k\in I$,
the identity morphism  $\Id_{V_i\boxtimes V_j}$ is in the image of
the linear map
\begin{equation}\label{E:isom}
\bigoplus_{k\in I} H^{ij}_k\otimes H_{ij}^k\to \End(V_i\boxtimes V_j),\quad x\otimes y\mapsto x\circ y.
\end{equation}
\end{enumerate}

We fix from now on a  $\Psi$-system in $\cat$ and keep the notation
introduced above.

\begin{lemma}\label{L:1}
For any $i,j,k\in I$, the linear spaces $H^{ij}_k$ and $H_{ij}^k$
are finite dimensional, and the bilinear pairing
\begin{equation}\label{E:pair}
H_{ij}^k\otimes H^{ij}_k\to \FK,\quad
x\otimes y\mapsto \langle x \circ y\rangle,
\end{equation}
is non-degenerate. In particular, $\dim H_{ij}^k=\dim H^{ij}_k$.
\end{lemma}
\begin{lemma}\label{L:0}
For any $i,j\in I$, there are only finitely many $k\in I$ such that
$H^{ij}_k \ne0$.
\end{lemma}
These lemmas will be proven  in the next subsection after a little
  preparation.
\subsection{The operators $\Aop$ and $\Bop$}\label{section 2.3}
Consider the vector space $H=\hat H\oplus\check H$, where
\[
\hat H=\bigoplus_{i,j,k\in I}H_{ij}^k\quad {\rm {and}} \quad \check H=\bigoplus_{i,j,k\in I}H^{ij}_k.
\]
Let
\[
\pi_{ij}^k\colon H\to H_{ij}^k,\quad  \pi^{ij}_k\colon H\to H^{ij}_k,\quad \hat\pi\colon H\to \hat H,\quad \check\pi\colon H\to \check H
\]
be the obvious  projections. We define  linear maps \(
\Aop,\Bop:H\to H \) by
\[
\Aop x=\sum_{i,j,k\in I}((\Id_{V_{i^*}}\boxtimes \, \pi_{ij}^kx)(b_{i^*}\boxtimes\Id_{V_j})+
(d_{i^*}\boxtimes\Id_{V_j})(\Id_{V_{i^*}}\boxtimes \, \pi^{ij}_kx)),
\]
\[
\Bop x=\sum_{i,j,k\in I} ((\pi_{ij}^kx\boxtimes\Id_{V_{j^*}})(\Id_{V_i}\boxtimes \, b_{j})+(\Id_{V_i}\boxtimes \, d_{j})(\pi^{ij}_kx\boxtimes\Id_{V_{j^*}})).
\]
For each $x\in H$, there are only finitely  many non-zero terms in
these sums, since $x$ has only finitely many non-zero components
$\pi_{ij}^kx$ and $\pi^{ij}_kx$. We can represent the  definitions
of $A$ and $B$  in the following graphical form:
\[
\begin{tikzpicture}[baseline=(x.base),hvector/.style={draw=blue!50,fill=blue!20,thick}]
\node (x) [hvector] {$\Aop x$};
\node (i) [above= of x.west ]{};
\node (j) [above=of x.east ]{};
\node (k) [below=of x.center ]{};
\draw (x.north west) to node[auto,inner sep=1pt]{\tiny $i$} (i);
\draw (x.north east) to node[auto,swap,inner sep=1pt]{\tiny $j$} (j);
\draw (x.south) to node[auto,inner sep=1pt]{\tiny $k$} (k);
\end{tikzpicture}
=
\begin{tikzpicture}[baseline=(x.base),hvector/.style={draw=blue!50,fill=blue!20,thick},
turn/.style={circle,draw=blue!50,fill=blue!20,thick,inner sep=2pt}]
\node (x) [hvector] {$x$};
\node (bi*) [turn,below left= of x.east,anchor=west]{};
\node (i) [above=2 of bi*.west ]{};
\node (j) [above=of x.center ]{};
\node (k) [below=of x.east ]{};
\draw (x.south east) to node[auto,inner sep=1pt]{\tiny $k$} (k);
\draw (x.south west) to node[auto,inner sep=1pt]{\tiny $i^*$} (bi*.east);
\draw (x.north) to node[auto,swap,inner sep=1pt]{\tiny $j$} (j);
\draw (bi*.west) to node[auto,inner sep=1pt]{\tiny $i$} (i);
\end{tikzpicture}
,\qquad
\begin{tikzpicture}[baseline=(x.base),hvector/.style={draw=blue!50,fill=blue!20,thick}]
\node (x) [hvector] {$\Aop x$};
\node (i) [below= of x.west ]{};
\node (j) [below=of x.east ]{};
\node (k) [above=of x.center ]{};
\draw (x.south west) to node[auto,swap,inner sep=1pt]{\tiny $i$} (i);
\draw (x.south east) to node[auto,inner sep=1pt]{\tiny $j$} (j);
\draw (x.north) to node[auto,swap,inner sep=1pt]{\tiny $k$} (k);
\end{tikzpicture}
=
\begin{tikzpicture}[baseline=(x.base),hvector/.style={draw=blue!50,fill=blue!20,thick},
turn/.style={circle,draw=blue!50,fill=blue!20,thick,inner sep=2pt}]
\node (x) [hvector] {$x$};
\node (di*) [turn,above left= of x.east,anchor=west]{};
\node (i) [below=2 of di*.west ]{};
\node (j) [below=of x.center ]{};
\node (k) [above=of x.east ]{};
\draw (x.north east) to node[auto,swap,inner sep=1pt]{\tiny $k$} (k);
\draw (x.north west) to node[auto,swap,inner sep=1pt]{\tiny $i^*$} (di*.east);
\draw (x.south) to node[auto,inner sep=1pt]{\tiny $j$} (j);
\draw (di*.west) to node[auto,swap,inner sep=1pt]{\tiny $i$} (i);
\end{tikzpicture},
\]
\[
\begin{tikzpicture}[baseline=(x.base),hvector/.style={draw=blue!50,fill=blue!20,thick}]
\node (x) [hvector] {$\Bop x$};
\node (i) [above=of x.west ]{};
\node (j) [above=of x.east ]{};
\node (k) [below=of x.center ]{};
\draw (x.north west) to node[auto,inner sep=1pt]{\tiny $i$}  (i);
\draw (x.north east) to node[auto,swap,inner sep=1pt]{\tiny $j$} (j);
\draw (x.south) to node[auto,swap,inner sep=1pt]{\tiny $k$} (k);
\end{tikzpicture}
=
\begin{tikzpicture}[baseline=(x.base),hvector/.style={draw=blue!50,fill=blue!20,thick},
turn/.style={circle,draw=blue!50,fill=blue!20,thick,inner sep=2pt}]
\node (x) [hvector] {$x$};
\node (bj) [turn,below right= of x.west,anchor=east]{};
\node (i) [above=of x.center]{};
\node (j) [above=2 of bj.east ]{};
\node (k) [below=of x.west ]{};
\draw (x.south west) to node[auto,swap,inner sep=1pt]{\tiny $k$} (k);
\draw (x.south east) to node[auto,swap,inner sep=1pt]{\tiny $j^*$} (bj.west);
\draw (x.north) to node[auto,inner sep=1pt]{\tiny $i$} (i);
\draw (bj.east) to node[auto,swap,inner sep=1pt]{\tiny $j$} (j);
\end{tikzpicture}
,\qquad
\begin{tikzpicture}[baseline=(x.base),hvector/.style={draw=blue!50,fill=blue!20,thick}]
\node (x) [hvector] {$\Bop x$};
\node (i) [below=of x.west ]{};
\node (j) [below=of x.east ]{};
\node (k) [above=of x.center ]{};
\draw (x.south west) to node[auto,swap,inner sep=1pt]{\tiny $i$} (i);
\draw (x.south east) to node[auto,inner sep=1pt]{\tiny $j$} (j);
\draw (x.north) to node[auto,inner sep=1pt]{\tiny $k$} (k);
\end{tikzpicture}
=
\begin{tikzpicture}[baseline=(x.base),hvector/.style={draw=blue!50,fill=blue!20,thick},
turn/.style={circle,draw=blue!50,fill=blue!20,thick,inner sep=2pt}]
\node (x) [hvector] {$x$};
\node (dj) [turn,above right= of x.west,anchor=east]{};
\node (i) [below=of x.center]{};
\node (j) [below=2 of dj.east ]{};
\node (k) [above=of x.west ]{};
\draw (x.north west) to node[auto,inner sep=1pt]{\tiny $k$} (k);
\draw (x.north east) to node[auto,inner sep=1pt]{\tiny $j^*$} (dj.west);
\draw (x.south) to node[auto,swap,inner sep=1pt]{\tiny $i$} (i);
\draw (dj.east) to node[auto,inner sep=1pt]{\tiny $j$} (j);
\end{tikzpicture},
\]
where we use the graphical notation
\[
\pi^{ij}_kx=
\begin{tikzpicture}[baseline=(x.base),hvector/.style={draw=blue!50,fill=blue!20,thick}]
\node (x) [hvector] {$x$};
\node (i) [above=of x.west ]{};
\node (j) [above=of x.east ]{};
\node (k) [below=of x.center ]{};
\draw (x.north west) to node[auto,inner sep=1pt]{\tiny $i$} (i);
\draw (x.north east) to node[auto,swap,inner sep=1pt]{\tiny $j$} (j);
\draw (x.south) to node[auto,inner sep=1pt]{\tiny $k$} (k);
\end{tikzpicture}
,\qquad
\pi_{ij}^kx=
\begin{tikzpicture}[baseline=(x.base),hvector/.style={draw=blue!50,fill=blue!20,thick}]
\node (x) [hvector] {$x$};
\node (i) [below=of x.west ]{};
\node (j) [below=of x.east ]{};
\node (k) [above=of x.center ]{};
\draw (x.south west) to node[auto,swap,inner sep=1pt]{\tiny $i$} (i);
\draw (x.south east) to node[auto,inner sep=1pt]{\tiny $j$} (j);
\draw (x.north) to node[auto,swap,inner sep=1pt]{\tiny $k$} (k);
\end{tikzpicture},
\qquad
b_i=
\begin{tikzpicture}[baseline=(a.north),turn/.style={circle,draw=blue!50,fill=blue!20,thick,inner sep=2pt}]
\node (a) [turn]{};
\node (i) [above=of a.west ]{};
\node (i*) [above=of a.east ]{};
\draw (a.east) to node[auto,swap,inner sep=1pt]{\tiny $i^*$} (i*);
\draw (a.west) to node[auto,inner sep=1pt]{\tiny $i$} (i);
\end{tikzpicture}
,\qquad
d_i=
\begin{tikzpicture}[baseline=(a.south),turn/.style={circle,draw=blue!50,fill=blue!20,thick,inner sep=2pt}]
\node (a) [turn]{};
\node (i) [below=of a.west ]{};
\node (i*) [below=of a.east ]{};
\draw (a.east) to node[auto,inner sep=1pt]{\tiny $i^*$} (i*);
\draw (a.west) to node[auto,swap,inner sep=1pt]{\tiny $i$} (i);
\end{tikzpicture}.
\]
\begin{lemma}\label{L:ab}
The operators $\Aop$ and $\Bop$ are involutive and
satisfy the ``exchange" relations:
\begin{equation}\label{E:exch}
\pi^{ij}_k\Aop=\Aop\pi^j_{i^*k},\quad\pi^{ij}_k\Bop=\Bop\pi^i_{kj^*},\quad
\pi_{ij}^k\Aop=\Aop\pi_j^{i^*k},\quad\pi_{ij}^k\Bop=\Bop\pi_i^{kj^*}.
\end{equation}
\end{lemma}
\begin{proof}
The exchange relations easily follow from the inclusions
\begin{equation}\label{E:exch++}
A(H_{ij}^k) \subset  H^{i^*k}_j, \quad  A(H^{ij}_k) \subset H^j_{i^*k}, \quad  B(H_{ij}^k) \subset  H^{kj^*}_i,
\quad  B(H^{ij}_k) \subset  H^i_{kj^*}.
\end{equation}
For any $x\in H$,
\[
\pi^{ij}_k\Aop^2x=
\begin{tikzpicture}[baseline=(x.base),hvector/.style={draw=blue!50,fill=blue!20,thick}]
\node (x) [hvector] {$\Aop^2 x$};
\node (i) [above=of x.west ]{};
\node (j) [above=of x.east ]{};
\node (k) [below=of x.center ]{};
\draw (x.north west) to node[auto,inner sep=1pt]{\tiny $i$} (i);
\draw (x.north east) to node[auto,swap,inner sep=1pt]{\tiny $j$} (j);
\draw (x.south) to node[auto,inner sep=1pt]{\tiny $k$} (k);
\end{tikzpicture}
=
\begin{tikzpicture}[baseline=(x.base),hvector/.style={draw=blue!50,fill=blue!20,thick},
turn/.style={circle,draw=blue!50,fill=blue!20,thick,inner sep=2pt}]
\node (x) [hvector] {$\Aop x$};
\node (bi*) [turn,below= of x.west,anchor=east]{};
\node (i) [above=2 of bi*.west ]{};
\node (j) [above=of x.center ]{};
\node (k) [below=of x.east ]{};
\draw (x.south east) to node[auto,inner sep=1pt]{\tiny $k$} (k);
\draw (x.south west) to node[auto,inner sep=1pt]{\tiny $i^*$} (bi*.east);
\draw (x.north) to node[auto,swap,inner sep=1pt]{\tiny $j$} (j);
\draw (bi*.west) to node[auto,inner sep=1pt]{\tiny $i$} (i);
\end{tikzpicture}
=
\begin{tikzpicture}[baseline=(x.base),hvector/.style={draw=blue!50,fill=blue!20,thick},
turn/.style={circle,draw=blue!50,fill=blue!20,thick,inner sep=2pt}]
\node (x) [hvector] {$x$};
\node (a) [turn,above left= of x.east,anchor=west]{};
\node (b) [turn,below= 2 of a.west,anchor=east]{};
\node (i) [above=2 of b.west ]{};
\node (j) [above=of x.east ]{};
\node (k) [below=of x.center ]{};
\draw (x.south) to node[auto,inner sep=1pt]{\tiny $k$} (k);
\draw (x.north west) to node[auto,swap,inner sep=1pt]{\tiny $i$} (a.east);
\draw (x.north east) to node[auto,swap,inner sep=1pt]{\tiny $j$} (j);
\draw (b.west) to node[auto,inner sep=1pt]{\tiny $i$} (i);
\draw (a.west) to node[auto,inner sep=1pt]{\tiny $i^*$} (b.east);
\end{tikzpicture}
=
\begin{tikzpicture}[baseline=(x.base),hvector/.style={draw=blue!50,fill=blue!20,thick}]
\node (x) [hvector] {$x$};
\node (i) [above=of x.west ]{};
\node (j) [above=of x.east ]{};
\node (k) [below=of x.center ]{};
\draw (x.north west) to node[auto,inner sep=1pt]{\tiny $i$} (i);
\draw (x.north east) to node[auto,swap,inner sep=1pt]{\tiny $j$} (j);
\draw (x.south) to node[auto,inner sep=1pt]{\tiny $k$} (k);
\end{tikzpicture}
=\pi^{ij}_kx
\]
and
\[
\pi_{ij}^k\Aop^2x=
\begin{tikzpicture}[baseline=(x.base),hvector/.style={draw=blue!50,fill=blue!20,thick}]
\node (x) [hvector] {$\Aop^2x$};
\node (i) [below=of x.west ]{};
\node (j) [below=of x.east ]{};
\node (k) [above=of x.center ]{};
\draw (x.south west) to node[auto,swap,inner sep=1pt]{\tiny $i$} (i);
\draw (x.south east) to node[auto,inner sep=1pt]{\tiny $j$} (j);
\draw (x.north) to node[auto,swap,inner sep=1pt]{\tiny $k$} (k);
\end{tikzpicture}
=
\begin{tikzpicture}[baseline=(x.base),hvector/.style={draw=blue!50,fill=blue!20,thick},
turn/.style={circle,draw=blue!50,fill=blue!20,thick,inner sep=2pt}]
\node (x) [hvector] {$\Aop x$};
\node (di*) [turn,above= of x.west,anchor=east]{};
\node (i) [below=2 of di*.west ]{};
\node (j) [below=of x.center ]{};
\node (k) [above=of x.east ]{};
\draw (x.north east) to node[auto,swap,inner sep=1pt]{\tiny $k$} (k);
\draw (x.north west) to node[auto,swap,inner sep=1pt]{\tiny $i^*$} (di*.east);
\draw (x.south) to node[auto,inner sep=1pt]{\tiny $j$} (j);
\draw (di*.west) to node[auto,swap,inner sep=1pt]{\tiny $i$} (i);
\end{tikzpicture}
=
\begin{tikzpicture}[baseline=(x.base),hvector/.style={draw=blue!50,fill=blue!20,thick},
turn/.style={circle,draw=blue!50,fill=blue!20,thick,inner sep=2pt}]
\node (x) [hvector] {$x$};
\node (b) [turn,below left= of x.east,anchor=west]{};
\node (a) [turn,above=2 of b.west,anchor=east]{};
\node (i) [below=2 of a.west ]{};
\node (j) [below=of x.east ]{};
\node (k) [above=of x.center ]{};
\draw (x.north) to node[auto,swap,inner sep=1pt]{\tiny $k$} (k);
\draw (x.south west) to node[auto,inner sep=1pt]{\tiny $i$} (b.east);
\draw (x.south east) to node[auto,inner sep=1pt]{\tiny $j$} (j);
\draw (a.west) to node[auto,swap,inner sep=1pt]{\tiny $i$} (i);
\draw (b.west) to node[auto,swap,inner sep=1pt]{\tiny $i^*$} (a.east);
\end{tikzpicture}
=
\begin{tikzpicture}[baseline=(x.base),hvector/.style={draw=blue!50,fill=blue!20,thick}]
\node (x) [hvector] {$x$};
\node (i) [below=of x.west ]{};
\node (j) [below=of x.east ]{};
\node (k) [above=of x.center ]{};
\draw (x.south west) to node[auto,swap,inner sep=1pt]{\tiny $i$} (i);
\draw (x.south east) to node[auto,inner sep=1pt]{\tiny $j$} (j);
\draw (x.north) to node[auto,swap,inner sep=1pt]{\tiny $k$} (k);
\end{tikzpicture}
=\pi_{ij}^kx.
\]
Thus, $\Aop^2=1$. A similar   calculation shows that    $\Bop^2=1$.
\end{proof}
\begin{proof}[Proof of Lemma~\ref{L:1}]
Assume first that $H^{ij}_k\ne0$. By the basic condition,
\begin{equation}\label{E:xdecomp--}
\Id_{V_i\boxtimes V_j}=\sum_{l\in {X}}\sum_{\alpha\in R_l} e_{l\alpha}  e^{l\alpha},
\end{equation}
where ${X}$ is  a finite  subset of $  I$ and for all   $l\in {X}$,
we have   a finite set of indices $R_l$, linearly independent
vectors $\{e_{l\alpha}\}_{\alpha\in R_l}$    in $H^{ij}_l$ and
certain vectors $e^{l\alpha}$ in $H_{ij}^l$. For any $  x\in
H^{ij}_k$,
\begin{equation}\label{E:xdecomp}
x=\Id_{V_i\boxtimes V_j}x=\sum_{l\in {X}} \sum_{\alpha\in R_l} e_{l\alpha}  e^{l\alpha} x=
\sum_{l\in {X}}\sum_{\alpha\in R_l} e_{l\alpha}   \langle e^{l\alpha} x\rangle\delta_{k,l}
= \sum_{\alpha\in R_k} e_{k\alpha}   \langle e^{k\alpha} x\rangle,
\end{equation}
where $\delta_{k,l}$ is the Kronecker delta. Thus,  the vectors
$e_{k\alpha}$ with $ \alpha\in R_k$ generate~$H^{ij}_k$. Since these
vectors are linearly independent, they form
  a (finite) basis of $H^{ij}_k$. Similarly, for any $
y\in H_{ij}^k$,
\begin{equation}\label{E:ydecomp}
y=y\Id_{V_i\boxtimes V_j} =
 \sum_{\alpha\in R_k} \langle ye_{k\alpha} \rangle e^{k\alpha} .
\end{equation}
Therefore the vectors $e^{k\alpha} $ with $  \alpha\in R_k$,
generate  $H_{ij}^k$. For all $\alpha, \beta\in R_k$,
Formula~\eqref{E:xdecomp} with $x=e_{k\beta}$  implies that \(
\langle e^{k\alpha} e_{k\beta} \rangle=\delta_{\alpha,\beta} \).
Hence $\{e^{k\alpha} \}_{ \alpha\in R_k}$ is a basis of $H^k_{ij}$
dual to the basis $\{e_{k\alpha} \}_{  \alpha\in R_k}$ of $H_k^{ij}$
with respect to the   pairing~\eqref{E:pair}. Therefore, this
pairing  is non-degenerate.

It remains to show that $H^{ij}_k=0$ implies $H_{ij}^k=0$. Indeed,
if $H^k_{ij}\ne 0$, then  $H^{i^*k}_j=\Aop (H^k_{ij}) \neq 0$.  By
the preceding argument, $
  \Aop(H_{k}^{ij})= H_{i*k}^j \neq 0$. Hence
$H_{k}^{ij}\ne 0$.
\end{proof}
\begin{proof}[Proof of Lemma~\ref{L:0}]
If $H^{ij}_k\neq 0$, then by Formula~\eqref{E:xdecomp},  $k$ belongs
to the finite set $X$ appearing in \eqref{E:xdecomp--}.
\end{proof}
\subsection{Transposition of operators}\label{2.4sec} We provide the vector space $H=\hat H\oplus\check H$  with the  symmetric bilinear
pairing $\langle\, ,\rangle$  by
\begin{equation}\label{E:pairing}
\langle x,y\rangle=\sum_{i,j,k\in I} (\langle\pi_{ij}^k x\,
\pi_k^{ij}y\rangle+\langle\pi_{ij}^ky\, \pi_k^{ij}x\rangle) \in
\FK\end{equation} for any $x,y\in H$. Note that $\langle  \hat H ,
\hat H \rangle=\langle  \check H, \check H \rangle=0$.

A \emph{transpose} of $f\in\End(H)$ is a map $f^*\in\End(H)$ such
that
 $
\langle fx,y\rangle=\langle x,f^*y\rangle $ for all $x,y\in H$.
Lemma~\ref{L:1} implies that if a transpose $f^*$ of $f$ exists,
then it is unique and  $(f^*)^*=f$.


\begin{lemma}
The   canonical projections have transposes computed as follows: \[
\check \pi^*=\hat \pi \quad {\rm {and}}\quad
(\pi^{ij}_k)^*=\pi^k_{ij}.\]
\end{lemma}
\begin{proof}
\[
\langle x,\hat\pi y\rangle=\langle \check\pi x,\hat\pi y\rangle=\langle \check\pi x,y\rangle,
\]
\[
\langle x,\pi_{ij}^ky\rangle=
\langle \pi^{ij}_kx,\pi_{ij}^ky\rangle=\langle \pi^{ij}_kx,y\rangle.
\]
\end{proof}
\begin{lemma}The transposes of the operators $\Aop$ and $\Bop$ exist
and
\begin{equation}\label{E:aba*=bab}
\Aop^*\Bop^*\Aop^*=\Bop\Aop\Bop.
\end{equation}
\end{lemma}
\begin{proof} The existence of $\Aop^*$ and $\Bop^*$   follows from Lemma~\ref{L:1} and the
 inclusions \eqref{E:exch++}. Note that
\begin{equation}\label{E:exch++bis}
A^*(H_{ij}^k) \subset  H^{i^*k}_j, \quad  A^*(H^{ij}_k) \subset H^j_{i^*k}, \quad  B^*(H_{ij}^k) \subset  H^{kj^*}_i,
\quad  B^*(H^{ij}_k) \subset  H^i_{kj^*}.
\end{equation}

To prove \eqref{E:aba*=bab}, observe that for any    $x\in \hat H $
and  $y\in \check H$,
  \begin{equation}\label{E:aba*=bab++}
\langle x, y \rangle =\langle \Bop \Aop \Bop y, \Aop \Bop \Aop x \rangle.
\end{equation}
  Here
 is a   graphical
 proof of this formula for   $x\in H^i_{jk}$ and  $y\in H^{jk}_i$ with $i,j,k\in I$.
\[
\begin{tikzpicture}[baseline=(bl.south),hvector/.style={draw=blue!50,fill=blue!20,thick},
turn/.style={circle,draw=blue!50,fill=blue!20,thick,inner sep=2pt}]
\node (x) [hvector] {$\Aop\Bop\Aop x$};
\node (y) [hvector,above=1.4 of x.center,anchor=center] {$\Bop\Aop\Bop y$};
\node (bl) [above=.7 of x.center,anchor=center] {};
\node (ku)[above= of y.center]{};
\node (kd)[below= of x.center]{};
\draw (y.north) to  node[auto,inner sep=1pt]{\tiny $i^*$}(ku);
\draw (x.north west) to node[auto,swap,inner sep=1pt]{\tiny $k^*$} (y.south west);
\draw (x.north east) to node[auto,inner sep=1pt]{\tiny $j^*$} (y.south east);
\draw (x.south) to node[auto,inner sep=1pt]{\tiny $i^*$} (kd);
\end{tikzpicture}
=
\begin{tikzpicture}[baseline=(x.base),hvector/.style={draw=blue!50,fill=blue!20,thick},
turn/.style={circle,draw=blue!50,fill=blue!20,thick,inner sep=2pt}]
\node (x) [hvector] {$x$};
\node (y) [hvector,left=2 of x] {$y$};
\node (b) [turn,below=2 of x.west,anchor=east]{};
\node (a) [turn,above= of b.west,anchor=east]{};
\node (e) [turn,above=2 of y.east,anchor=west]{};
\node (d) [turn,below= of e.east,anchor=west]{};
\node (c) [turn,above= of x.east,anchor=west]{};
\node (f) [turn,below = of y.west,anchor=east]{};
\node (ku)[above=3 of f.west]{};
\node (kd)[below=3 of c.east]{};
\draw (x.north) to  node[auto,inner sep=1pt]{\tiny $i$}(c.west);
\draw (x.south west) to node[auto,swap,inner sep=1pt]{\tiny $j$} (a.east);
\draw (x.south east) to node[auto,inner sep=1pt]{\tiny $k$} (b.east);
\draw (c.east) to node[auto,inner sep=1pt]{\tiny $i^*$} (kd);
\draw (a.west) to node[auto,swap,inner sep=1pt]{\tiny $j^*$} (e.east);
\draw (b.west) to node[auto,inner sep=1pt]{\tiny $k^*$} (d.east);
\draw (y.north west) to node[auto,inner sep=1pt]{\tiny $j$} (e.west);
\draw (y.north east) to node[auto,swap,inner sep=1pt]{\tiny $k$} (d.west);
\draw (y.south) to node[auto,inner sep=1pt]{\tiny $i$} (f.east);
\draw (f.west) to node[auto,inner sep=1pt]{\tiny $i^*$} (ku);
\end{tikzpicture}
=
\begin{tikzpicture}[baseline=(bl.south),hvector/.style={draw=blue!50,fill=blue!20,thick},
turn/.style={circle,draw=blue!50,fill=blue!20,thick,inner sep=2pt}]
\node (x) [hvector] {$x$};
\node (bl) [below=.5 of x.center,anchor=center]{};
\node (y) [hvector,below= of x.center,anchor=center] {$y$};
\node (c) [turn,above= of x.east,anchor=west]{};
\node (f) [turn,below = of y.west,anchor=east]{};
\node (ku)[above=3 of f.west]{};
\node (kd)[below=3 of c.east]{};
\draw (x.north) to  node[auto,inner sep=1pt]{\tiny $i$}(c.west);
\draw (x.south west) to node[auto,inner sep=1pt]{\tiny $j$} (y.north west);
\draw (x.south east) to node[auto,swap,inner sep=1pt]{\tiny $k$} (y.north east);
\draw (c.east) to node[auto,inner sep=1pt]{\tiny $i^*$} (kd);
\draw (y.south) to node[auto,inner sep=1pt]{\tiny $i$} (f.east);
\draw (f.west) to  node[auto,inner sep=1pt]{\tiny $i^*$} (ku);
\end{tikzpicture}.
\]
Now we can prove~\eqref{E:aba*=bab}. Applying \eqref{E:aba*=bab++}
to $x=x_1 $ and $y=BABx_2$ with  $x_1, x_2 \in \hat H $, we obtain
\[\langle x_1, BA\Bop x_2\rangle=\langle (BAB)^2 x_2, AB\Aop x_1\rangle=\langle x_2, AB\Aop
x_1\rangle =\langle AB\Aop x_1,   x_2\rangle. \] Applying
\eqref{E:aba*=bab++} to $x=ABA y_1$ and $y=y_2$ with  $y_1, y_2 \in
\check H $, we obtain
\[\langle y_1, BA\Bop y_2\rangle=\langle  BAB y_2, y_1\rangle=\langle
AB\Aop y_1, y_2
 \rangle.\]
 Hence $BAB=(ABA)^\ast=A^* B^* A^*$.
\end{proof}

\section{The tetrahedral  forms}\label{Section3now}

\subsection{Operations on tensor powers} We   recall  the usual notation for operations on the tensor powers
 of a vector space.   Given a $\FK$-vector space $V$
and an integer $n\geq 2$, the symbol  $V^{\otimes n}$ denotes
 the tensor product of $n$ copies of $V$ over~$\FK$. Let~$ \mathbb{S}_n$ be the symmetric group on $n\geq 2$ letters.
 Recall the standard action $ \mathbb{S}_n\to \Aut(V^{\otimes
n})$, $\sigma\mapsto P_\sigma$.
   By definition, for distinct $i,j\in \{1,\ldots,
   n\}$, the   flip  $P_{(ij)}$  permutes the $i$-th
   and the $j$-th tensor factors keeping the other tensor  factors. For   $f\in\End (V)$ and $i=1, \ldots, n$, set
   $$f_i={\rm {id}}^{\otimes (i-1)}\otimes f\otimes {\rm {id}}^{\otimes (n-i)}\in \End (V^{\otimes
   n}).$$
   Note the exchange relations   $P_\sigma f_i=f_{\sigma(i)}P_\sigma$ for any $\sigma \in \mathbb{S}_n$ and the commutativity relation
   $f_ig_j=g_jf_i$
   for any $f, g\in \End (V)$ and $i\neq j$.

   Given   $F\in \End(V \otimes_\FK V) $, we define for any   $i,j\in \{1,\ldots,
   n\}$ with $i\neq j$ an  endomorphism $F_{ij}$ of $V^{\otimes n}$
   as follows. If $i<j$, then $F_{ij}$ acts as $F$ on the $i$-th and
   $j$-th tensor factors of $V^{\otimes n}$ keeping the other tensor
   factors. If $i>j$, then $F_{ij}=P_{(ij)} F_{ji} P_{(ij)}$.

\subsection{The forms $T$ and $\bar T$}\label{TheformsTand} Recall the vector space   $H=\hat H\oplus\check H$ from Section~\ref{section
2.3}.
   We define two linear forms \( T, \bar T\colon
H^{\otimes4}\to \FK \)
 by the following diagrammatic formulae:  for any $ u,v,x,y\in H$,
\[
T(u\otimes v\otimes x\otimes y)=\sum_{i,\ldots,n\in I}
\left\langle
\begin{tikzpicture}[baseline=(bl.south),hvector/.style={draw=blue!50,fill=blue!20,thick},
turn/.style={circle,draw=blue!50,fill=blue!20,thick,inner sep=2pt}]
\node (x) [hvector] {$u$};
\node (y) [hvector,below= of x.west,anchor=east]{$v$};
\node (u) [hvector,below= of y.east,anchor=west] {$x$};
\node (v) [hvector,below= of u.west,anchor=east]{$y$};
\node (up) [above= of x.center,anchor=center]{};
\node (down)[below= of v.center,anchor=center]{};
\node (bl)[below=.5 of y.center,anchor=center]{};
\draw (x.north) to  node[auto,inner sep=1pt]{\tiny $m$}(up.center);
\draw (x.south west) to node[auto,swap,inner sep=1pt]{\tiny $k$} (y.north);
\draw (x.south east) to node[auto,swap,inner sep=1pt]{\tiny $l$} (u.north east);
\draw (y.south east) to node[auto,inner sep=1pt]{\tiny $j$} (u.north west);
\draw (y.south west) to node[auto,inner sep=1pt]{\tiny $i$} (v.north west);
\draw (u.south) to  node[auto,inner sep=1pt]{\tiny $n$} (v.north east);
\draw (v.south) to node[auto,inner sep=1pt]{\tiny $m$} (down.center);
\end{tikzpicture}
\right\rangle
,\qquad
\bar T(u\otimes v\otimes x\otimes y)=\sum_{i,\ldots,n\in I}
\left\langle
\begin{tikzpicture}[baseline=(bl.south),hvector/.style={draw=blue!50,fill=blue!20,thick},
turn/.style={circle,draw=blue!50,fill=blue!20,thick,inner sep=2pt}]
\node (x) [hvector] {$u$};
\node (y) [hvector,below= of x.east,anchor=west]{$v$};
\node (u) [hvector,below= of y.west,anchor=east] {$x$};
\node (v) [hvector,below= of u.east,anchor=west]{$y$};
\node (up) [above= of x.center,anchor=center]{};
\node (down)[below= of v.center,anchor=center]{};
\node (bl)[below=.5 of y.center,anchor=center]{};
\draw (x.north) to  node[auto,swap,inner sep=1pt]{\tiny $m$}(up.center);
\draw (x.south west) to node[auto,inner sep=1pt]{\tiny $i$} (u.north west);
\draw (x.south east) to node[auto,inner sep=1pt]{\tiny $n$} (y.north);
\draw (y.south east) to node[auto,swap,inner sep=1pt]{\tiny $l$} (v.north east);
\draw (y.south west) to node[auto,swap,inner sep=1pt]{\tiny $j$} (u.north east);
\draw (u.south) to  node[auto,swap,inner sep=1pt]{\tiny $k$} (v.north west);
\draw (v.south) to node[auto,swap,inner sep=1pt]{\tiny $m$} (down.center);
\end{tikzpicture}
\right\rangle.
\]
 The indices $i,j,k,l,m,n$ in both formulas
 run over all elements of $I$. For any given $ u,v,x,y \in H$, only a
 finite number of terms in these formulas may be non-zero.
\begin{lemma}[Fundamental lemma]\label{Fundamental lemma}
We have
\begin{subequations}\label{E:sym}
\begin{align}
TP_{(4321)}&=\bar T\Aop_1^*\Aop_3,\label{E:sym01}\\
TP_{(23)}&=\bar T\Aop_2\Bop_3,\label{E:sym12}\\
TP_{(1234)}&=\bar T\Bop_2\Bop_4^*.\label{E:sym23}
 \end{align}
\end{subequations}
\end{lemma}
\begin{proof} Since $A^*$ is an involution, \eqref{E:sym01} is equivalent to the identity
\[
T(u\otimes v\otimes x\otimes \Aop^*y)=\bar T(y\otimes u\otimes \Aop v\otimes x),\quad u,v,x,y \in H,
\]
which is a direct consequence of the identity
\[
\left\langle
\begin{tikzpicture}[baseline=(bl.south),hvector/.style={draw=blue!50,fill=blue!20,thick},
turn/.style={circle,draw=blue!50,fill=blue!20,thick,inner sep=2pt}]
\node (x) [hvector] {$u$};
\node (y) [hvector,below= of x.west,anchor=east]{$v$};
\node (u) [hvector,below= of y.east,anchor=west] {$x$};
\node (v) [hvector,below= of u.center,anchor=east]{$\Aop^*y$};
\node (up) [above= of x.center,anchor=center]{};
\node (down)[below= of v.center,anchor=center]{};
\node (bl)[below=.5 of y.center,anchor=center]{};
\draw (x.north) to  node[auto,inner sep=1pt]{\tiny $m$}(up.center);
\draw (x.south west) to node[auto,swap,inner sep=1pt]{\tiny $k$} (y.north);
\draw (x.south east) to node[auto,swap,inner sep=1pt]{\tiny $l$} (u.north east);
\draw (y.south east) to node[auto,inner sep=1pt]{\tiny $j$} (u.north west);
\draw (y.south west) to node[auto,inner sep=1pt]{\tiny $i$} (v.north west);
\draw (u.south) to  node[auto,swap,inner sep=1pt]{\tiny $n$} (v.north east);
\draw (v.south) to node[auto,inner sep=1pt]{\tiny $m$} (down.center);
\end{tikzpicture}
\right\rangle
=\langle\xi,\Aop^*y\rangle=\langle\Aop\xi,y\rangle=\langle y, \Aop\xi \rangle=
\left\langle
\begin{tikzpicture}[baseline=(bl.south),hvector/.style={draw=blue!50,fill=blue!20,thick},
turn/.style={circle,draw=blue!50,fill=blue!20,thick,inner sep=2pt}]
\node (x) [hvector] {$y$};
\node (y) [hvector,below= of x.east,anchor=west]{$u$};
\node (u) [hvector,below=2 of x.east,anchor=east] {$\Aop v$};
\node (v) [hvector,below=2 of y.east,anchor=east]{$x$};
\node (up) [above= of x.center,anchor=center]{};
\node (down)[below= of v.center,anchor=center]{};
\node (bl)[below=.5 of y.center,anchor=center]{};
\draw (x.north) to  node[auto,swap,inner sep=1pt]{\tiny $n$}(up.center);
\draw (x.south west) to node[auto,inner sep=1pt]{\tiny $i^*$} (u.north west);
\draw (x.south east) to node[auto,inner sep=1pt]{\tiny $m$} (y.north);
\draw (y.south east) to node[auto,swap,inner sep=1pt]{\tiny $l$} (v.north east);
\draw (y.south west) to node[auto,swap,inner sep=1pt]{\tiny $k$} (u.north east);
\draw (u.south) to  node[auto,swap,inner sep=1pt]{\tiny $j$} (v.north west);
\draw (v.south) to node[auto,swap,inner sep=1pt]{\tiny $n$} (down.center);
\end{tikzpicture}
\right\rangle,\quad
\xi=\begin{tikzpicture}[baseline=(y.base),hvector/.style={draw=blue!50,fill=blue!20,thick},
turn/.style={circle,draw=blue!50,fill=blue!20,thick,inner sep=2pt}]
\node (x) [hvector] {$u$};
\node (y) [hvector,below= of x.west,anchor=east]{$v$};
\node (u) [hvector,below= of y.east,anchor=west] {$x$};
\node (up) [above= of x.center,anchor=center]{};
\node (dl) [below=2 of y.west,anchor=center]{};
\node (dr) [below= of u.center,anchor=center]{};
\draw (x.north) to  node[auto,inner sep=1pt]{\tiny $m$}(up.center);
\draw (x.south west) to node[auto,swap,inner sep=1pt]{\tiny $k$} (y.north);
\draw (x.south east) to node[auto,swap,inner sep=1pt]{\tiny $l$} (u.north east);
\draw (y.south east) to node[auto,inner sep=1pt]{\tiny $j$} (u.north west);
\draw (y.south west) to node[auto,inner sep=1pt]{\tiny $i$} (dl.center);
\draw (u.south) to  node[auto,inner sep=1pt]{\tiny $n$} (dr.center);
\end{tikzpicture}\in H^m_{in} .
\]
The other two identities are verified in a similar manner.
\end{proof}

 The formulas
\[
P_{(12)}=P_{(4321)}P_{(23)}P_{(1234)},\quad P_{(34)}=P_{(1234)}P_{(23)}P_{(4321)},
\]
allow us to compute the action  of the permutations $P_{(12)}$ and
$P_{(34)}$ on $T$:
 \begin{equation}\label{actionof12}
TP_{(12)}=\bar TP_{(23)}P_{(1234)}\Aop_4^*\Aop_1=TP_{(1234)}(\Bop\Aop)_1\Aop_2\Aop_4^*=
\bar T(\Bop\Aop)_1(\Bop\Aop)_2(\Aop\Bop)_4^*,
\end{equation}
\begin{equation}\label{actionof34}
TP_{(34)}=\bar TP_{(23)} P_{(4321)} \Bop_4\Bop_1^*=TP_{(4321)}\Bop_3(\Aop\Bop)_4\Bop_1^*=
\bar T(\Bop\Aop)_1^*(\Aop\Bop)_3(\Aop\Bop)_4.
\end{equation}
The action of the  permutations on  $\bar T$ can be easily
determined from the involutivity of $A$, $B$, $P_{(12)}$,
$P_{(23)}$, $P_{(34)}$. The resulting formulae can be obtained from
those for~$T$ via the substitutions $T\leftrightarrow \bar T$ and
$\Aop\leftrightarrow\Bop$.

The formulas   computing  the action of the  permutations   on $T$ and $\bar T$  may be rewritten in a simpler form in terms of the equivalent tensors $S=TP_{(2134)}\colon
H^{\otimes4}\to \FK$ and $\bar S=\bar TP_{(1243)}\colon
H^{\otimes4}\to \FK$.
For these tensors, equations~\eqref{E:sym01}--\eqref{E:sym23} take the following form:
\begin{subequations}\label{E:symS}
\begin{align}
SP_{(12)}&=\bar S\Aop_3^*\Aop_4,\label{E:sym01S}\\
SP_{(23)}&=\bar S\Aop_1\Bop_4,\label{E:sym12S}\\
SP_{(34)}&=\bar S\Bop_1\Bop_2^*.\label{E:sym23S}
 \end{align}
 \end{subequations}
 Though these symmetry relations for   $S$, $\bar S$ are simpler than the symmetry relations for   $T$, $\bar T$, we shall mainly work with $T$ and $\bar T$. A geometric interpretation of  these  symmetry relations   will be outlined in the appendix to the paper.

\subsection{The adjoint operators}\label{section33now}  We define   the operators
  $\check H^{\otimes 2} \to \check H^{\otimes 2}$ adjoint to~$T$ and $\bar T$. For all $i,j,k\in
I$ pick  dual bases $(e^{ij}_{k\alpha})_\alpha$ and
$(e_{ij}^{k\alpha})_\alpha$
 in the multiplicity spaces $H^{ij}_k$ and $H_{ij}^k$,
respectively. For the vectors of these bases, we shall use the
graphical notation
$$
e^{ij}_{k\alpha}=
\begin{tikzpicture}[baseline=(b.base),hvector/.style={draw=blue!50,fill=blue!20,thick},
turn/.style={circle,draw=blue!50,fill=blue!20,thick,inner
sep=2pt},basis/.style={circle,draw=red!50,fill=red!20,thick,inner
sep=2pt}]
\node (b) [basis]{$\alpha$}; \node (ul) [above= of
b.west,anchor=center]{}; \node (ur) [above= of
b.east,anchor=center]{}; \node (d) [below= of
b.center,anchor=center]{}; \draw (b.west) to node[auto,swap,inner
sep=1pt]{\tiny $i$} (ul.center); \draw (b.east) to node[auto,inner
sep=1pt]{\tiny $j$} (ur.center); \draw (b.south) to node[auto,inner
sep=1pt]{\tiny $k$} (d.center);
\end{tikzpicture} \quad \quad {\text {and}} \quad \quad   e_{ij}^{k\alpha}=
\begin{tikzpicture}[baseline=(b.base),hvector/.style={draw=blue!50,fill=blue!20,thick},
turn/.style={circle,draw=blue!50,fill=blue!20,thick,inner
sep=2pt},basis/.style={circle,draw=red!50,fill=red!20,thick,inner
sep=2pt}]
\node (b) [basis]{$\alpha$}; \node (dl) [below= of
b.west,anchor=center]{}; \node (dr) [below= of
b.east,anchor=center]{}; \node (u) [above= of
b.center,anchor=center]{}; \draw (b.west) to node[auto,inner
sep=1pt]{\tiny $i$} (dl.center); \draw (b.east) to
node[auto,swap,inner sep=1pt]{\tiny $j$} (dr.center); \draw
(b.north) to node[auto,swap,inner sep=1pt]{\tiny $k$} (u.center);
\end{tikzpicture}.$$
Let $\tau,\bar \tau\in\End(\check H^{\otimes2})$ be the operators
defined by the  graphical formulae
\[
\tau(x\otimes y)=\sum_{i,\ldots,n\in I}\sum_{\alpha}
\begin{tikzpicture}[baseline=(x.base),hvector/.style={draw=blue!50,fill=blue!20,thick},
turn/.style={circle,draw=blue!50,fill=blue!20,thick,inner sep=2pt},basis/.style={circle,draw=red!50,fill=red!20,thick,inner sep=2pt}]
\node (b) [basis]{$\alpha$};
\node (x) [hvector,below= of b.east,anchor=west] {$x$};
\node (y) [hvector,below= of x.west,anchor=east]{$y$};
\node (ul) [above= of b.center,anchor=center]{};
\node (ur) [above=2 of x.east,anchor=center]{};
\node (d) [below= of y.center,anchor=center]{};
\draw (b.north) to node[auto,swap,inner sep=1pt]{\tiny $k$} (ul.center);
\draw (x.north east) to node[auto,inner sep=1pt]{\tiny $l$} (ur.center);
\draw (b.east) to node[auto,inner sep=1pt]{\tiny $j$} (x.north west);
\draw (b.west) to node[auto,inner sep=1pt]{\tiny $i$} (y.north west);
\draw (x.south) to  node[auto,inner sep=1pt]{\tiny $n$} (y.north east);
\draw (y.south) to node[auto,inner sep=1pt]{\tiny $m$} (d.center);
\end{tikzpicture}
\otimes
\begin{tikzpicture}[baseline=(b.base),hvector/.style={draw=blue!50,fill=blue!20,thick},
turn/.style={circle,draw=blue!50,fill=blue!20,thick,inner sep=2pt},basis/.style={circle,draw=red!50,fill=red!20,thick,inner sep=2pt}]
\node (b) [basis]{$\alpha$};
\node (ul) [above= of b.west,anchor=center]{};
\node (ur) [above= of b.east,anchor=center]{};
\node (d) [below= of b.center,anchor=center]{};
\draw (b.west) to node[auto,swap,inner sep=1pt]{\tiny $i$} (ul.center);
\draw (b.east) to node[auto,inner sep=1pt]{\tiny $j$} (ur.center);
\draw (b.south) to node[auto,inner sep=1pt]{\tiny $k$} (d.center);
\end{tikzpicture},
\qquad
\bar \tau(x\otimes y)=\sum_{i,\ldots,n\in I}\sum_{\alpha}
\begin{tikzpicture}[baseline=(x.base),hvector/.style={draw=blue!50,fill=blue!20,thick},
turn/.style={circle,draw=blue!50,fill=blue!20,thick,inner sep=2pt},basis/.style={circle,draw=red!50,fill=red!20,thick,inner sep=2pt}]
\node (b) [basis]{$\alpha$};
\node (x) [hvector,below= of b.west,anchor=east] {$x$};
\node (y) [hvector,below= of x.east,anchor=west]{$y$};
\node (ur) [above= of b.center,anchor=center]{};
\node (ul) [above=2 of x.west,anchor=center]{};
\node (d) [below= of y.center,anchor=center]{};
\draw (b.north) to node[auto,inner sep=1pt]{\tiny $n$} (ur.center);
\draw (x.north west) to node[auto,swap,inner sep=1pt]{\tiny $i$} (ul.center);
\draw (b.east) to node[auto,swap,inner sep=1pt]{\tiny $l$} (y.north east);
\draw (b.west) to node[auto,swap,inner sep=1pt]{\tiny $j$} (x.north east);
\draw (x.south) to  node[auto,swap,inner sep=1pt]{\tiny $k$} (y.north west);
\draw (y.south) to node[auto,swap,inner sep=1pt]{\tiny $m$} (d.center);
\end{tikzpicture}
\otimes
\begin{tikzpicture}[baseline=(b.base),hvector/.style={draw=blue!50,fill=blue!20,thick},
turn/.style={circle,draw=blue!50,fill=blue!20,thick,inner sep=2pt},basis/.style={circle,draw=red!50,fill=red!20,thick,inner sep=2pt}]
\node (b) [basis]{$\alpha$};
\node (ul) [above= of b.west,anchor=center]{};
\node (ur) [above= of b.east,anchor=center]{};
\node (d) [below= of b.center,anchor=center]{};
\draw (b.west) to node[auto,swap,inner sep=1pt]{\tiny $j$} (ul.center);
\draw (b.east) to node[auto,inner sep=1pt]{\tiny $l$} (ur.center);
\draw (b.south) to node[auto,inner sep=1pt]{\tiny $n$} (d.center);
\end{tikzpicture}
\]
where $i,j,k,l,m,n$ run over all elements of $I$. By
Lemma~\ref{L:0}, for any  $x,y\in \check H$, there are only finitely
many terms in the expansions for $\tau(x\otimes y)$ and $\bar
\tau(x\otimes y)$.

The operators $\tau$ and $\bar \tau$ do not depend on the choice of
the bases in the multiplicity spaces. Indeed, these operators are
adjoint to $T, \bar T$ in the sense that
 \[
\langle \langle  u\otimes v,\tau(x\otimes y)\rangle  \rangle=T(u\otimes v\otimes x\otimes y)
 \]
 and
 \[
\langle \langle  u\otimes v, \bar \tau(x\otimes y)\rangle  \rangle= \bar T(u\otimes v\otimes x\otimes y)
 \]
 for all $u,v\in\hat H,  x,y\in \check H$.
  Here the   bilinear pairing
 $$\langle \langle  \cdot , \cdot \rangle  \rangle:(\hat H \otimes \hat H) \times (\check H\otimes \check H)\to \FK$$ is defined by
 $\langle \langle  u\otimes v , x\otimes y \rangle  \rangle =\langle u, x\rangle\, \langle v, y \rangle$
 where $\langle\cdot, \cdot \rangle$ is the   symmetric bilinear form on $H=\hat H\oplus \check H$ introduced in Section~\ref{2.4sec}.

\subsection{The Pentagon and Inversion identities}  To formulate the properties of $\tau$ and $\bar \tau$, we   need further notation.
For any  $ i,j \in I $ set $$  g_{i,j}=\left\{\begin{array}{cl}
 1&\mathrm{if}\ \mathrm{there}\ \mathrm{is} \ k\in I \  \mathrm{such} \ \mathrm{that} \ H^{ij}_k\neq 0\\
 0&\mathrm{otherwise}.
 \end{array}
 \right.
 $$
 We define two endomorphisms ${}^\bullet \pi$ and $\pi^\bullet$ of $\check H^{\otimes 2}$ by
  \[
 {}^\bullet \pi=\sum_{i,j,k,l,m\in I} \, g_{i,j}\, \pi^{il}_m\otimes \pi^{jk}_l \quad {\text {and}} \quad
 \pi^\bullet=\sum_{i,j,k,l,m\in I} \, g_{j,l}\, \pi^{ij}_k\otimes \pi^{kl}_m   .
 \]
Clearly,   ${}^\bullet \pi$ and $\pi^\bullet$ are commuting projectors onto certain subspaces of $\check H^{\otimes 2}$.

 \begin{lemma}
 The  operators $\tau$ and $\bar \tau$ satisfy
 \begin{itemize}
 \item[(i)]
 the pentagon identity in $\End (\check H^{\otimes 3})$:
 \[{\tau}_{23}{\tau}_{13}{\tau}_{12}=
  {\tau}_{12}{\tau}_{23} ({}^\bullet \pi)_{21},
 \]
 \item[(ii)] the inversion relations in $\End (\check H^{\otimes
 2})$:
 \[
 {\tau}_{21}\bar {\tau} =\pi^\bullet \quad  {\text {and}} \quad \bar {\tau} {\tau}_{21}={}^\bullet \pi,
 \]
 where ${\tau}_{21}=P_{(12)} \tau P_{(12)}$ and ${\bar
 \tau}_{21}=P_{(12)} \bar \tau P_{(12)}$.
 \end{itemize}

 \end{lemma}
 \begin{proof} ({\bf i}) For $x,y,z\in \check H$,
 \begin{multline*}
 {\tau}_{23}{\tau}_{13}{\tau}_{12}(x\otimes y\otimes z)\\
 =
 \sum_{i,\ldots,n,\iota}{\tau}_{23}{\tau}_{13}(
\begin{tikzpicture}[baseline=(x.base),hvector/.style={draw=blue!50,fill=blue!20,thick},
turn/.style={circle,draw=blue!50,fill=blue!20,thick,inner
sep=2pt},basis/.style={circle,draw=red!50,fill=red!20,thick,inner
sep=2pt}]
\node (b) [basis]{$\iota$}; \node (x) [hvector,below= of
b.east,anchor=west] {$x$}; \node (y) [hvector,below= of
x.west,anchor=east]{$y$}; \node (ul) [above= of
b.center,anchor=center]{}; \node (ur) [above=2 of
x.east,anchor=center]{}; \node (d) [below= of
y.center,anchor=center]{}; \draw (b.north) to node[auto,swap,inner
sep=1pt]{\tiny $k$} (ul.center); \draw (x.north east) to
node[auto,inner sep=1pt]{\tiny $l$} (ur.center); \draw (b.east) to
node[auto,inner sep=1pt]{\tiny $j$} (x.north west); \draw (b.west)
to node[auto,inner sep=1pt]{\tiny $i$} (y.north west); \draw
(x.south) to  node[auto,inner sep=1pt]{\tiny $n$} (y.north east);
\draw (y.south) to node[auto,inner sep=1pt]{\tiny $m$} (d.center);
\end{tikzpicture}
\otimes
\begin{tikzpicture}[baseline=(b.base),hvector/.style={draw=blue!50,fill=blue!20,thick},
turn/.style={circle,draw=blue!50,fill=blue!20,thick,inner
sep=2pt},basis/.style={circle,draw=red!50,fill=red!20,thick,inner
sep=2pt}]
\node (b) [basis]{$\iota$}; \node (ul) [above= of
b.west,anchor=center]{}; \node (ur) [above= of
b.east,anchor=center]{}; \node (d) [below= of
b.center,anchor=center]{}; \draw (b.west) to node[auto,swap,inner
sep=1pt]{\tiny $i$} (ul.center); \draw (b.east) to node[auto,inner
sep=1pt]{\tiny $j$} (ur.center); \draw (b.south) to node[auto,inner
sep=1pt]{\tiny $k$} (d.center);
\end{tikzpicture}
\otimes z) =\sum_{i,\ldots,q,\iota,\kappa}{\tau}_{23}(
\begin{tikzpicture}[baseline=(x.base),hvector/.style={draw=blue!50,fill=blue!20,thick},
turn/.style={circle,draw=blue!50,fill=blue!20,thick,inner
sep=2pt},basis/.style={circle,draw=red!50,fill=red!20,thick,inner
sep=2pt}]
\node (b) [basis]{$\iota$}; \node (x) [hvector,below= of
b.east,anchor=west] {$x$}; \node (y) [hvector,below= of
x.west,anchor=east]{$y$}; \node (z) [hvector,below= of
y.west,anchor=east]{$z$}; \node (a) [basis,above=of
b.west,anchor=east]{$\kappa$}; \node (ul) [above= of
a.center,anchor=center]{}; \node (ur) [above=3 of
x.east,anchor=center]{}; \node (d) [below= of
z.center,anchor=center]{}; \draw (a.north) to node[auto,swap,inner
sep=1pt]{\tiny $p$} (ul.center); \draw (a.west) to node[auto,inner
sep=1pt]{\tiny $o$} (z.north west); \draw (x.north east) to
node[auto,inner sep=1pt]{\tiny $l$} (ur.center); \draw (a.east) to
node[auto,inner sep=1pt]{\tiny $k$} (b.north); \draw (b.west) to
node[auto,inner sep=1pt]{\tiny $i$} (y.north west); \draw (b.east)
to node[auto,inner sep=1pt]{\tiny $j$} (x.north west); \draw
(x.south) to  node[auto,inner sep=1pt]{\tiny $n$} (y.north east);
\draw (y.south) to node[auto,inner sep=1pt]{\tiny $m$} (z.north
east); \draw (z.south) to node[auto,inner sep=1pt]{\tiny $q$}
(d.center);
\end{tikzpicture}
\otimes
\begin{tikzpicture}[baseline=(b.base),hvector/.style={draw=blue!50,fill=blue!20,thick},
turn/.style={circle,draw=blue!50,fill=blue!20,thick,inner
sep=2pt},basis/.style={circle,draw=red!50,fill=red!20,thick,inner
sep=2pt}]
\node (b) [basis]{$\iota$}; \node (ul) [above= of
b.west,anchor=center]{}; \node (ur) [above= of
b.east,anchor=center]{}; \node (d) [below= of
b.center,anchor=center]{}; \draw (b.west) to node[auto,swap,inner
sep=1pt]{\tiny $i$} (ul.center); \draw (b.east) to node[auto,inner
sep=1pt]{\tiny $j$} (ur.center); \draw (b.south) to node[auto,inner
sep=1pt]{\tiny $k$} (d.center);
\end{tikzpicture}
\otimes
\begin{tikzpicture}[baseline=(b.base),hvector/.style={draw=blue!50,fill=blue!20,thick},
turn/.style={circle,draw=blue!50,fill=blue!20,thick,inner
sep=2pt},basis/.style={circle,draw=red!50,fill=red!20,thick,inner
sep=2pt}]
\node (b) [basis]{$\kappa$}; \node (ul) [above= of
b.west,anchor=center]{}; \node (ur) [above= of
b.east,anchor=center]{}; \node (d) [below= of
b.center,anchor=center]{}; \draw (b.west) to node[auto,swap,inner
sep=1pt]{\tiny $o$} (ul.center); \draw (b.east) to node[auto,inner
sep=1pt]{\tiny $k$} (ur.center); \draw (b.south) to node[auto,inner
sep=1pt]{\tiny $p$} (d.center);
\end{tikzpicture}
)
\end{multline*}
\[
=\sum_{i,\ldots,r,\iota,\kappa,\lambda}
\begin{tikzpicture}[baseline=(x.base),hvector/.style={draw=blue!50,fill=blue!20,thick},
turn/.style={circle,draw=blue!50,fill=blue!20,thick,inner sep=2pt},basis/.style={circle,draw=red!50,fill=red!20,thick,inner sep=2pt}]
\node (b) [basis]{$\iota$};
\node (x) [hvector,below= of b.east,anchor=west] {$x$};
\node (y) [hvector,below= of x.west,anchor=east]{$y$};
\node (z) [hvector,below= of y.west,anchor=east]{$z$};
\node (a) [basis,above=of b.west,anchor=east]{$\kappa$};
\node (ul) [above= of a.center,anchor=center]{};
\node (ur) [above=3 of x.east,anchor=center]{};
\node (d) [below= of z.center,anchor=center]{};
\draw (a.north) to node[auto,swap,inner sep=1pt]{\tiny $p$} (ul.center);
\draw (a.west) to node[auto,inner sep=1pt]{\tiny $o$} (z.north west);
\draw (x.north east) to node[auto,inner sep=1pt]{\tiny $l$} (ur.center);
\draw (a.east) to node[auto,inner sep=1pt]{\tiny $k$} (b.north);
\draw (b.west) to node[auto,inner sep=1pt]{\tiny $i$} (y.north west);
\draw (b.east) to node[auto,inner sep=1pt]{\tiny $j$} (x.north west);
\draw (x.south) to  node[auto,inner sep=1pt]{\tiny $n$} (y.north east);
\draw (y.south) to node[auto,inner sep=1pt]{\tiny $m$} (z.north east);
\draw (z.south) to node[auto,inner sep=1pt]{\tiny $q$} (d.center);
\end{tikzpicture}
\otimes
\begin{tikzpicture}[baseline=(x.base),hvector/.style={draw=blue!50,fill=blue!20,thick},
turn/.style={circle,draw=blue!50,fill=blue!20,thick,inner sep=2pt},basis/.style={circle,draw=red!50,fill=red!20,thick,inner sep=2pt}]
\node (b) [basis]{$\lambda$};
\node (x) [basis,below= of b.east,anchor=west] {$\iota$};
\node (y) [basis,below= of x.west,anchor=east]{$\kappa$};
\node (ul) [above= of b.center,anchor=center]{};
\node (ur) [above=2 of x.east,anchor=center]{};
\node (d) [below= of y.center,anchor=center]{};
\draw (b.north) to node[auto,swap,inner sep=1pt]{\tiny $r$} (ul.center);
\draw (x.east) to node[auto,inner sep=1pt]{\tiny $j$} (ur.center);
\draw (b.east) to node[auto,inner sep=1pt]{\tiny $i$} (x.west);
\draw (b.west) to node[auto,inner sep=1pt]{\tiny $o$} (y.west);
\draw (x.south) to  node[auto,inner sep=1pt]{\tiny $k$} (y.east);
\draw (y.south) to node[auto,inner sep=1pt]{\tiny $p$} (d.center);
\end{tikzpicture}
\otimes
\begin{tikzpicture}[baseline=(b.base),hvector/.style={draw=blue!50,fill=blue!20,thick},
turn/.style={circle,draw=blue!50,fill=blue!20,thick,inner sep=2pt},basis/.style={circle,draw=red!50,fill=red!20,thick,inner sep=2pt}]
\node (b) [basis]{$\lambda$};
\node (ul) [above= of b.west,anchor=center]{};
\node (ur) [above= of b.east,anchor=center]{};
\node (d) [below= of b.center,anchor=center]{};
\draw (b.west) to node[auto,swap,inner sep=1pt]{\tiny $o$} (ul.center);
\draw (b.east) to node[auto,inner sep=1pt]{\tiny $i$} (ur.center);
\draw (b.south) to node[auto,inner sep=1pt]{\tiny $r$} (d.center);
\end{tikzpicture}
=
\sum_{i,\ldots,r,\iota,\ldots,\mu}
\left\langle
\begin{tikzpicture}[baseline=(bl.center),hvector/.style={draw=blue!50,fill=blue!20,thick},
turn/.style={circle,draw=blue!50,fill=blue!20,thick,inner sep=2pt},basis/.style={circle,draw=red!50,fill=red!20,thick,inner sep=2pt}]
\node (b) [basis]{$\lambda$};
\node (a) [basis,above= of b.east,anchor=west]{$\mu$};
\node (x) [basis,below= of b.east,anchor=west] {$\iota$};
\node (y) [basis,below= of x.west,anchor=east]{$\kappa$};
\node (u) [above= of a.center,anchor=center]{};
\node (d) [below= of y.center,anchor=center]{};
\node (bl)[below= .5 of b.center,anchor=center]{};
\draw (a.north) to node[auto,swap,inner sep=1pt]{\tiny $p$} (u.center);
\draw (b.north) to node[auto,swap,inner sep=1pt]{\tiny $r$} (a.west);
\draw (x.east) to node[auto,inner sep=1pt]{\tiny $j$} (a.east);
\draw (b.east) to node[auto,inner sep=1pt]{\tiny $i$} (x.west);
\draw (b.west) to node[auto,inner sep=1pt]{\tiny $o$} (y.west);
\draw (x.south) to  node[auto,inner sep=1pt]{\tiny $k$} (y.east);
\draw (y.south) to node[auto,inner sep=1pt]{\tiny $p$} (d.center);
\end{tikzpicture}
\right\rangle
\begin{tikzpicture}[baseline=(x.base),hvector/.style={draw=blue!50,fill=blue!20,thick},
turn/.style={circle,draw=blue!50,fill=blue!20,thick,inner sep=2pt},basis/.style={circle,draw=red!50,fill=red!20,thick,inner sep=2pt}]
\node (b) [basis]{$\iota$};
\node (x) [hvector,below= of b.east,anchor=west] {$x$};
\node (y) [hvector,below= of x.west,anchor=east]{$y$};
\node (z) [hvector,below= of y.west,anchor=east]{$z$};
\node (a) [basis,above=of b.west,anchor=east]{$\kappa$};
\node (ul) [above= of a.center,anchor=center]{};
\node (ur) [above=3 of x.east,anchor=center]{};
\node (d) [below= of z.center,anchor=center]{};
\draw (a.north) to node[auto,swap,inner sep=1pt]{\tiny $p$} (ul.center);
\draw (a.west) to node[auto,inner sep=1pt]{\tiny $o$} (z.north west);
\draw (x.north east) to node[auto,inner sep=1pt]{\tiny $l$} (ur.center);
\draw (a.east) to node[auto,inner sep=1pt]{\tiny $k$} (b.north);
\draw (b.west) to node[auto,inner sep=1pt]{\tiny $i$} (y.north west);
\draw (b.east) to node[auto,inner sep=1pt]{\tiny $j$} (x.north west);
\draw (x.south) to  node[auto,inner sep=1pt]{\tiny $n$} (y.north east);
\draw (y.south) to node[auto,inner sep=1pt]{\tiny $m$} (z.north east);
\draw (z.south) to node[auto,inner sep=1pt]{\tiny $q$} (d.center);
\end{tikzpicture}
\otimes
\begin{tikzpicture}[baseline=(b.base),hvector/.style={draw=blue!50,fill=blue!20,thick},
turn/.style={circle,draw=blue!50,fill=blue!20,thick,inner sep=2pt},basis/.style={circle,draw=red!50,fill=red!20,thick,inner sep=2pt}]
\node (b) [basis]{$\mu$};
\node (ul) [above= of b.west,anchor=center]{};
\node (ur) [above= of b.east,anchor=center]{};
\node (d) [below= of b.center,anchor=center]{};
\draw (b.west) to node[auto,swap,inner sep=1pt]{\tiny $r$} (ul.center);
\draw (b.east) to node[auto,inner sep=1pt]{\tiny $j$} (ur.center);
\draw (b.south) to node[auto,inner sep=1pt]{\tiny $p$} (d.center);
\end{tikzpicture}
\otimes
\begin{tikzpicture}[baseline=(b.base),hvector/.style={draw=blue!50,fill=blue!20,thick},
turn/.style={circle,draw=blue!50,fill=blue!20,thick,inner sep=2pt},basis/.style={circle,draw=red!50,fill=red!20,thick,inner sep=2pt}]
\node (b) [basis]{$\lambda$};
\node (ul) [above= of b.west,anchor=center]{};
\node (ur) [above= of b.east,anchor=center]{};
\node (d) [below= of b.center,anchor=center]{};
\draw (b.west) to node[auto,swap,inner sep=1pt]{\tiny $o$} (ul.center);
\draw (b.east) to node[auto,inner sep=1pt]{\tiny $i$} (ur.center);
\draw (b.south) to node[auto,inner sep=1pt]{\tiny $r$} (d.center);
\end{tikzpicture}
\]
\[
=
\sum_{i,\ldots,r,\lambda,\mu}
g_{i,j}\begin{tikzpicture}[baseline=(x.base),hvector/.style={draw=blue!50,fill=blue!20,thick},
turn/.style={circle,draw=blue!50,fill=blue!20,thick,inner sep=2pt},basis/.style={circle,draw=red!50,fill=red!20,thick,inner sep=2pt}]
\node (x) [hvector] {$x$};
\node (y) [hvector,below= of x.west,anchor=east]{$y$};
\node (z) [hvector,below= of y.west,anchor=east]{$z$};
\node (b) [basis,above=2 of y.west,anchor=east]{$\lambda$};
\node (a) [basis,above=2 of x.west,anchor=east]{$\mu$};
\node (ul) [above= of a.center,anchor=center]{};
\node (ur) [above=3 of x.east,anchor=center]{};
\node (d) [below= of z.center,anchor=center]{};
\draw (a.north) to node[auto,swap,inner sep=1pt]{\tiny $p$} (ul.center);
\draw (a.east) to node[auto,inner sep=1pt]{\tiny $j$} (x.north west);
\draw (x.north east) to node[auto,inner sep=1pt]{\tiny $l$} (ur.center);
\draw (a.west) to node[auto,inner sep=1pt]{\tiny $r$} (b.north);
\draw (b.west) to node[auto,inner sep=1pt]{\tiny $o$} (z.north west);
\draw (b.east) to node[auto,inner sep=1pt]{\tiny $i$} (y.north west);
\draw (x.south) to  node[auto,inner sep=1pt]{\tiny $n$} (y.north east);
\draw (y.south) to node[auto,inner sep=1pt]{\tiny $m$} (z.north east);
\draw (z.south) to node[auto,inner sep=1pt]{\tiny $q$} (d.center);
\end{tikzpicture}
\otimes
\begin{tikzpicture}[baseline=(b.base),hvector/.style={draw=blue!50,fill=blue!20,thick},
turn/.style={circle,draw=blue!50,fill=blue!20,thick,inner sep=2pt},basis/.style={circle,draw=red!50,fill=red!20,thick,inner sep=2pt}]
\node (b) [basis]{$\mu$};
\node (ul) [above= of b.west,anchor=center]{};
\node (ur) [above= of b.east,anchor=center]{};
\node (d) [below= of b.center,anchor=center]{};
\draw (b.west) to node[auto,swap,inner sep=1pt]{\tiny $r$} (ul.center);
\draw (b.east) to node[auto,inner sep=1pt]{\tiny $j$} (ur.center);
\draw (b.south) to node[auto,inner sep=1pt]{\tiny $p$} (d.center);
\end{tikzpicture}
\otimes
\begin{tikzpicture}[baseline=(b.base),hvector/.style={draw=blue!50,fill=blue!20,thick},
turn/.style={circle,draw=blue!50,fill=blue!20,thick,inner sep=2pt},basis/.style={circle,draw=red!50,fill=red!20,thick,inner sep=2pt}]
\node (b) [basis]{$\lambda$};
\node (ul) [above= of b.west,anchor=center]{};
\node (ur) [above= of b.east,anchor=center]{};
\node (d) [below= of b.center,anchor=center]{};
\draw (b.west) to node[auto,swap,inner sep=1pt]{\tiny $o$} (ul.center);
\draw (b.east) to node[auto,inner sep=1pt]{\tiny $i$} (ur.center);
\draw (b.south) to node[auto,inner sep=1pt]{\tiny $r$} (d.center);
\end{tikzpicture}
 =\sum_{i,\ldots,r,\lambda} g_{i,j} {\tau}_{12}(
 \begin{tikzpicture}[baseline=(x.base),hvector/.style={draw=blue!50,fill=blue!20,thick}]
\node (x) [hvector] {$x$};
\node (i) [above=of x.west ]{};
\node (j) [above=of x.east ]{};
\node (k) [below=of x.center ]{};
\draw (x.north west) to node[auto,swap,inner sep=1pt]{\tiny $j$} (i);
\draw (x.north east) to node[auto,inner sep=1pt]{\tiny $l$} (j);
\draw (x.south) to node[auto,inner sep=1pt]{\tiny $n$} (k);
\end{tikzpicture}
\otimes
 \begin{tikzpicture}[baseline=(x.base),hvector/.style={draw=blue!50,fill=blue!20,thick},
turn/.style={circle,draw=blue!50,fill=blue!20,thick,inner sep=2pt},basis/.style={circle,draw=red!50,fill=red!20,thick,inner sep=2pt}]
\node (b) [basis]{$\lambda$};
\node (x) [hvector,below= of b.east,anchor=west] {$y$};
\node (y) [hvector,below= of x.west,anchor=east]{$z$};
\node (ul) [above= of b.center,anchor=center]{};
\node (ur) [above=2 of x.east,anchor=center]{};
\node (d) [below= of y.center,anchor=center]{};
\draw (b.north) to node[auto,swap,inner sep=1pt]{\tiny $r$} (ul.center);
\draw (x.north east) to node[auto,inner sep=1pt]{\tiny $n$} (ur.center);
\draw (b.east) to node[auto,inner sep=1pt]{\tiny $i$} (x.north west);
\draw (b.west) to node[auto,inner sep=1pt]{\tiny $o$} (y.north west);
\draw (x.south) to  node[auto,inner sep=1pt]{\tiny $m$} (y.north east);
\draw (y.south) to node[auto,inner sep=1pt]{\tiny $q$} (d.center);
\end{tikzpicture}
\otimes
\begin{tikzpicture}[baseline=(b.base),hvector/.style={draw=blue!50,fill=blue!20,thick},
turn/.style={circle,draw=blue!50,fill=blue!20,thick,inner sep=2pt},basis/.style={circle,draw=red!50,fill=red!20,thick,inner sep=2pt}]
\node (b) [basis]{$\lambda$};
\node (ul) [above= of b.west,anchor=center]{};
\node (ur) [above= of b.east,anchor=center]{};
\node (d) [below= of b.center,anchor=center]{};
\draw (b.west) to node[auto,swap,inner sep=1pt]{\tiny $o$} (ul.center);
\draw (b.east) to node[auto,inner sep=1pt]{\tiny $i$} (ur.center);
\draw (b.south) to node[auto,inner sep=1pt]{\tiny $r$} (d.center);
\end{tikzpicture}
)
 \]
 \[
=\sum_{i,\ldots,n}g_{i,j}
 {\tau}_{12}{\tau}_{23}(
 \begin{tikzpicture}[baseline=(x.base),hvector/.style={draw=blue!50,fill=blue!20,thick}]
\node (x) [hvector] {$x$};
\node (i) [above=of x.west ]{};
\node (j) [above=of x.east ]{};
\node (k) [below=of x.center ]{};
\draw (x.north west) to node[auto,swap,inner sep=1pt]{\tiny $j$} (i);
\draw (x.north east) to node[auto,inner sep=1pt]{\tiny $l$} (j);
\draw (x.south) to node[auto,inner sep=1pt]{\tiny $n$} (k);
\end{tikzpicture}
\otimes
\begin{tikzpicture}[baseline=(x.base),hvector/.style={draw=blue!50,fill=blue!20,thick}]
\node (x) [hvector] {$y$};
\node (i) [above=of x.west ]{};
\node (j) [above=of x.east ]{};
\node (k) [below=of x.center ]{};
\draw (x.north west) to node[auto,swap,inner sep=1pt]{\tiny $i$} (i);
\draw (x.north east) to node[auto,inner sep=1pt]{\tiny $n$} (j);
\draw (x.south) to node[auto,inner sep=1pt]{\tiny $m$} (k);
\end{tikzpicture}
\otimes z)
={\tau}_{12}{\tau}_{23} ({}^\bullet \pi)_{21}(x\otimes y\otimes z).
\]
 All these equalities  follow directly from the definitions except the fifth equality. The sum on its left-hand side
  is preserved if we insert an additional
 factor $ g_{o,k}$. Indeed, if the   triangular block with the vertices labeled $\mu, \lambda, \iota$  contributes
  non-zero, then necessarily $g_{o,k}=1$. This   allows us to   sum up
  over all $p,  \kappa$ and then  over all $k, \iota$ to obtain the fifth equality.

({\bf ii}) For $x,y \in \check H$,
\[
{\tau}_{21}\bar {\tau} (x\otimes y)=
\sum_{i,\ldots,n,\lambda}{\tau}_{21}(
\begin{tikzpicture}[baseline=(x.base),hvector/.style={draw=blue!50,fill=blue!20,thick},
turn/.style={circle,draw=blue!50,fill=blue!20,thick,inner sep=2pt},basis/.style={circle,draw=red!50,fill=red!20,thick,inner sep=2pt}]
\node (b) [basis]{$\lambda$};
\node (x) [hvector,below= of b.west,anchor=east] {$x$};
\node (y) [hvector,below= of x.east,anchor=west]{$y$};
\node (ur) [above= of b.center,anchor=center]{};
\node (ul) [above=2 of x.west,anchor=center]{};
\node (d) [below= of y.center,anchor=center]{};
\draw (b.north) to node[auto,inner sep=1pt]{\tiny $n$} (ur.center);
\draw (x.north west) to node[auto,swap,inner sep=1pt]{\tiny $i$} (ul.center);
\draw (b.east) to node[auto,swap,inner sep=1pt]{\tiny $l$} (y.north east);
\draw (b.west) to node[auto,swap,inner sep=1pt]{\tiny $j$} (x.north east);
\draw (x.south) to  node[auto,swap,inner sep=1pt]{\tiny $k$} (y.north west);
\draw (y.south) to node[auto,swap,inner sep=1pt]{\tiny $m$} (d.center);
\end{tikzpicture}
\otimes
\begin{tikzpicture}[baseline=(b.base),hvector/.style={draw=blue!50,fill=blue!20,thick},
turn/.style={circle,draw=blue!50,fill=blue!20,thick,inner sep=2pt},basis/.style={circle,draw=red!50,fill=red!20,thick,inner sep=2pt}]
\node (b) [basis]{$\lambda$};
\node (ul) [above= of b.west,anchor=center]{};
\node (ur) [above= of b.east,anchor=center]{};
\node (d) [below= of b.center,anchor=center]{};
\draw (b.west) to node[auto,swap,inner sep=1pt]{\tiny $j$} (ul.center);
\draw (b.east) to node[auto,inner sep=1pt]{\tiny $l$} (ur.center);
\draw (b.south) to node[auto,inner sep=1pt]{\tiny $n$} (d.center);
\end{tikzpicture}
)
=
\sum_{i,\ldots,o,\lambda,\mu}
\begin{tikzpicture}[baseline=(b.base),hvector/.style={draw=blue!50,fill=blue!20,thick},
turn/.style={circle,draw=blue!50,fill=blue!20,thick,inner sep=2pt},basis/.style={circle,draw=red!50,fill=red!20,thick,inner sep=2pt}]
\node (b) [basis]{$\mu$};
\node (ul) [above= of b.west,anchor=center]{};
\node (ur) [above= of b.east,anchor=center]{};
\node (d) [below= of b.center,anchor=center]{};
\draw (b.west) to node[auto,swap,inner sep=1pt]{\tiny $i$} (ul.center);
\draw (b.east) to node[auto,inner sep=1pt]{\tiny $j$} (ur.center);
\draw (b.south) to node[auto,inner sep=1pt]{\tiny $o$} (d.center);
\end{tikzpicture}
\otimes
\begin{tikzpicture}[baseline=(b.base),hvector/.style={draw=blue!50,fill=blue!20,thick},
turn/.style={circle,draw=blue!50,fill=blue!20,thick,inner sep=2pt},basis/.style={circle,draw=red!50,fill=red!20,thick,inner sep=2pt}]
\node (b) [basis]{$\lambda$};
\node (x) [hvector,below= of b.west,anchor=east] {$x$};
\node (y) [hvector,below= of x.east,anchor=west]{$y$};
\node (d) [below= of y.center,anchor=center]{};
\node (bb) [basis,above=of b.center,anchor=center]{$\lambda$};
\node (bbb)[basis,above=3 of x.center,anchor=center]{$\mu$};
\node (ur) [above=2 of bb.east,anchor=center]{};
\node (ul) [above= of bbb.center,anchor=center]{};
\draw (b.north) to node[auto,inner sep=1pt]{\tiny $n$} (bb.south);
\draw (bb.west) to node[auto,inner sep=1pt]{\tiny $j$} (bbb.east);
\draw (bb.east) to node[auto,inner sep=1pt]{\tiny $l$} (ur.center);
\draw (bbb.north) to node[auto,inner sep=1pt]{\tiny $o$} (ul.center);
\draw (x.north west) to node[auto,swap,inner sep=1pt]{\tiny $i$} (bbb.west);
\draw (b.east) to node[auto,swap,inner sep=1pt]{\tiny $l$} (y.north east);
\draw (b.west) to node[auto,swap,inner sep=1pt]{\tiny $j$} (x.north east);
\draw (x.south) to  node[auto,swap,inner sep=1pt]{\tiny $k$} (y.north west);
\draw (y.south) to node[auto,swap,inner sep=1pt]{\tiny $m$} (d.center);
\end{tikzpicture}
\]
\[
=
\sum_{i,\ldots,m,\mu}g_{j,l}
\begin{tikzpicture}[baseline=(b.base),hvector/.style={draw=blue!50,fill=blue!20,thick},
turn/.style={circle,draw=blue!50,fill=blue!20,thick,inner sep=2pt},basis/.style={circle,draw=red!50,fill=red!20,thick,inner sep=2pt}]
\node (b) [basis]{$\mu$};
\node (ul) [above= of b.west,anchor=center]{};
\node (ur) [above= of b.east,anchor=center]{};
\node (d) [below= of b.center,anchor=center]{};
\draw (b.west) to node[auto,swap,inner sep=1pt]{\tiny $i$} (ul.center);
\draw (b.east) to node[auto,inner sep=1pt]{\tiny $j$} (ur.center);
\draw (b.south) to node[auto,inner sep=1pt]{\tiny $k$} (d.center);
\end{tikzpicture}
\otimes
\begin{tikzpicture}[baseline=(b.base),hvector/.style={draw=blue!50,fill=blue!20,thick},
turn/.style={circle,draw=blue!50,fill=blue!20,thick,inner sep=2pt},basis/.style={circle,draw=red!50,fill=red!20,thick,inner sep=2pt}]
\node (x) [hvector] {$x$};
\node (y) [hvector,below= of x.east,anchor=west]{$y$};
\node (d) [below= of y.center,anchor=center]{};
\node (bbb)[basis,above= of x.center,anchor=center]{$\mu$};
\node (ur) [above=3 of y.east,anchor=center]{};
\node (ul) [above= of bbb.center,anchor=center]{};
\draw (bbb.north) to node[auto,inner sep=1pt]{\tiny $k$} (ul.center);
\draw (x.north west) to node[auto,swap,inner sep=1pt]{\tiny $i$} (bbb.west);
\draw (x.north east) to node[auto,inner sep=1pt]{\tiny $j$} (bbb.east);
\draw (x.south) to  node[auto,swap,inner sep=1pt]{\tiny $k$} (y.north west);
\draw (y.south) to node[auto,swap,inner sep=1pt]{\tiny $m$} (d.center);
\draw (y.north east) to node[auto,inner sep=1pt]{\tiny $l$} (ur.center);
\end{tikzpicture}
=\sum_{i,\ldots,m}g_{j,l}
 \begin{tikzpicture}[baseline=(x.base),hvector/.style={draw=blue!50,fill=blue!20,thick}]
\node (x) [hvector] {$x$};
\node (i) [above=of x.west ]{};
\node (j) [above=of x.east ]{};
\node (k) [below=of x.center ]{};
\draw (x.north west) to node[auto,swap,inner sep=1pt]{\tiny $i$} (i);
\draw (x.north east) to node[auto,inner sep=1pt]{\tiny $j$} (j);
\draw (x.south) to node[auto,inner sep=1pt]{\tiny $k$} (k);
\end{tikzpicture}
\otimes
\begin{tikzpicture}[baseline=(x.base),hvector/.style={draw=blue!50,fill=blue!20,thick}]
\node (x) [hvector] {$y$};
\node (i) [above=of x.west ]{};
\node (j) [above=of x.east ]{};
\node (k) [below=of x.center ]{};
\draw (x.north west) to node[auto,swap,inner sep=1pt]{\tiny $k$} (i);
\draw (x.north east) to node[auto,inner sep=1pt]{\tiny $l$} (j);
\draw (x.south) to node[auto,inner sep=1pt]{\tiny $m$} (k);
\end{tikzpicture}
=\pi^\bullet (x\otimes y).
\]
  The second inversion relation is proved similarly.
 \end{proof}

 \section{The $6j$-symbols}

\subsection{Notation}  For any $i,j,k\in I$, the non-degenerate pairing $H_{ij}^k\otimes
H^{ij}_k\to \FK$ defined in Lemma~\ref{L:1}  will be denoted
$\ast^k_{ij}$. Composing this pairing with the flip $H^{ij}_k\otimes
H_{ij}^k\to H_{ij}^k\otimes H^{ij}_k$, we obtain a non-degenerate
pairing   $H^{ij}_k\otimes H_{ij}^k\to \FK$ denoted $\ast^{ij}_k$. We
shall use these pairings to identify the dual of $H_{ij}^k$ with
$H^{ij}_k$ and the dual of $H^{ij}_k$ with $H_{ij}^k$. The pairings
$\ast^k_{ij}$ and $\ast_k^{ij}$ induce the tensor contractions
$$U\otimes H_{ij}^k\otimes V \otimes H^{ij}_k \otimes W\to U\otimes
V\otimes  W,$$
 $$U \otimes  H^{ij}_k\otimes V \otimes  H_{ij}^k \otimes W\to U\otimes  V\otimes  W ,$$ where
 $U$, $V$, $W$ are
 arbitrary $\FK$-vector spaces. These tensor contractions will be denoted by the same symbols $\ast^k_{ij}$ and
 $\ast_k^{ij}$ respectively.

\subsection{Definition of $6j$-symbols}\label{section42Ibelieve} For any $i,j,k,l,m,n\in I$, the restriction
of   $T:H^{\otimes 4}\to \FK$ to the tensor product
$$H^m_{kl}\otimes H^k_{ij}\otimes H^{jl}_n \otimes H^{in}_m \subset
\hat H \otimes \hat H\otimes \check H\otimes \check H \subset
H^{\otimes 4}$$ gives a vector in the $\FK$-vector space
\begin{equation}\label{ambientvectorspace} \Hom (H^m_{kl}\otimes H^k_{ij}\otimes H^{jl}_n \otimes
H^{in}_m,\FK)= H_m^{kl}\otimes H_k^{ij}\otimes H_{jl}^n \otimes
H_{in}^m . \end{equation}
This vector is denoted \begin{equation}\label{positive6j-symbol}\sj
ijklmn
\end{equation}  and called the {\it positive $6j$-symbol} determined by the
tuple $i,j,k,l,m,n$. In graphical notation, the
$6j$-symbol~\eqref{positive6j-symbol} is   the summand in the
definition of $T$ in Section~\ref{TheformsTand} corresponding to the
tuple $i,j,k,l,m,n\in I$. Thus,  for any $u,v,x,y\in H$,
$$T(u\otimes v\otimes x\otimes y)$$
$$=\sum_{i, j, k, l, m, n\in I} \ast^m_{kl}   \ast^k_{ij}  \ast^{jl}_n
\ast^{in}_m \big (\pi^m_{kl} (u) \otimes \pi^k_{ij} (v) \otimes
\pi^{jl}_n (x) \otimes \pi^{in}_m (y) \otimes \sj ijklmn \big ).$$
The adjoint operator  $\tau \in \End (\check H^{\otimes 2})$ expands
as follows: for any $x,y \in \check H$,
$$\tau(x\otimes y)= \sum_{i, j, k, l, m, n\in I}    \ast^{jl}_n
\ast^{in}_m \big ( \pi^{jl}_n (x) \otimes \pi^{in}_m (y) \otimes \sj
ijklmn \big ). $$

Similarly restricting $\bar T$ to the tensor product
$H^m_{in}\otimes H^n_{jl}\otimes H^{ij}_k \otimes H^{kl}_m$, we
obtain the {\it negative $6j$-symbol}
\begin{equation}\label{negative6j-symbol}\sjn
ijklmn \in H_m^{in}\otimes H_n^{jl}\otimes H_{ij}^k \otimes
H_{kl}^m.
\end{equation}
For any $u,v,x,y\in H$,
$$\bar T(u\otimes v\otimes x\otimes y)$$
$$=\sum_{i, j, k, l, m, n\in I} \ast^m_{in}   \ast^n_{jl}  \ast^{ij}_k
\ast^{kl}_m \big (\pi^m_{in} (u) \otimes \pi^n_{jl} (v) \otimes
\pi^{ij}_k (x) \otimes \pi^{kl}_m (y) \otimes \sjn ijklmn \big ).$$
The adjoint operator  $  \bar \tau \in \End (\check H^{\otimes 2})$
expands   as follows: for any $x,y \in \check H$,
$$\bar \tau(x\otimes y)=\sum_{i, j, k, l, m, n\in I}   \ast^{ij}_k
\ast^{kl}_m \big (
\pi^{ij}_k (x) \otimes \pi^{kl}_m (y) \otimes \sjn ijklmn \big ). $$

\subsection{Identities} The properties of the forms $T$ and $\bar T$
established in Section~\ref{Section3now} can be rewritten in terms
of the $6j$-symbols. Formula~\eqref{E:sym01} yields
$$ \sj ijklmn=  P_{(4321)} A_1  A_3^*  \left ( \sjn {i^*}kjlnm
\right ),$$ where $A_1$ is induced by the restriction of $A$ to
$H^{i^*m}_n$  and $A_3^*$ is induced by the restriction of $A^*$ to
$H^{j}_{i^*k}$. Formula~\eqref{E:sym12} yields
$$ \sj ijklmn=  P_{(23)} A_2^*  B_3^*  \left ( \sjn k{j^*}inml
\right ),$$ where $A_2^*$ is induced by the restriction of $A^*$ to
$H^{{j^*}n}_l$  and $B_3^*$ is induced by the restriction of $B^*$
to $H^{i}_{kj^*}$.   Formula~\eqref{E:sym23} yields
$$ \sj ijklmn=  P_{(1234)} B_2^*  B_4  \left ( \sjn inm{l^*}kj
\right ),$$ where $B_2^*$ is induced by the restriction of $B^*$ to
$H^{nl^*}_j$  and $B_4$ is induced by the restriction of $B_4$ to
$H_{ml^*}^{k}$.

Note for the record that Formula~\eqref{actionof12} yields
$$ \sj ijklmn=  P_{(12)} (BA)^*_1  (BA)_2^*  (AB)_4 \left ( \sjn
jln{m^*}{i^*}{k^*} \right ),$$ where the operator $(BA)^*_1$ is
induced by the restriction of $(BA)^*=A^*B^*$ to $H^{jk^*}_{i^*}$;
the operator $(BA)^*_2$ is induced by the restriction of $(BA)^* $
to $H^{lm^*}_{k^*}$, and $(AB)_4$ is induced by the restriction of
$AB$ to $H^{i^*}_{nm^*}$. Formula~\eqref{actionof34} yields
$$ \sj ijklmn=  P_{(34)} (BA)_1  (AB)_3^*  (AB)_4^*  \left ( \sjn
 {m^*}i{n^*}j{l^*}k \right ),$$ where the operator $(BA)_1$ is
induced by the restriction of $ BA  $ to $H^{m^*k}_{l^*}$; the
operator $(AB)^*_3$ is induced by the restriction of $(AB)^* $ to
$H^{n^*}_{m^*i}$, and $(AB)^*_4$ is induced by the restriction of
$(AB)^*$ to $H^{l^*}_{n^*j}$.

The pentagon identity   yields that for any $j_0,j_1,
 \ldots , j_8\in I$,
 \[\sum_{j\in I} \ast^{jj_4}_{j_7} \ast^{j_2j_3}_j \ast^{j_1j}_{j_6}
 \left (\sj {j_1}{j_2}{j_5}{j_3}{j_6}j \otimes
  \sj {j_1}j{j_6}{j_4}{j_0}{j_7} \otimes
  \sj {j_2}{j_3}{j}{j_4}{j_7}{j_8} \right )\]
 \[=g_{j_2,j_3} P_{(135642)} \ast^{j_5j_8}_{j_0} \left (
 \sj {j_1}{j_2}{j_5}{j_8}{j_0}{j_7} \otimes
 \sj {j_5}{j_3}{j_6}{j_4}{j_0}{j_8} \right ).\]
Here both sides lie in the $\FK$-vector space
\begin{equation}\label{ambientvectorspacepentagon}H^{{j_5}{j_3}}_{j_6} \otimes H^{{j_1}{j_2}}_{j_5} \otimes
H^{{j_6}{j_4}}_{j_0} \otimes H_{j_1 j_7}^{j_0} \otimes H^{j_3
j_4}_{j_8}\otimes H_{j_2 j_8}^{j_7}.\end{equation}

To rewrite the inversion relations in terms of the $6j$-symbols,
observe that the   transpose of the pairing $\ast^k_{ij}:
H^k_{ij}\otimes H_k^{ij}\to \FK$ is a homomorphism $\FK\to
H^{ij}_k\otimes H_{ij}^k$. The image of the unit $1\in \FK$ under this
  homomorphism is denoted by $\delta^{ij}_k$. In the notation
of Section~\ref{section33now}, we have $\delta^{ij}_k=\sum_\alpha
e^{ij}_{k\alpha}  \otimes
 e_{ij}^{k\alpha} $. The   relation $\tau_{21}\bar \tau=\pi^\bullet$ may be
rewritten as the identity
$$\sum_{n\in I} \ast^n_{jl}\ast^m_{in} \sj ijklmn \otimes \sjn
ij{k'}lmn=\delta_{k'}^k \, g_{j,l} \,  P_{(432)} (\delta^{kl}_{m}\otimes
\delta^{ij}_k) $$ for all $i,j,k,k',l,m \in I$. The   relation  $\bar \tau \tau_{21}={}^\bullet \pi$ may be
rewritten as the identity
$$\sum_{k\in I} \ast^k_{ij}\ast^m_{kl} \sjn ijklm{n'} \otimes \sj
ij{k}lmn=\delta_{n'}^n\,  g_{i,j}\, P_{(432)} (\delta^{in}_{m}\otimes
\delta^{jl}_n) $$ for all $i,j, l,m ,n,n' \in I$.

\begin{remark} As an exercise, the reader may prove that
  ${\tau}_{12}{\tau}_{23} ({}^\bullet
 \pi)_{21}=(\pi^{\bullet})_{32}  \tau_{12}{\tau}_{23}$.
This formula
  does not give non-trivial identities between
 the $6j$-symbols.
 \end{remark}

\section{The $T$-calculus}

\subsection{$T$-equalities}
 We say that  two endomorphisms $a,b$ of $H$ are
  \emph{$T$-equal} and write $a\stackrel{T}{=} b$ if $Ta_i=Tb_i$ for all $i=1,2,3,  4$.
   It is clear that the $T$-equality is an equivalence relation. If
    $a\stackrel{T}{=} b$, then $ac\stackrel{T}{=} bc$ for all
   $c\in \End (H)$.


   Though the definition of the $T$-equality
   involves four conditions,    two of them may be eliminated as is clear from the following lemma.
   \begin{lemma}\label{L:1234}
For any   $a, b \in \End (H)$, we have $Ta_1=Tb_1
\Longleftrightarrow Ta_2=Tb_2$ and $Ta_3=Tb_3 \Longleftrightarrow
Ta_4=Tb_4$.
   \end{lemma}
\begin{proof}
Formulas~\eqref{E:sym}  and the  identity $P_{(431)}=P_{(4321)}
P_{(23)}$
  imply that
\begin{equation}\label{E:tp431}
TP_{(431)} =T\Aop_1^*(\Bop\Aop)_2\Aop_3 .
\end{equation}
Multiplying   on the right by $  A^*_1 (AB)_2 A_3
P_{(134)}=P_{(134)} A_1 (AB)_2 A^*_4$, we obtain
\begin{equation}\label{E:tp431+}
TP_{(134)} =T\Aop_1 (\Aop\Bop)_2 \Aop^*_4  .
\end{equation}
Similar arguments prove that
\begin{equation}\label{E:tp124}
TP_{(124)} =T\Bop_2(\Aop\Bop)_3\Bop_4^*,
\end{equation}
\begin{equation}\label{E:tp124+}
TP_{(421)} =T\Bop^*_1(\Bop\Aop)_3  \Bop_4  .
\end{equation}
For  each $i=1,2,3,4$ one of the
Equations~\eqref{E:tp431}--\eqref{E:tp124+} has the form
\begin{equation}\label{E:gen-str}
TP_\sigma=TX_kY_lZ_m,
\end{equation}
where
\[
\sigma\in \mathbb{S}_4,\quad
 X,Y,Z\in \{A, A^*, B, B^*, AB, BA\} ,\quad
\{k,l,m\}=\{1,2,3,4\}\setminus\{i\},
\]
and the set $\{i,\sigma(i)\}$ is either $\{1,2\}$ or $\{3,4\}$.
  Set $j=\sigma(i)$ and  observe that
\[
Ta_i=TP_\sigma (X_kY_lZ_m)^{-1}a_i=TP_\sigma a_i (X_kY_lZ_m)^{-1}=Ta_jP_\sigma (X_kY_lZ_m)^{-1}.\]
Also, $Tb_i= Tb_jP_\sigma (X_kY_lZ_m)^{-1}$.  Hence, $Ta_i=Tb_i$ if
and only if $Ta_j=Tb_j$.
\end{proof}
\begin{corollary}
The $T$-equality $a\teq b$ holds if and only if $Ta_i=Tb_i$ for some
$i\in \{1,2\}$ and for some   $i\in\{3,4\}$.
\end{corollary}

   \subsection{$T$-scalars} An endomorphism $t$ of $H$ is a \emph{$T$-scalar}
    if $t$ has a  transpose   $t^*$   and
   \begin{equation}\label{E:q1=q2}
   Tt_1=Tt_2=Tt_3^*=Tt_4^* \quad \,\,  {\text {and}} \,\, \quad Tt^*_1=Tt_2^*=Tt_3=Tt_4 .
   \end{equation}
For example, all scalar automorphisms of $H$ are $T$-scalars. A more
interesting example of a $T$-scalar will be given in
Lemma~\ref{L:[lr]} below. If $t$ is a $T$-scalar, then the adjoint
operator $\tau:\check H^{\otimes 2} \to \check H^{\otimes 2} $
introduced in Section~\ref{section33now} satisfies
   \begin{equation}\label{E:q1=q2++++} t_1 \tau =t_2\tau=\tau t_1=\tau t_2 \quad \,\,  {\text {and}} \,\, \quad t^*_1 \tau
=t^*_2\tau=\tau t^*_1=\tau t^*_2.   \end{equation}

If $t\in \End (H)$ is a $T$-scalar, then so is $t^*$. If a
$T$-scalar $t$ is invertible in $\End (H)$, then $t^{-1}$ is a
$T$-scalar. Indeed, the equality $Tt_1=Tt_2$ implies that
$Tt_1^{-1}=Tt_2^{-1}$ and similarly for all the other required
equalities.

The   product of any two $T$-scalars $t,u\in \End (H)$ is a
$T$-scalar. Indeed, for any $r\in \{1,2\}$ and $s \in \{3,4\}$,
\[T(tu)_r=Tt_ru_r=Tt_s^* u_r= T  u_r t_s^* =Tu_s^* t_s^*=T(tu)^*_s\]
and similarly $ T(tu)^*_r=T(tu)_s$. Thus, the $T$-scalars form a
subalgebra of the $\FK$-algebra $\End (H)$ invariant under the
involution $a\mapsto a^*$.

If $t$ is a $T$-scalar, then $a\stackrel{T}{=}  b \Longrightarrow
ta\stackrel{T}{=} tb$ for any $a,b\in \End(H)$. Indeed,
$$T (ta)_1=Tt_1a_1=Tt_2a_1=Ta_1t_2=Tb_1t_2=Tt_2b_1=Tt_1b_1=T(tb)_1
$$ and similarly, $T(ta)_3=T(tb)_3$.


We call  an invertible endomorphism $t$ of $H$ {\it unitary}  if
$t^*=t^{-1}$. More generally, an invertible endomorphism $t$ of $H$
is  {\it $T$-unitary} if $t^*\teq t^{-1}$. For a $T$-unitary $t\in
\End(H)$, Equations \eqref{E:q1=q2} simplify to $Tt_r t_s=T$ for all
$r\in \{1,2\}$ and $s \in \{3,4\}$.

 \subsection{$T$-commutation relations} We now show that   $T$-scalars
 $T$-commute with every   product of an even number of operators
 $A,B, A^*, B^*$.
  To give a more precise statement, we define a group $F$
by the presentation
   \begin{equation}\label{E:groupF}
   {F}=\langle a,b,a^*,b^*\,\vert \, a^2=b^2=(a^*)^2=(b^*)^2=1\rangle .
   \end{equation}
Consider the group homomorphism $ F\to \mathbb{Z}/2\mathbb{Z}$
carrying the generators $a, b, a^*, b^*$ to $1\, ({\text {mod}} \,
2)$. Elements of ${F}$ belonging to the kernel of this homomorphism
are said to be {\it even}; all other elements  of $F$ are said to be
{\it odd}. In other words, an element of $F$ is even if it expands
as a product of an even number of generators and odd otherwise. The
group ${F}$ acts
   on $ H $ by   $a\mapsto A$, $b \mapsto B $, $a^* \mapsto A^*  $, $b^* \mapsto
   B^*$. The endomorphism of
   $H$ determined by $g\in F$ is denoted    ${\underline{g}}$.

   \begin{lemma}\label{L:xg=x^gg}
   For any $T$-scalar $t\in \End (H)$ any $g\in{F}$, we have
   $
   {\underline{g}}t\teq  {t} {\underline{g}} $ if $g$ is even and $
   {\underline{g}}t\teq  {t}^* {\underline{g}} $ if $g$ is odd.
\end{lemma}
\begin{proof} Fix a $T$-scalar $t\in \End (H)$. For $g\in F$, set
  $\act{t}{g}=t$  if $g$ is  even    and  $\act{t}{g}=t^*$   if $g$ is
  odd. We need to prove that $
   {\underline{g}}t\teq  \act{t}{g} {\underline{g}} $ for all $g\in {F}$.

For   $i=1,2,3,4$, set
\[
 \Delta^i=\{g\in  {F}\, \vert\ T({\underline{g}}t)_i=T(\act{t}{g}{\underline{g}})_i\} \subset F .
\]
By the previous lemma, $\Delta^1=\Delta^2$ and $\Delta^3=\Delta^4$.
We claim that for any generator $c\in \{a,b,a^*, b^*\}$, we have $
c\, \Delta^1\subset \Delta^3 $ and $ c\, \Delta^3\subset \Delta^1$.
This will imply that the set $ \Delta^1 \cap \Delta^3\subset {F}$ is
closed under left multiplication by the generators of $ {F}$. Since
this set contains the neutral element of $F$,   we have $ \Delta^1
\cap \Delta^3 = {F}$. In other words, $\Delta^i=F$ for all
$i=1,2,3,4$. This means that $
   {\underline{g}}t\teq  \act{t}{g} {\underline{g}} $ for all $g\in {F}$.

  To prove our claim, consider  again Equality~\eqref{E:gen-str}.   Pick any    $g\in
\Delta^{\sigma (k)}$. Then
\[
T(X {\underline{g}} t)_k=TP_\sigma (X_kY_lZ_m)^{-1}(X {\underline{g}} t)_k=TP_\sigma
(Y_lZ_m)^{-1}({\underline{g}} t)_k \] \[ =T({\underline{g}} t)_{\sigma(k)}P_\sigma (Y_lZ_m)^{-1}
  = T(\act{t}{g} {\underline{g}})_{\sigma(k)}P_\sigma (Y_lZ_m)^{-1} =
  T t^{g}_{\sigma(k)} {\underline{g}}_{\sigma(k)}P_\sigma (Y_lZ_m)^{-1}, \] where
the fourth equality follows from the inclusion $g\in \Delta^{\sigma
(k)}$. Since $t\in \End  H $ is a $T$-scalar, $Tt^{g}_{\sigma(k)}=
Tt'_{\sigma(i)}$ for some $t' \in\{t,t^*\}$. Then
\[Tt^{g}_{\sigma(k)} {\underline{g}}_{\sigma(k)}P_\sigma (Y_lZ_m)^{-1} =Tt'_{\sigma(i)} {\underline{g}}_{\sigma(k)}P_\sigma (Y_lZ_m)^{-1}\\=TP_\sigma
(Y_lZ_m)^{-1}t'_{i}{\underline{g}}_{k}=Tt'_{i}(X{\underline{g}})_{k}
\] where the last equality follows from~\eqref{E:gen-str}. Since $t$ is a $T$-scalar, $Tt'_{i}  =Tt''_{k}$ for some $
t''\in\{t,t^*\}$. Recall that $X= {\underline{x}}$ for some $x \in
\{a, a^*, b, b^*, ab, ba\}$. A case by case analysis shows that
$t''=\act{t}{x g}$. Combining the formulas above, we obtain that
$T({\underline{xg}}     t)_k=T(X {\underline{g}} t)_k=T( \act{t}{x
g} {\underline{x g}} )_k$, i.e., $xg\in \Delta^{ k }$. Thus, $x
\Delta^{\sigma(k)}\subset \Delta^{ k }$. Applying this inclusion to
all forms \eqref{E:tp431}--\eqref{E:tp124+} of~\eqref{E:gen-str} and
to all possible choices of $k$, we obtain $ c\, \Delta^1\subset
\Delta^3 $ and $ c\, \Delta^3\subset \Delta^1$ for all $c\in
\{a,b,a^*, b^*\}$.
\end{proof}


\section{Further properties of $T$}\label{section6}

  To study   the form $T$ we introduce the operators $$L=A^*A ,
\,\,  R=B^*B, \,\, C=(AB)^3 \in \End(H).$$   We shall study
  these operators and   show that the commutator of $L$
  and $R$ is a $T$-scalar. Though this fact will not be directly
  used in the sequel, the properties of the operators $L$, $R$, $C$
  will lead us    in the next section to a notion of a $\spsi$-system.

  An operator $f\in\End (H)$ such
that $f^*=f$   is called \emph{symmetric}. An operator $f\in\End
(H)$ such that $f(H^{ij}_k) \subset H^{ij}_k$ and
$f(H_{ij}^k)\subset H_{ij}^k$ for all $i,j,k\in I$  is called
\emph{grading-preserving}.

\begin{lemma}\label{L:propAB}
The operators $L$, $R$, $C$ are     invertible, symmetric, and
grading-preser\-ving. They satisfy the following identities:
\begin{equation}\label{identity1}
ACA=BCB=C^{-1},
\end{equation}
\begin{equation}\label{identity2}
LCL^{-1}=RCR^{-1}=C,
\end{equation}
\begin{equation}\label{identity3}
ALA=L^{-1},\quad BRB=R^{-1},
\end{equation}
\begin{equation}\label{E:ara}
ARA=L^{-1}RC^{-1},\quad BLB=R^{-1}LC.
\end{equation}
\end{lemma}
\begin{proof} That $L, R,
C$ are invertible follows from the fact that $A$ and $B$ are
invertible. The inverses of these operators are computed by
$L^{-1}=AA^*$, $R^{-1}=BB^*$, and $C^{-1}=(BA)^3$. The operators $L$
and $R$ are manifestly symmetric. We have
\[C=(AB)^3= ABA BAB =(ABA) (ABA)^*.
\]
Therefore $C^*=C$. That $L, R, C$ are grading-preserving follows
  from \eqref{E:exch++} and \eqref{E:exch++bis}.

The  identities \eqref{identity1}--\eqref{E:ara} are checked as
follows:
\[
ACA=A(AB)^3A=A^2(BA)^3=(BA)^3=C^{-1},
\]
and similarly for $BCB$;
\[
LCL^{-1}=A^*ACAA^*=A^*C^{-1}A^*=(AC^{-1}A)^*=C^*=C,
\]
and similarly for $RCR^{-1}$;
\[
ALA=AA^*AA=AA^*=L^{-1},
\]
and similarly for $BRB$;
\[
L^{-1}RC^{-1}=AA^*B^*A(BA)^2=AA^*B^*(BAB)^*BA=AB^*BA=ARA,
\]
and similarly for $R^{-1}LC$.
\end{proof}

\begin{lemma}\label{L:semidir}
The following identities hold:
\begin{subequations}\label{E:symsp}
\begin{align}
\label{E:c1c2=c3c4}
TC_1C_2&=TC_3C_4,\\
\label{E:r1l2t=r3l4t}
T\Rop_1\Lop_2&=T\Rop_3\Lop_4,\\
\label{E:r1r2t=c3r4t}
T\Rop_1\Rop_2&=TC_3\Rop_4,\\
\label{E:l1c2-1t=l3l4t}
T\Lop_1&=TC_2\Lop_3\Lop_4,
\end{align}
\end{subequations}
\end{lemma}
\begin{proof} The   proof is based on a study of the   action of the standard generators of~$\mathbb{S}_4$ on~$T$ using the formulas of
 Section~\ref{TheformsTand}. We have
\[
TP_{(12)}P_{(23)}=\bar TP_{(23)}(\Bop\Aop)_1(\Bop\Aop)_3(\Aop\Bop)_4^*=
T(\Bop\Aop)_1\Bop_2(\Aop\Bop\Aop)_3(\Aop\Bop)_4^*.
\]
The  Coxeter relation  $(P_{(12)}P_{(23)})^3=1$  yields
Equation~\eqref{E:c1c2=c3c4}:
\begin{multline*}
T=T(P_{(12)}P_{(23)})^3=
T(P_{(12)}P_{(23)})^2(\Bop\Aop)_2\Bop_3(\Aop\Bop\Aop)_1(\Aop\Bop)_4^*\\=
TP_{(12)}P_{(23)}(\Bop\Aop\Bop)_3(\Aop(\Bop\Aop)^2)_2(\Bop\Aop)_1^2(\Aop\Bop)_4^{*2}\\=
T(\Bop\Aop)_1^3(\Bop\Aop)_2^3(\Aop\Bop)_3^3(\Aop\Bop)_4^{*3}=TC_1^{-1}C_2^{-1}C_3C_4.
\end{multline*}
The  identity $P_{(12)}P_{(23)}P_{(34)}P_{(4321)}=1$  implies that
\begin{multline*}
T=TP_{(12)}P_{(23)}P_{(34)}P_{(4321)}=\bar
TP_{(23)}P_{(34)}P_{(4321)}(\Bop\Aop)_2(\Bop\Aop)_1(\Aop\Bop)_4^*\\=
TP_{(34)}P_{(4321)}\Bop_3(\Bop\Aop)_2(\Aop\Bop\Aop)_1(\Aop\Bop)_4^*\\
=
\bar TP_{(4321)}(\Aop^*\Bop^*\Bop\Aop)_2\Bop_3((\Aop\Bop)^2\Aop)_1(\Aop\Bop\Bop^*\Aop^*)_4\\
=
T(\Bop^*(\Aop\Bop)^2\Aop)_1(\Aop^*\Bop^*\Bop\Aop)_2(\Aop\Bop\Bop^*\Aop^*)_4=
T(\Rop C^{-1})_1(\Rop C^{-1})_2( \Rop^{-1}C)_4
\end{multline*}
This and \eqref{E:c1c2=c3c4} yields~\eqref{E:r1r2t=c3r4t}. The
identity  $P_{(34)}P_{(23)}P_{(12)}P_{(1234)}=1$ implies that
\begin{multline*}
T=TP_{(34)}P_{(23)}P_{(12)}P_{(1234)}= \bar
TP_{(23)}P_{(12)}P_{(1234)}(\Bop\Aop)_1^*(\Aop\Bop)_4(\Aop\Bop)_3\\=
TP_{(12)}P_{(1234)}\Aop_2(\Bop\Aop)_1^*(\Bop\Aop\Bop)_4(\Aop\Bop)_3\\=
\bar TP_{(1234)}(\Bop\Aop\Aop^*\Bop^*)_1
\Aop_2((\Bop\Aop)^2\Bop)_4(\Bop^*\Aop^*\Aop\Bop)_3\\=
T(\Bop\Aop\Aop^*\Bop^*)_1(\Bop^*\Aop^*\Aop\Bop)_3
(\Aop^*(\Bop\Aop)^2\Bop)_4= T(\Lop^{-1}C^{-1})_1(\Lop C)_3(\Lop
C)_4.
\end{multline*}
This and \eqref{E:c1c2=c3c4} yields~\eqref{E:l1c2-1t=l3l4t}.
Finally,
\begin{multline*}
TP_{(12)}P_{(34)}=\bar
TP_{(34)}(\Bop\Aop)_1(\Bop\Aop)_2(\Aop\Bop)_3^*\\=
T(\Bop^*\Aop^*\Bop\Aop)_1(\Bop\Aop)_2(\Bop\Aop\Bop^*\Aop^*)_3(\Bop\Aop)_4.
\end{multline*}  The relation $(P_{(12)}P_{(34)})^2=1$ gives
\begin{multline*}
T=T(P_{(12)}P_{(34)})^2=
TP_{(12)}P_{(34)}(\Bop^*\Aop^*\Bop\Aop)_2(\Bop\Aop)_1(\Bop\Aop\Bop^*\Aop^*)_4(\Bop\Aop)_3\\=
T(\Bop^*\Aop^*(\Bop\Aop)^2)_1(\Bop\Aop\Bop^*\Aop^*\Bop\Aop)_2(\Bop\Aop\Bop^*\Aop^*\Bop\Aop)_3
((\Bop\Aop)^2\Bop^*\Aop^*)_4\\=
T(\Aop^*\Aop)_1(\Bop\Aop\Aop^*\Bop)_2(\Bop\Aop\Aop^*\Bop)_3
(\Bop\Bop^*)_4=T\Lop_1(C^{-1}\Lop^{-1}\Rop)_2(C^{-1}\Lop^{-1}\Rop)_3\Rop_4^{-1}.
\end{multline*} This formula can be rewritten in the
following equivalent form
\[
T(\Rop^{-1}\Lop)_2C_3\Rop_4=T\Lop_1C^{-1}_2(\Lop^{-1}\Rop)_3
\]
which reduces to~\eqref{E:r1l2t=r3l4t} after using
\eqref{E:r1r2t=c3r4t} on the left hand side, and
\eqref{E:l1c2-1t=l3l4t} on the right hand side.
\end{proof}

\begin{lemma} \label{L:[lr]}   $\Qop=\Lop\Rop\Lop^{-1}\Rop^{-1}\in \End(H)$ is
a $T$-unitary $T$-scalar.
\end{lemma}
\begin{proof}
Applying consecutively~\eqref{E:l1c2-1t=l3l4t} and
\eqref{E:r1l2t=r3l4t} in alternating order, we obtain
\begin{multline*}
T\Qop_1=T(\Rop\Lop^{-1}\Rop^{-1})_1C_2\Lop_3\Lop_4=T(\Lop^{-1}\Rop^{-1})_1(\Lop^{-1}C)_2(\Rop\Lop)_3
\Lop_4^2\\
=T\Rop^{-1}_1\Lop^{-1}_2(\Lop^{-1}\Rop\Lop)_3\Lop_4=T\Rop^{-1}_3
(\Lop^{-1}\Rop\Lop)_3=T\Qop^*_3.
\end{multline*}
Similar   transformations   using~\eqref{E:l1c2-1t=l3l4t} and
\eqref{E:r1r2t=c3r4t} yield \( T\Qop_1=T\Qop^*_4  \). Analogously,
using~\eqref{E:r1l2t=r3l4t} and \eqref{E:r1r2t=c3r4t}, we obtain \(
T\Qop_2=T\Qop^*_4. \) This verifies the first three of
Equalities~\eqref{E:q1=q2}. The other three equalities are checked
similarly.

Since $Q$ is a $T$-scalar, so is
$\Qop^{*}=\Rop^{-1}\Lop^{-1}\Rop\Lop$.  We have
\[
\Qop^{-1}=\Rop\Lop\Rop^{-1}\Lop^{-1}=\Lop\Rop \Qop^* \Rop^{-1}\Lop^{-1}  \teq \Qop^* \Lop\Rop \Rop^{-1}\Lop^{-1}  = \Qop^*,
\]
where the    $T$-equality follows from Lemma~\ref{L:xg=x^gg} applied
to the $T$-scalar $Q^*$.
\end{proof}

\begin{remark}  It is clear that  $Q$ is   grading-preserving.   For any $i,j,k,l,m,n\in I$, the restrictions of $Q$ to the
corresponding multiplicity spaces induce the endomorphisms $Q_1,
Q_2, Q_3, Q_4$ of the vector space~\eqref{ambientvectorspace}.
Lemma~\ref{L:[lr]} implies that for any $r\in \{1,2\}$ and  $s\in
\{3,4\}$, the composition $Q_rQ_s$ preserves the
$6j$-symbol~\eqref{positive6j-symbol}.
  \end{remark}


\section{$\spsi$-systems}\label{section7}
\subsection{The operators $\scop{}$, $\srop{}$, and $\slop{}$} Recall the   symmetric,
grading-preserving, invertible operators $  C, R, L \in \End(H)$.
Suppose that we have symmetric, grading-preserving, invertible
operators $\scop{}, \srop{}\in \End(H)$ such that
\begin{subequations}\label{E:srts}
\begin{align}
(\scop{})^2&=C,\quad   \Aop \scop{}\Aop=\Bop \scop{}\Bop= \scop{-},\\
(\srop{})^2&=\Rop,\quad
\Bop\srop{}\Bop=\srop{-},\quad\srop{}\scop{}=\scop{}\srop{}
\end{align}
\end{subequations}
where by definition $ \scop{-} =(\scop{})^{-1}$ and $ \srop{-}
=(\srop{})^{-1}$. Set
\begin{equation}\label{E:sqrtL}
\slop{}=\Bop\Aop \srop{-}  \Aop\Bop \in \End(H) \quad {\text {and}} \quad    \slop{-}= (\slop{})^{-1}=\Bop\Aop \srop{}  \Aop\Bop \in \End(H).
\end{equation}
The properties of $\slop{}$ are summarized in the following lemma.

\begin{lemma}\label{New1204}
The operator $\slop{}$ is symmetric, grading-preserving, and
\begin{equation}\label{lll++}
(\slop{})^2=\Lop, \quad
\Aop\slop{}\Aop=\slop{-},\quad
\slop{}\scop{}=\scop{}\slop{},
\end{equation}
\end{lemma}
\begin{proof}
We have
$$(\slop{})^2=BAR^{-1}AB=BABB^*AB=A^*B^*A^*B^*AB=A^*ABAAB=A^*A=L,$$
$$\Aop\slop{}\Aop=ABA\srop{-}ABA=(AB)^2\srop{}(BA)^2$$
$$=BAC\srop{}C^{-1}AB=BA \srop{} AB =\slop{-},$$
$$(\slop{})^*=B^*A^*\srop{-}A^*B^*=B^*A^*B^*\srop{}B^*A^*B^*=ABA\srop{}ABA
=A\slop{-}A=\slop{}.$$ That $\slop{}$ is   grading-preserving and commutes with $\scop{}$ follows from the definitions.
\end{proof}

\subsection{$\spsi$-systems  in $\cat$}
A \emph{$\spsi$-system} in $\cat$ is a $\Psi$-system in $\cat$
together with a choice of invertible, symmetric, grading-preserving
operators
  $\scop{}, \srop{} \in \End(H)$ satisfying
Equalities~\eqref{E:srts} as well as the identities
\begin{subequations}\label{E:sqrtCC}
\begin{align}\label{E:sqrtC1C2}
T\scop{}_1\scop{}_2&=T\scop{}_3\scop{}_4,\\
\label{E:sqrtR1L2}
T\srop{}_1\slop{}_2&=T\srop{}_3\slop{}_4,\\
\label{E:sqrtR1R2}
T\srop{}_1\srop{}_2&=T\scop{}_3\srop{}_4,\\
\label{E:sqrtL1}
T\slop{}_1&=T\scop{}_2\slop{}_3\slop{}_4.
\end{align}
\end{subequations}
where $\slop{}$ is defined by~\eqref{E:sqrtL}.

Generally speaking, a $\Psi$-system may not allow operators
$\scop{}, \srop{} \in
 \End(H)$ as above.

Equations~\eqref{E:sqrtR1R2} and \eqref{E:sqrtL1} above are not
independent. In fact, any one of them may be omitted.

 We suppose from now on
that we do have a $\spsi$-system.

\subsection{Commutation relations} We establish   commutation
relation analogous to those in Lemma~\ref{L:xg=x^gg}. Consider the
group  $$ \widehat {F}=\langle a,b,a^*,b^*,c,r\,\vert \,
   a^2=b^2=(a^*)^2=(b^*)^2=1\rangle.
$$
Consider the group homomorphism $ \widehat {F} \to
\mathbb{Z}/2\mathbb{Z}$ carrying   $a, b, a^*, b^*$ to $1\, ({\text
{mod}} \, 2)$ and carrying $c,r$ to $0$. Elements of $\widehat {F}$
belonging to the kernel of this homomorphism are said to be {\it
even}; all other elements of $\widehat {F}$ are said to be {\it odd}.
  The group $\widehat {F}$ acts
   on $ H $ by   $$a\mapsto A,\, b \mapsto B ,\, a^* \mapsto A^* ,\, b^* \mapsto
   B^*,\, c\mapsto \scop{},\, r\mapsto \srop{}.$$ The endomorphism of
   $H$ determined by $g\in \widehat {F}$ is denoted
   ${\underline{g}}$.

\begin{lemma}\label{lelele} Let $t\in \End (H)$ be a $T$-scalar
that $T$-commutes with $\scop{}$ in the sense that
$t\scop{}\teq\scop{} t$. For $g\in \widehat F$, set
  $\act{t}{g}=t$  if $g$ is  even    and  $\act{t}{g}=t^*$   if $g$ is
  odd. Then $
   {\underline{g}}t\teq  \act{t}{g} {\underline{g}} $ for all $g\in \widehat
   {F}$.
\end{lemma}
\begin{proof}   Observe   that $t$ also $T$-commutes with
$\scop{-}$ and   $t^*$ $T$-commutes with both $\scop{}$ and
$\scop{-}$.
 Indeed, let $\{i,j\}$ be  the set $\{1,2\}$ or the set
$\{3,4\}$. Set $\{k,l\}=\{1,2,3,4\}\setminus\{i,j\}$. Formula
\eqref{E:sqrtC1C2} implies that
\begin{multline*}
T(t\scop{-})_i=Tt_i\scop{-}_i=Tt_j\scop{-}_i=T\scop{-}_it_j=T\scop{}_j\scop{-}_k\scop{-}_lt_j\\=
Tt_j\scop{}_j\scop{-}_k\scop{-}_l =Tt_i\scop{}_j\scop{-}_k\scop{-}_l
=T\scop{}_j\scop{-}_k\scop{-}_l t_i= T(\scop{-}t)_i,
\end{multline*}
\begin{multline*}
T(t^*\scop{})_i=Tt^*_i\scop{}_i=Tt_k\scop{}_i=T\scop{}_it_k
=T\scop{-}_j\scop{}_k\scop{}_lt_k\\ =
Tt_k\scop{-}_j\scop{}_k\scop{}_l
=Tt^*_i\scop{-}_j\scop{}_k\scop{}_l=T\scop{-}_j\scop{}_k\scop{}_l
t^*_i= T(\scop{}t^*)_i.
\end{multline*}

 For $i\in\{1,2,3,4\}$, set  \[\widehat  {\Delta}_i=\{g\in\widehat
{F}\vert\ T(\underline{g} t)_i=T(t^g \underline{g})_i\} \subset
\widehat {F}.\]  By Lemma~\ref{L:1234}, we have $\widehat
{\Delta}_1=\widehat {\Delta}_2$ and $\widehat {\Delta}_3=\widehat
{\Delta}_4$. Set $\widehat\Delta=\widehat {\Delta}_1 \cap \widehat
{\Delta}_3 \subset \widehat {F}$. Clearly, $1\in\widehat \Delta $.
It is enough to show that $ \widehat {\Delta} =\widehat {F}$.

 Pick any index $i\in \{1,2,3,4\}$ and let
$j,k,l\in \{1,2,3,4\}$ be such  that either $\{i,j\}=\{1,2\}$ or
$\{i,j\}=\{3,4\}$, and $\{k,l\}=\{1,2,3,4\}\setminus\{i,j\}$. For
  $X=\underline c=\scop{}$ and any $g\in \widehat \Delta$,
Formula~\eqref{E:sqrtC1C2} implies that
\[T(  \underline{cg} t)_i=
T(X \underline{g} t)_i=T(\underline{g} t)_iX_j^{-1}X_kX_l=T(t^g
\underline{g})_iX_j^{-1}X_kX_l=Tt^g_j X_j^{-1}X_kX_l \underline{g}_i
\] (the last two equalities follow  from the assumptions $g\in
\widehat \Delta$ and   $t$ is a $T$-scalar, respectively). Since
$t$, $t^*$ both $T$-commute with $\scop{\pm}$,  we similarly have
\[ Tt^g_j X_j^{-1}X_kX_l \underline{g}_i =T X_j^{-1}X_kX_l
\underline{g}_i t^g_j=T(X \underline {g})_it^g_j\]\[ =T t^g_j(X
\underline {g})_i =T(t^gX \underline{g})_i=T(t^{cg}
\underline{cg})_i. \]
  Thus, $T(  \underline{cg} t)_i=T(t^{cg}
\underline{cg})_i$ for all $i$ so that $cg\in \widehat \Delta$. This
shows the inclusion  $c\widehat \Delta\subset\widehat \Delta$. The
inclusion $c^{-1}\widehat \Delta\subset\widehat \Delta$ is proved
similarly.

Next, pick indices $i,j,k$ such that $\{i,j,k\}=\{1,2,4\}$. Set
  $X=\underline r=\srop{}$ and $Y=\scop{}$. For  any $g\in \widehat \Delta$, Formula~\eqref{E:sqrtR1R2} implies
   that for some $\varepsilon
=\pm 1$ and  $X',X''\in\{X,X^{-1}\}$,
\[T(  \underline{rg} t)_i=
T X_i ( \underline{g} t)_i=TY_3^{\varepsilon}X'_jX''_k(\underline{g}
t)_i= T(t^g \underline{g})_iY_3^{\varepsilon}X'_jX''_k=Tt'_3
\underline{g}_iY_3^{\varepsilon}X'_jX''_k,\] where  $t'\in \{t,
t^*\}$ is such that $Tt^g_i =Tt'_3$. Since $t'Y\teq Y t'$,
\[Tt'_3
\underline{g}_iY_3^{\varepsilon}X'_jX''_k= T
Y_3^{\varepsilon}X'_jX''_k \underline{g}_i t'_3=T(X\underline{g})_it'_3=T t'_3 (X\underline{g})_i =T(t^gX\underline{g})_i=T(t^{rg}
\underline{rg})_i.\]
 Thus, $rg\in \widehat \Delta$ and so $r\widehat \Delta\subset\widehat
\Delta$. Similarly, $r^{-1} \widehat \Delta\subset\widehat \Delta$.
The rest of the proof is as in  Lemma~\ref{L:xg=x^gg}.
\end{proof}

\subsection{Stable $T$-equivalence} We say that two  operators $a,b\in\End(H)$ are \emph{stably $T$-equal}
and write $a\steq b$ if   $\underline{f}a \teq \underline{f}b$ for
all $f\in\widehat F$. Obviously, if $a\steq b$, then $ a\teq b$ and
$ \underline{g}a\steq\underline{g}b$ for all $g\in\widehat F$.

\begin{corollary}\label{nnn17---} For   any $g\in\widehat F$ and   any $T$-scalar $t$ that $T$-commutes with $\scop{}$,
\begin{equation}\label{E:gtTstg}
\underline{g}\, t\steq t^{g}\underline{g}.
\end{equation}
\end{corollary}

Indeed, by Lemma~\ref{lelele}, for all $f\in\widehat F$,
\[
\underline{f}\, \underline{g}\, t =\underline{fg} \, t\teq t^{fg}\, \underline{f g}=  (t^{g})^{f}\, \underline{f}\, \underline{g}
\teq\underline{f}t^{g}\underline{g}.
\]

Corollary~\ref{nnn17---} and the evenness of the elements $r$ and
$ba r^{-1}ab$  of $ \widehat F$ imply   that if a $T$-scalar $t$
commutes with $\scop{}$, then $t$ stably $T$-commutes with $\srop{}$
and $\slop{}$ in the sense that $\srop{}t\steq t\srop{}$ and
$\slop{}t\steq t\slop{}$.

\subsection{The $T$-scalar $q$} The following
$T$-scalar   will play a key role in the sequel.

\begin{lemma}\label{nnn17}
The grading-preserving operator $$\sqop{}{8}=\srop{} A \srop{-} A
\slop{-}\scop{-}=\srop{}B\slop{} B \slop{-}\scop{-}$$ is a unitary
$T$-scalar commuting with $\scop{}$.
\end{lemma}
\begin{proof} That $\sqop{}{8}$ is grading-preserving is obvious because $\srop{}$, $\scop{}$, and $\slop{}$
are grading-preserving and $A$ is involutive. Since $\srop{},
\slop{}$, and $\Aop\srop{-}\Aop$ commute with $\scop{}$, the
operator $\sqop{}{8}$ also commutes with $\scop{}$. To prove the
remaining claims, set
\begin{equation}\label{E:d}
\Dop=\Aop\srop{-}\Aop =\Bop\slop{}\Bop.
\end{equation}
It is clear that $\Dop\scop{}=\scop{}\Dop$. Note also that
\begin{subequations}\label{E:d-prop}
\begin{equation}\label{E:d*}
\Dop^*=\Lop\Dop\Lop^{-1} ,
\end{equation}
\begin{equation}\label{E:d2}
\Dop^2=\Rop^{-1}\Lop\Cop=CR^{-1}L.
\end{equation}
\end{subequations}
Indeed,
$$\Dop^*=\Aop^*\srop{-}\Aop^*=\Aop^*\Aop\Bop\slop{}\Bop\Aop \Aop^*=L\Bop
\slop{} \Bop L^{-1}=\Lop\Dop\Lop^{-1}$$ and
$$\Dop^2=BLB=BA^*AB=BB^*A^* A^*B^*A^*AB$$
$$=BB^*A^*AABABAB=R^{-1}LC.$$

We claim that the operator $\Dop$ satisfies the following identities
\begin{subequations}\label{E:dop}
\begin{align}
\label{E:dop1}
T\Dop_1&=T\srop{}_2\slop{}_3\Dop_4^*,\\
\label{E:dop2}
T\scop{}_1\Dop_2&=T\srop{}_3\Dop_4^*,\\
\label{E:dop3}
T\slop{}_2\Dop_3\scop{}_4&=T\Dop_1^*.
\end{align}
\end{subequations}
Formula \eqref{E:dop1} is proved as follows:
\begin{multline*}
T\Dop_1
=TP_{(134)}\Aop_4^*(\Bop\Aop)_2(\srop{-}\Aop)_1=
T\slop{}_4\srop{-}_1\slop{-}_2P_{(134)}\Aop_4^*(\Bop\Aop)_2\Aop_1\\=
TP_{(134)}\slop{}_3(\srop{-}\Aop^*)_4(\slop{-}\Bop\Aop)_2\Aop_1=
T\srop{}_2\slop{}_3\Dop_4^*,
\end{multline*}
where in the first, second, etc. equalities we use respectively: the
definition of $\Dop$ and  \eqref{E:tp431+}; the action of the
permutation group  $ \mathbb{S}_4$ and \eqref{E:sqrtR1L2}; the
action of $ \mathbb{S}_4$;  the definitions of $\slop{}$, $\Dop$ and
\eqref{E:tp431+}. The proof of   \eqref{E:dop2} is similar:
\begin{multline*}
T\Dop_2
=TP_{(124)}\Bop_4^*(\Bop\Aop)_3(\slop{}\Bop)_2=
T\slop{}_1\scop{-}_2\slop{-}_3P_{(124)}\Bop_4^*(\Bop\Aop)_3\Bop_2\\=
TP_{(124)}\slop{}_4\scop{-}_1\slop{-}_3\Bop^*_4(\Bop\Aop)_3\Bop_2=
T\scop{-}_1\srop{}_3\Dop_4^*,
\end{multline*}
where we use consecutively: the definition of $\Dop$ and
\eqref{E:tp124}; the action of  $ \mathbb{S}_4$ and
\eqref{E:sqrtL1}; the action of  $ \mathbb{S}_4$;  the definitions
of $\slop{}$, $\Dop$ and \eqref{E:tp124}. Finally, we prove
\eqref{E:dop3}:
\begin{multline*}
T\Dop_3 =TP_{(431)}\Aop_1^*(\Aop\Bop)_2(\srop{-}\Aop)_3=
T\srop{}_2\scop{-}_3\srop{-}_4P_{(431)}\Aop_1^*(\Aop\Bop)_2\Aop_3\\=
TP_{(431)}\srop{}_2\scop{-}_4\srop{-}_1\Aop_1^*(\Aop\Bop)_2\Aop_3=
T\Dop_1^*\slop{-}_2\scop{-}_4.
\end{multline*}
Here we use: the definition of $\Dop$ and \eqref{E:tp431}; the
action of  $ \mathbb{S}_4$ and \eqref{E:sqrtR1R2}; the action of  $
\mathbb{S}_4$;  the definitions of $\slop{}$, $\Dop$ and
\eqref{E:tp431}.

The definition of $q$ may be rewritten as
\begin{equation}\label{defq}
\sqop{}{8}=\srop{}\Dop\slop{-}\scop{-}.
\end{equation} We can now prove the
unitarity of $q$:
\[
\sqop{*}{8}=\scop{-}\slop{-}\Dop^*\srop{}=
\scop{-}\slop{}\Dop\Lop^{-1}\srop{}\]
\[
=\scop{-}\slop{}\Dop^{-1}\Dop^2\Lop^{-1}\srop{}=\scop{-}\slop{}\Dop^{-1}C\srop{-}=\scop{}\slop{}\Dop^{-1} \srop{-}= \sqop{-1}{8},
\]
where we use the relations \eqref{E:d-prop}. Now,
\begin{multline*}
T\sqop{}{8}_1=T\srop{-}_2\scop{}_3\srop{}_4(\Dop\slop{-}\scop{-})_1=
T\slop{}_3\Dop_4^*\scop{}_3\srop{}_4(\slop{-}\scop{-})_1\\
=T\scop{-}_2\slop{-}_4\Dop_4^*\scop{}_3\srop{}_4\scop{-}_1=
T\scop{-}_1\scop{-}_2\scop{}_3(\sqop{}{8}\scop{})_4^*=T\sqop{*}{8}_4,
\end{multline*}
where we use consecutively: Formulas~\eqref{defq} and
\eqref{E:sqrtR1R2}; Formula \eqref{E:dop1}; Formula
\eqref{E:sqrtL1}; Formula~\eqref{defq}; Formula \eqref{E:sqrtC1C2}.
 Similarly,
\begin{multline*}
T\sqop{}{8}_2=T\srop{-}_1\scop{}_3\srop{}_4(\Dop\slop{-}\scop{-})_2=
T\scop{-}_1\srop{}_3\Dop_4^*\srop{-}_1\scop{}_3\srop{}_4(\slop{-}\scop{-})_2\\
=T\scop{-}_1\slop{-}_4\Dop_4^*\scop{}_3\srop{}_4\scop{-}_2=
T\scop{-}_1\scop{-}_2\scop{}_3(\sqop{}{8}\scop{})_4^*=T\sqop{*}{8}_4,
\end{multline*}
where we use consecutively: Formulas~\eqref{defq} and
\eqref{E:sqrtR1R2}; Formula \eqref{E:dop2}; Formula
\eqref{E:sqrtR1L2}; Formula~\eqref{defq}; Formula
\eqref{E:sqrtC1C2}. Similarly,
\begin{multline*}
T\sqop{}{8}_3=T\srop{}_1\slop{}_2\slop{-}_4(\Dop\slop{-}\scop{-})_3=
T\Dop_1^*\scop{-}_4\srop{}_1\slop{-}_4(\slop{-}\scop{-})_3\\
=T\slop{-}_1\scop{}_2\Dop_1^*\scop{-}_4\srop{}_1\scop{-}_3=
T(\sqop{}{8}\scop{})_1^*\scop{}_2\scop{-}_3\scop{-}_4=T\sqop{*}{8}_1,
\end{multline*}
where we use: Formulas~\eqref{defq} and \eqref{E:sqrtR1L2};
Formula~\eqref{E:dop3}; Formula~\eqref{E:sqrtL1};
Formula~\eqref{defq}; Formula \eqref{E:sqrtC1C2}. Together with the
unitarity of $\sqop{}{8}$ these identities imply that $\sqop{}{8}$
is a $T$-scalar.
\end{proof}

\begin{lemma}
For all $a,b\in\frac12\mathbb{Z}$,
\begin{equation}\label{E:larb=q8abrbla}
\Lop^{a}\Rop^{b}\steq\sqop{8ab}{}\Rop^{b}\Lop^{a}.
\end{equation}
 For all
$a,b\in\frac12\mathbb{Z}$ and $c\in\mathbb{Z}$,
\begin{equation}\label{E:lr}
(\Lop^{a}\Rop^{b})^c\steq\sqop{4abc(1-c)}{}\Lop^{ac}\Rop^{bc} .
\end{equation}
\end{lemma}
\begin{proof}
Formulas~\eqref{defq} and~\eqref{E:d2} imply that
\[\sqop{}{}\slop{}\srop{-}\sqop{}{}=
\srop{-}\slop{}.
\]
  This   and  ~\eqref{E:gtTstg} yield~\eqref{E:larb=q8abrbla}
for $a=b=\frac12$:
\[
\slop{}\srop{}=\srop{}\srop{-}\slop{}\srop{}=\srop{}\sqop{}{}\slop{}\srop{-}\sqop{}{}\srop{}
\steq \sqop{2}{}\srop{}\slop{}\srop{-}\srop{}=\sqop{2}{}\srop{}\slop{}.
\]
Assuming   \eqref{E:larb=q8abrbla}   for some $a\in\frac12\mathbb{Z}
$ and $b=\frac12$, we obtain
\begin{multline*}
\Lop^{a+\frac12}\srop{}=\Lop^{a}\slop{}\srop{}\steq\Lop^{a}\sqop{2}{}\srop{}\slop{}\\
\steq\sqop{2}{}\Lop^{a}\srop{}\slop{}\steq\sqop{2}{}\sqop{4a}{}\srop{}\Lop^{a}\slop{}=
\sqop{4(a+\frac12)}{}\srop{}\Lop^{a+\frac12}.
\end{multline*}
This proves \eqref{E:larb=q8abrbla} for all positive
$a\in\frac12\mathbb{Z} $ and $b=\frac12$. Similarly, assuming
  \eqref{E:larb=q8abrbla}  for some
positive $a,b\in\frac12\mathbb{Z} $, we obtain
\begin{multline*}
\Lop^{a}\Rop^{b+\frac12}=\Lop^{a}\srop{}\Rop^{b}\steq\sqop{4a}{}\srop{}\Lop^{a}\Rop^{b}
\steq\sqop{4a}{}\srop{}\sqop{8ab}{}\Rop^{b}\Lop^{a}\\
\steq\sqop{4a}{}\sqop{8ab}{}\srop{}\Rop^{b}\Lop^{a}=\sqop{8a(b+\frac12)}{}\Rop^{b+\frac12}\Lop^{a}.
\end{multline*}
This proves \eqref{E:larb=q8abrbla} for all positive
 $a,b\in\frac12\mathbb{Z} $. Moreover, for such $a,b$,
\[
\Lop^{-a}\Rop^{b}=\Lop^{-a}\Rop^{b}\Lop^{a}\Lop^{-a}
\steq\Lop^{-a}\sqop{-8ab}{}\Lop^{a}\Rop^{b}\Lop^{-a} \steq \sqop{-8ab}{} \Lop^{-a}\Lop^{a}\Rop^{b}\Lop^{-a}
=\sqop{-8ab}{}\Rop^{b}\Lop^{-a},
\]
\[
\Lop^{a}\Rop^{-b}=\Rop^{-b}\Rop^{b}\Lop^{a}\Rop^{-b}
\steq\Rop^{-b}\sqop{-8ab}{}\Lop^{a}\Rop^{b}\Rop^{-b}
\steq \sqop{-8ab}{} \Rop^{-b}\Lop^{a}\Rop^{b}\Rop^{-b}= \sqop{-8ab}{}\Rop^{-b}\Lop^{a},
\]
and
\[
\Lop^{-a}\Rop^{-b}=\Rop^{-b}\Rop^{b}\Lop^{-a}\Rop^{-b}
\steq\Rop^{-b}\sqop{8ab}{}\Lop^{-a}\Rop^{b}\Rop^{-b}
\steq   \sqop{8ab}{}\Rop^{-b}\Lop^{-a},
\]
This proves~\eqref{E:larb=q8abrbla} for all non-zero $a$, $b$. For
$a=0$ or $b=0$, Formula~\eqref{E:larb=q8abrbla} is obvious.

 The case $c=0$ of~\eqref{E:lr} is obvious. Assuming
\eqref{E:lr}  for some $c\in\mathbb{Z}$, we obtain
\begin{multline*}
(\Lop^{a}\Rop^{b})^{c+1}=(\Lop^{a}\Rop^{b})^{c}\Lop^{a}\Rop^{b}
\steq\sqop{4abc(1-c)}{}\Lop^{ac}\Rop^{bc}\Lop^{a}\Rop^{b}
\steq\sqop{4abc(1-c)}{}\Lop^{ac}\sqop{-8abc}{}\Lop^{a}\Rop^{bc}\Rop^{b}\\
\steq\sqop{-4abc(c+1)}{}\Lop^{ac}{}\Lop^{a}\Rop^{bc}\Rop^{b}
=\sqop{4ab(c+1)(1-(1+c))}{}\Lop^{a(c+1)}\Rop^{b(c+1)}.
\end{multline*}
This implies \eqref{E:lr} for all $c\geq 0$. The proof for negative
$c $ is similar.
\end{proof}

\begin{remark}\label{R:q8Q} Applying Formula~\ref{E:larb=q8abrbla} to $a=b=1$, we obtain  $\sqop{8}{} \steq \Qop $, where $\Qop=\Lop{}\Rop{}
\Lop^{-1}\Rop^{-1} $ is the operator studied
 in Lemma~\ref{L:[lr]}. \end{remark}

\section{Charged $T$-forms}\label{section8now}

\subsection{Definition}
For any $a , c\in\frac12\Z$, we define the ``charged" $T$-forms
\[
T(a,c)=T\sqop{4ac}{}_1\Rop^{c}_1\Rop^{-a}_2\Lop^{-a}_3\Rop^{-c}_3\colon
H^{\otimes 4}\to \FK\] and \[ \bar T(a,c)=\bar
T\sqop{-4ac}{}_1\Lop^{-a}_2\Rop^{-c}_2\Rop^{-a}_3\Rop^{c}_4 \colon
H^{\otimes 4}\to \FK.
\]
\begin{lemma}\label{le789} Set $ \Aops=\Aop\slop{-}, \Bops=\Bop\srop{-} \in \End(H)$.
Then $\Aops$, $\Bops$ are symmetric involutions and for any
$a,b,c\in\frac12\Z$ such that $a+b+c=\frac12$,
\begin{subequations}\label{E:syms}
\begin{align}
\label{E:syms01}
 T(a,c)P_{(4321)}&=\bar T(a,b)\sqop{-2a}{}_1\Aops_1\Aops_3,\\
\label{E:syms12}
T(a,c)P_{(23)}&=\bar T(b,c)\sqop{2c}{}_1\Aops_2\Bops_3,\\
\label{E:syms23}
T(a,c)P_{(1234)} &=\bar T(a,b)\sqop{-2a}{}_1\Bops_2\Bops_4.
\end{align}
\end{subequations}

\end{lemma}
\begin{proof} The equalities $\Aops^2=\Bops^2=1$  directly follow  from
the the involutivity of the  operators $\Aop$, $\Bop$ and the
formulas $\Aop\slop{}\Aop=\slop{-}$ and $\Bop\srop{}\Bop=\srop{-}$.
We have
$$\Aops^*=(\slop{-})^*A^*=\slop{-}A^*=\slop{-}AAA^*=A\slop{} AA^*=A\slop{}L^{-1}=A\slop{-}=\Aops.$$
Similarly, $\Bops^*=\Bops$. To prove~\eqref{E:syms}, observe that
 \begin{multline*}
\Aop\Lop^{-b}\Rop^{-c}\Aops=\Aops\Lop^{\frac12-b}\Rop^{-c}\Aops=\Lop^{b-\frac12}\Aops\Rop^{-c}\Aops=
\Lop^{b-\frac12}\left(\Aops\srop{-}\Aops\right)^{2c}\\
=\Lop^{b-\frac12}\left(\slop{}\srop{-}\scop{}\sqop{}{}\right)^{2c}
\steq\sqop{2c}{}\Cop^{c}\Lop^{b-\frac12}\left(\slop{}\srop{-}\right)^{2c}\\
\steq\sqop{2c}{}\Cop^{c}\Lop^{b-\frac12}\sqop{-2c(1-2c)}{}\Lop^{c}\Rop^{-c}
\steq\sqop{4c^2}{}\Cop^{c}\Lop^{b+c-\frac12}\Rop^{-c}.
\end{multline*}
Here the first, second, etc. equalities follow respectively from:
the definition of~$\Aops$; the formula $\Aops
\slop{}=\slop{-}\Aops$; the involutivity of $\Aops$; the definition
of $\sqop{}{}$; the fact that $\srop{}$ and $\slop{}$ commute with
$\scop{}$ and stably $T$-commute with $\sqop{}{}$; Formula
\eqref{E:lr}; the fact that $\slop{}$ stably $T$-commutes with
$\sqop{}{}$. A similar computation shows that
\begin{equation}\label{dddf} B\Lop^{-a}\Rop^{-b}\Bops\steq\sqop{4a^2-4a}{}
\Cop^{-a}\Lop^{-a}\Rop^{a+b-\frac12}.
\end{equation}

Formula~\eqref{E:syms12} is equivalent to the formula $\bar
T(b,c)\sqop{2c}{}_1\Aops_2\Bops_3 P_{(23)}= T(a,c) $ which we now
prove:
\begin{multline*}
\bar T(b,c)\sqop{2c}{}_1\Aops_2\Bops_3P_{(23)}=  \bar T
\sqop{2c(1-2b)}{}_1\Lop^{-b}_2\Rop^{-c}_2\Rop^{-b}_3\Rop^{c}_4\Aops_2\Bops_3P_{(23)}
\\ = TP_{(23)}A_2B_3
\sqop{2c(1-2b)}{}_1\Lop^{-b}_2\Rop^{-c}_2\Rop^{-b}_3\Rop^{c}_4\Aops_2\Bops_3P_{(23)}\\
=T\sqop{2c(1-2b)}{}_1(B\Rop^{-b}\Bops)_2(\Aop\Lop^{-b}\Rop^{-c}\Aops)_3\Rop^{c}_4\\
=T\sqop{2c(1-2b)}{}_1 \sqop{4c^2}{}_3
\Rop^{b-\frac12}_2(\Cop^c\Lop^{-a}\Rop^{-c}
)_3\Rop^{c}_4\\
=T\sqop{4ac}{}_1\Rop^{b-\frac12}_2(\Cop^c\Lop^{-a}\Rop^{-c}
)_3\Rop^{c}_4\\
=T(\Rop^c \sqop{4ac}{})_1\Rop^{-a}_2(\Lop^{-a}\Rop^{-c}
)_3=T(\sqop{4ac}{}\Rop^c)_1\Rop^{-a}_2(\Lop^{-a}\Rop^{-c}
)_3=T(a,c).
\end{multline*}
Here we use consecutively: the definition of $\bar T(b,c)$;
Formula~\eqref{E:sym12};  the action of~$ \mathbb{S}_4$; the
equalities $B\Rop^{-b}\Bops=\Rop^{b-\frac12}$  and
$\Aop\Lop^{-b}\Rop^{-c}\Aops\teq
 \sqop{4c^2}{}\Cop^{c}\Lop^{-a}\Rop^{-c}$
established above; the equality $Tq_3 =Tq_1^{-1}$;
Formula~\eqref{E:sqrtR1R2}; the fact that $\Rop $ stably
$T$-commutes (and therefore $T$-commutes)  with $\sqop{}{}$;
  the definition of $T(a,c)$.

The proof of \eqref{E:syms23} is similar:
\begin{multline*}
\bar T(a,b)\sqop{-2a}{}_1\Bops_2\Bops_4P_{(4321)}
=TP_{(1234)}B_2B_4^*\sqop{-2a(2b+1)}{}_1\Lop^{-a}_2\Rop^{-b}_2
\Rop^{-a}_3\Rop^{b}_4\Bops_2\Bops_4P_{(4321)}\\
=T(B^*\Rop^{b}\Bops)_1\sqop{-2a(2b+1)}{}_2(B\Lop^{-a}\Rop^{-b}\Bops)_3\Rop^{-a}_4
=T\Rop^{\frac12-b}_1\sqop{-2a(2b+1)}{}_2(B\Lop^{-a}\Rop^{-b}\Bops)_3\Rop^{-a}_4 \\
=T\Rop^{\frac12-b}_1\sqop{4ac}{}_2(\Cop^{-a}\Lop^{-a}\Rop^{-c})_3
\Rop^{-a}_4=T\Rop^{c}_1(\Rop^{-a}
\sqop{4ac}{})_2(\Lop^{-a}\Rop^{-c})_3=T(a,c).
\end{multline*}
Here we use consecutively: the definition of $\bar T(a,b)$ and
Formula~\eqref{E:sym23};  the action of~$ \mathbb{S}_4$; Formula
$B^*\Rop^{b}\Bops=\Rop^{\frac12-b}$ which follows from the
definitions of $\Rop$, $\Bops$ and the equality
 $\Bop\srop{}\Bop=\srop{-}$; Formula~\eqref{dddf} and the
equality $Tq_3 =Tq_2^{-1}$;
  Formula~\eqref{E:sqrtR1R2}; the
  formulas
  $\Rop \sqop{}{}\teq \sqop{}{} \Rop$,   $Tq_2 =Tq_1$
 and
  the definition of $T(a,c)$.

Finally, we prove \eqref{E:syms01}:
\begin{multline*}
\bar T(a,b)\sqop{-2a}{}_1\Aops_1\Aops_3P_{(1234)}
=
TP_{(4321)}A_1^*A_3\sqop{-2a(2b+1)}{}_1\Lop^{-a}_2\Rop^{-b}_2\Rop^{-a}_3\Rop^{b}_4
\Aops_1\Aops_3P_{(1234)}
\\
=T(\Lop^{-a}\Rop^{-b})_1(A\Rop^{-a}\Aops)_2\Rop^{b}_3
(A^*\sqop{-2a(2b+1)}{}\Aops)_4\\
=T(\Lop^{-a}\Rop^{-b})_1(\sqop{4a^2}{}\Cop^a\Lop^{a-\frac12}\Rop^{-a})_2
\Rop^{b}_3
(\sqop{2a(2b+1)}{}\slop{})_4\\
=T\Rop^{-b}_1(\sqop{-4ac}{}\Lop^{a-\frac12}\Rop^{-a})_2
(\Rop^{b}\Lop^{-a})_3
\Lop^{\frac12-a}_4\\
=T(\sqop{-4ac}{}\Rop^{c})_1\Rop^{-a}_2
(\Rop^{-c}\Lop^{-a})_3=T(a,c),
\end{multline*}
where we use (among others) Formulas~\eqref{E:sym01},
\eqref{E:sqrtL1}, \eqref{E:sqrtR1L2}, and~\eqref{E:larb=q8abrbla}.
\end{proof}

\begin{remark}
Though we shall not need it in the sequel, note  that the
involutions $\Aops, \Bops\colon H\to H$ introduced in
Lemma~\ref{le789}   satisfy the
  relations  $ (\Bops\Aops)^3=\sqop{2}{4}$ and
\[
 (\Aops \scop{})^2=(\Bops \scop{})^2=(\Aops \sqop{}{8})^2=(\Bops \sqop{}{8})^2= (\Aops \slop{})^2=
(\Bops \srop{})^2=1.
\]
  \end{remark}

 \subsection{The charged Pentagon and Inversion identities} For any $a , c\in\frac12\Z$, we define the ``charged"
  operators ${\tau}(a,c),\bar {\tau}(a,c):\check H^{\otimes 2}\to \check H^{\otimes 2}$  by
 \[
 {\tau}(a,c)=\sqop{-4ac}{}_1\Rop^{c}_1\Rop^{-a}_2{\tau}\Lop^{-a}_1\Rop^{-c}_1 \quad {\text {and}}\quad
\bar {\tau}(a,c)=\sqop{4ac}{}_1\Rop^{-c}_2\Lop^{-a}_2\bar {\tau}\Rop^{-a}_1\Rop^{c}_2,
 \]
 where $\tau$ and $\bar \tau$ are the endomorphisms of $\check H^{\otimes
 2}$ introduced in Section~\ref{section33now}. The operator
 ${\tau}(a,c)$ is
adjoint to $T(a,c)$: for any $u,v\in\hat H,  x,y\in \check H$,
 \[
\langle \langle  u\otimes v,\tau (a,c)(x\otimes y)\rangle  \rangle=
\langle \langle  u\otimes v, (q^{-4ac} \Rop^{c}\otimes \Rop^{-a}) \tau (\Lop^{-a}\Rop^{-c}(x)\otimes y)
\rangle \rangle\]
\[=
\langle \langle  q^{4ac} \Rop^{c}(u) \otimes \Rop^{- a }(v),   \tau (\Lop^{-a}\Rop^{-c}(x)\otimes y)
\rangle \rangle \] \[=T(q^{4ac} \Rop^{c}(u) \otimes \Rop^{-a}(v) \otimes \Lop^{-a}\Rop^{-c}(x)\otimes y) = T(a,c)(u\otimes v\otimes x\otimes y),
 \]
 where we use  the pairing $\langle\langle \cdot,
 \cdot \rangle\rangle$ introduced in Section~\ref{section33now}, the unitarity of $q$, and the symmetry of
 $\slop{}$, $\srop{}$. Similarly,
 ${\bar\tau}(a,c)$ is
adjoint to $\bar T(a,c)$.

 \begin{lemma}
 The charged  operators $\tau$, $\bar \tau$ satisfy the following
 identities.
 \begin{itemize}
 \item[(i)] The charged pentagon identity:
  \[{\tau}_{23}(a_0,c_0) \, {\tau}_{13}(a_2,c_2) \, {\tau}_{12}(a_4,c_4)=
 {\tau}_{12}(a_3,c_3)\, {\tau}_{23}(a_1,c_1)  \, ({}^\bullet \pi)_{21}
 \]
 for any   $a_0,a_1, a_2,a_3, a_4, c_0, c_1, c_2, c_3, c_4
 \in\frac12\Z$  such that
\begin{equation}\label{pentagonconditions}
 a_1=a_0+a_2,\ a_3=a_2+a_4,\ c_1=c_0+a_4,\ c_3=a_0+c_4,\ c_2= c_1+c_3.
\end{equation}
\item[(ii)] The charged inversion identities: for any    $a, c
 \in\frac12\Z$,
  \[
 {\tau}_{21}(a,c)\, \bar {\tau}(-a,-c)=\pi^\bullet  \quad  {\text {and}} \quad \bar {\tau}(-a,-c) \, {\tau}_{21}(a,c)={}^\bullet \pi.
 \]
\end{itemize}
 \end{lemma}
 \begin{proof} We first rewrite Formulas~\eqref{E:sqrtCC} in terms
 of   $\tau$ and $\bar \tau$:
\[
\scop{}_1\scop{}_2\tau  =\tau\scop{}_1\scop{}_2,\,\,
\srop{}_1\slop{}_2\tau  =\tau\srop{}_1\slop{}_2,\,\,
\srop{}_1\srop{}_2\tau  =\tau \scop{}_1\srop{}_2,\,\,
\slop{}_1\scop{-}_2\tau  =\tau \slop{}_1\slop{}_2.
\]
In the following computations, the underlined expressions are
transformed via one of these four equalities. We have
 \begin{multline*}\sqop{4(a_1c_1+a_3c_3)}{}_1{\tau}_{12}(a_3,c_3)\, {\tau}_{23}(a_1,c_1) \, ({}^\bullet \pi)_{21}\\
 =
\Rop^{c_3}_1\Rop^{-a_3}_2{\tau}_{12}\Lop^{-a_3}_1\Rop^{-c_3}_1
\Rop^{c_1}_2\Rop^{-a_1}_3{\tau}_{23}\Lop^{-a_1}_2\Rop^{-c_1}_2
 ({}^\bullet \pi)_{21}\\
  =
 \Rop^{c_3}_1\Rop^{-a_3}_2\Rop^{-a_1}_3\underline{{\tau}_{12}\Rop^{c_1}_2}
 {\tau}_{23}\Lop^{-a_3}_1\Rop^{-c_3}_1\Lop^{-a_1}_2\Rop^{-c_1}_2
 ({}^\bullet \pi)_{21}\\
  = \Rop^{c_3}_1\Rop^{-a_3}_2\Rop^{-a_1}_3
 \Rop^{c_1 }_1\Rop^{c_1}_2
 {\tau}_{12}\, \Cop^{-c_1}_1 {\tau}_{23}\Lop^{-a_3}_1\Rop^{-c_3}_1\Lop^{-a_1}_2\Rop^{-c_1}_2
 ({}^\bullet \pi)_{21}\\
 =
\Rop^{c_2}_1\Rop^{c_0-a_2}_2\Rop^{-a_1}_3
 {\tau}_{12}{\tau}_{23}({}^\bullet
 \pi)_{21}\Cop^{-c_1}_1\Lop^{-a_3}_1\Rop^{-c_3}_1\Lop^{-a_1}_2\Rop^{-c_1}_2,
 \end{multline*}
 where the last equality follows from the definition of ${}^\bullet
 \pi$ and the fact that $\srop{}$, $\slop{}$, $\scop{}$ are
 grading-preserving operators.  We also have
 \begin{multline*}
 \sqop{4(a_0c_0+a_2c_2+a_4c_4)}{}_1 \, {\tau}_{23}(a_0,c_0)\, {\tau}_{13}(a_2,c_2) \, {\tau}_{12}(a_4,c_4)\\
 =
\Rop^{c_0}_2\Rop^{-a_0}_3{\tau}_{23}\Lop^{-a_0}_2\Rop^{-c_0}_2
\Rop^{c_2}_1\Rop^{-a_2}_3{\tau}_{13}\Lop^{-a_2}_1\Rop^{-c_2}_1
\Rop^{c_4}_1\Rop^{-a_4}_2{\tau}_{12}\Lop^{-a_4}_1\Rop^{-c_4}_1\\
  =
\Rop^{c_2}_1\Rop^{c_0}_2\Rop^{-a_0}_3\underline{{\tau}_{23}\Rop^{-a_2}_3}
{\tau}_{13}\Lop^{-a_2}_1\Rop^{-a_0-c_1}_1
 \Lop^{-a_0}_2\Rop^{-c_1}_2{\tau}_{12}\Lop^{-a_4}_1\Rop^{-c_4}_1\\
 =
\Rop^{c_2}_1\Rop^{c_0-a_2}_2\Rop^{-a_0-a_2}_3{\tau}_{23}
\Cop^{a_2}_2 {\tau}_{13}
 \Lop^{-a_2}_1
 \Rop^{-a_0}_1 \Lop^{-a_0}_2\underline{\Rop_1^{-c_1}\Rop_2^{-c_1}
 {\tau}_{12}}\Lop^{-a_4}_1\Rop^{-c_4}_1\\
 =\Rop^{c_2}_1\Rop^{c_0-a_2}_2\Rop^{-a_1}_3{\tau}_{23}{\tau}_{13}
 \Lop^{-a_2}_1
 \Cop^{a_2}_2\underline{\Rop_1^{-a_0}\Lop_2^{-a_0}{\tau}_{12}}\Cop^{-c_1}_1 \Rop^{-c_1}_2 \Lop^{-a_4}_1\Rop^{-c_4}_1\\
  =\Rop^{c_2}_1\Rop^{c_0-a_2}_2\Rop^{-a_1}_3{\tau}_{23}{\tau}_{13}
 \underline{\Lop_1^{-a_2}\Cop_2^{a_2}
 {\tau}_{12}}\Rop^{-a_0}_1 \Lop^{-a_0}_2 \Cop^{-c_1}_1\Lop^{-a_4}_1\Rop^{-c_4}_1\Rop^{-c_1}_2\\
  =\Rop^{c_2}_1\Rop^{c_0-a_2}_2\Rop^{-a_1}_3{\tau}_{23}{\tau}_{13}
 {\tau}_{12}\Cop^{-c_1}_1\Lop^{-a_2}_1\Rop^{-a_0}_1\Lop^{-a_4}_1\Rop^{-c_4}_1\Lop^{-a_1}_2
 \Rop^{-c_1}_2\\
  =\sqop{-8a_0a_4}{}_1\Rop^{c_2}_1\Rop^{c_0-a_2}_2\Rop^{-a_1}_3{\tau}_{23}{\tau}_{13}
 {\tau}_{12}\Cop^{-c_1}_1\Lop^{-a_3}_1\Rop^{-c_3}_1\Lop^{-a_1}_2
 \Rop^{-c_1}_2.
 \end{multline*}
 Comparing the   obtained expressions and   using that
 ${\tau}_{12}{\tau}_{23}({}^\bullet \pi)_{21}={\tau}_{23}{\tau}_{13}
 {\tau}_{12}$, we conclude that the charged pentagon equality
 follows from the formula
 \[
 a_0c_0+a_2c_2+a_4c_4+2a_0a_4=a_1c_1+a_3c_3.
 \]
 This  formula is   verified as follows:
 \begin{multline*}
 a_0c_0+a_2c_2+a_4c_4+2a_0a_4=a_0(c_0+a_4)+a_2c_2+a_4(a_0+c_4)\\
 =a_0c_1+a_2(c_1+c_3)+a_4c_3=(a_0+a_2)c_1+(a_2+a_4)c_3=a_1c_1+a_3c_3.
 \end{multline*}

  We now prove the first inversion identity:
 \begin{multline*}
 {\tau}_{21}(a,c) \, \bar {\tau}(-a,-c)=\sqop{-4ac}{}_1\Rop^{c}_2\Rop^{-a}_1{\tau}_{21}\Lop^{-a}_2\Rop^{-c}_2
 \sqop{4ac}{}_1\Rop^{c}_2\Lop^{a}_2\bar {\tau}\Rop^{a}_1\Rop^{-c}_2\\
 =\Rop^{c}_2\Rop^{-a}_1{\tau}_{21}\bar {\tau}\Rop^{a}_1\Rop^{-c}_2=
\Rop^{c}_2\Rop^{-a}_1 \pi^\bullet \Rop^{a}_1\Rop^{-c}_2=\pi^\bullet,
 \end{multline*}
 where the second equality follows from~\eqref{E:q1=q2++++} since $q$ is a $T$-scalar.
 The equality $\Rop^{c}_2\Rop^{-a}_1 \pi^\bullet \Rop^{a}_1\Rop^{-c}_2=\pi^\bullet$ follows from the fact that
  $\pi^\bullet$ is the projector on a direct sum of   multiplicity spaces.
The second inversion identity is proved similarly.
 \end{proof}

\section{Charged $6j$-symbols}\label{section9}

  Let $i,j,k,l,m,n\in I$ and $a, c
 \in\frac12\Z$. Replacing $T$ by $T(a,c)$ in the definition of the positive $6j$-symbol in Section~\ref{section42Ibelieve},
 we obtain the {\it charged positive
 $6j$-symbol}
 \begin{equation}
 \sj ijklmn (a,c) \in  H_m^{kl}\otimes H_k^{ij}\otimes H_{jl}^n \otimes
H_{in}^m .
\end{equation}
Replacing $\bar T$ by $\bar T(a,c)$ in the definition of the
negative $6j$-symbol in Section~\ref{section42Ibelieve},
 we obtain the {\it charged negative
 $6j$-symbol}
 \begin{equation}
 \sjn ijklmn (a,c)    \in H_m^{in}\otimes H_n^{jl}\otimes H_{ij}^k \otimes
H_{kl}^m.
\end{equation}
The formulas of   Section~\ref{section42Ibelieve} computing $T, \bar
T, \tau, \bar \tau$ in terms of the $6j$-symbols extend to the
present setting by inserting   $(a,c)$ after each occurrence of $T,
\bar T, \tau, \bar \tau$ and after each $6j$-symbol.

  The properties of the charged $T$-forms
  established in Section~\ref{section8now} can be
rewritten in terms of the charged $6j$-symbols. For any
$a,b,c\in\frac12\Z$ such that $a+b+c=\frac12$,
Formula~\eqref{E:syms01} yields
\begin{equation}\label{E:6jsyms01}
 \sj ijklmn (a,c) =  P_{(4321)}
q_1^{2a}\Aops_1 \Aops_3   \left ( \sjn {i^*}kjlnm (a,b)\right ),
\end{equation}
where $\Aops_1$ is induced by the restriction of $\Aops$ to
$H^{i^*m}_n$ and $\Aops_3$ is induced by the restriction of $\Aops$
to $H^{j}_{i^*k}$. Formula~\eqref{E:syms12} yields
\begin{equation}\label{E:6jsyms12}
 \sj ijklmn (a,c) =  P_{(23)} q_1^{-2c}\Aops_2 \Bops_3   \left ( \sjn
k{j^*}inml (b,c) \right ),
\end{equation}
 where $\Aops_2$ is induced by the
restriction of $\Aops$ to $H^{{j^*}n}_l$  and $\Bops_3$ is induced
by the restriction of $\Bops$ to $H^{i}_{kj^*}$. Finally,
Formula~\eqref{E:syms23} yields
\begin{equation}\label{E:6jsyms23}
 \sj ijklmn (a,c)=  P_{(1234)} q_1^{2a}\Bops_2 \Bops_4  \left ( \sjn inm{l^*}kj
(a,b) \right ),
\end{equation}
 where $\Bops_2$ is induced by the restriction of
$\Bops$ to $H^{nl^*}_j$  and $\Bops_4$ is induced by the restriction
of $\Bops$ to $H_{ml^*}^{k}$.

The charged pentagon identity   yields that for any $j_0,j_1,
 \ldots , j_8\in I$ and   any   $a_0,a_1, a_2,a_3, a_4, c_0, c_1, c_2, c_3, c_4
 \in\frac12\Z$ satisfying~\eqref{pentagonconditions},
 $$\sum_{j\in I} \ast^{jj_4}_{j_7} \ast^{j_2j_3}_j \ast^{j_1j}_{j_6}
   (\sj {j_1}{j_2}{j_5}{j_3}{j_6}j (a_0, c_0)\otimes
  \sj {j_1}j{j_6}{j_4}{j_0}{j_7} (a_2, c_2)$$
  $$ \otimes
  \sj {j_2}{j_3}{j}{j_4}{j_7}{j_8} (a_4, c_4)  )$$
 \begin{equation}\label{E:ChPentId}
 =g_{j_2,j_3} P_{(135642)} \ast^{j_5j_8}_{j_0}   (
 \sj {j_1}{j_2}{j_5}{j_8}{j_0}{j_7} (a_1, c_1) \otimes
 \sj {j_5}{j_3}{j_6}{j_4}{j_0}{j_8} (a_3, c_3)   ).
 \end{equation}
Here both sides lie in the $\FK$-vector
space~\eqref{ambientvectorspacepentagon}.

The first inversion relations  gives that for all $i,j,k,k',l,m \in
I$ and $a,c \in\frac12\Z$,
$$\sum_{n\in I} \ast^n_{jl}\ast^m_{in} (\sj ijklmn (a,c) \otimes \sjn
ij{k'}lmn (-a, -c))$$ $$=\delta_{k'}^k \, g_{j,l} \,  P_{(432)}
(\delta^{kl}_{m}\otimes \delta^{ij}_k) .$$   The   second inversion
relations  gives that for all $i,j, l,m ,n,n' \in I$ and $a,c
\in\frac12\Z$,
$$\sum_{k\in I} \ast^k_{ij}\ast^m_{kl} (\sjn ijklm{n'} (-a,-c) \otimes \sj
ij{k}lmn (a,c)) $$ $$= \delta_{n'}^n\,  g_{i,j}\, P_{(432)} (\delta^{in}_{m}\otimes
\delta^{jl}_n) .$$
\section{Three-manifold invariants}\label{section10}
In this section, following the ideas of the paper \cite{K1}, we associate to a $\spsi$-system an invariant of an oriented compact three-manifold together with a non-empty link.
\subsection{Topological preliminaries}\label{sect-MMMLLL} Throughout this subsection, the symbol $M$ denotes a closed connected orientable
$3$-manifold. Following \cite{BB}, by a
\emph{quasi-regular triangulation} of $M$ we mean a decomposition of $M$ as a
union of embedded tetrahedra such that the intersection of any two tetrahedra
is a union (possibly, empty) of several of their vertices, edges, and
(2-dimensional) faces.  Quasi-regular triangulations differ from the usual
triangulations in that they may have tetrahedra meeting along several
vertices, edges, and faces. Note that  each edge  of a quasi-regular
triangulation has two distinct endpoints.

A \emph{Hamiltonian link} in a quasi-regular
triangulation $\T$ of $M$ is a set $\LL$ of unoriented edges of $\T$ such that every
vertex of $\T$ belongs to exactly two elements of $\LL$.  Then the union of the
edges of $\T$ belonging to $\LL$ is a link $L$ in $M$. We call the pair
$(\T,\LL)$ an \emph{$H$-triangulation} of $(M,L)$.

\begin{proposition}[\cite{BB}, Proposition 4.20]\label{L:Toplemma-}
 For any non-empty link $L$ in $M$, the pair ($M$,
 $L$) admits an $H$-triangulation.
\end{proposition}

 $H$-triangulations   of $(M,L)$ can be related by
elementary moves of two types, the {\bubble} moves and the $H$-Pachner $2\leftrightarrow
 3$ moves.
 The
{\it positive {\bubble} move} on an $H$-triangulation $(\T,\LL)$ starts with a choice of a face $F=v_1v_2v_3$ of
$\T$ such that at least one of its edges, say $v_1v_3$, is in $\LL$. Consider
two tetrahedra of $\T$ meeting along~$F$. We unglue these tetrahedra along $F$
and insert a $3$-ball between the resulting two copies of $F$. We triangulate
this $3$-ball by adding a vertex $v_4$ at its center and three edges connecting  $v_4$ to $ v_1$,
$ v_2$, and $ v_3$.  The edge $v_1v_3$ is removed from $\LL$ and replaced by the
edges $ v_1v_4$ and $ v_3v_4$.
 This move can be visualized as the transformation
$$ \epsh{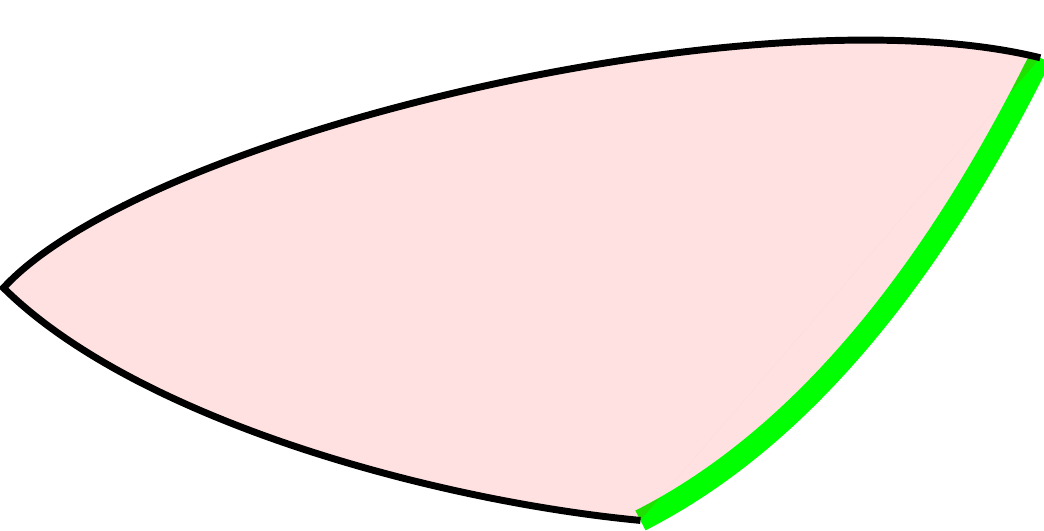}{30pt}\longrightarrow \epsh{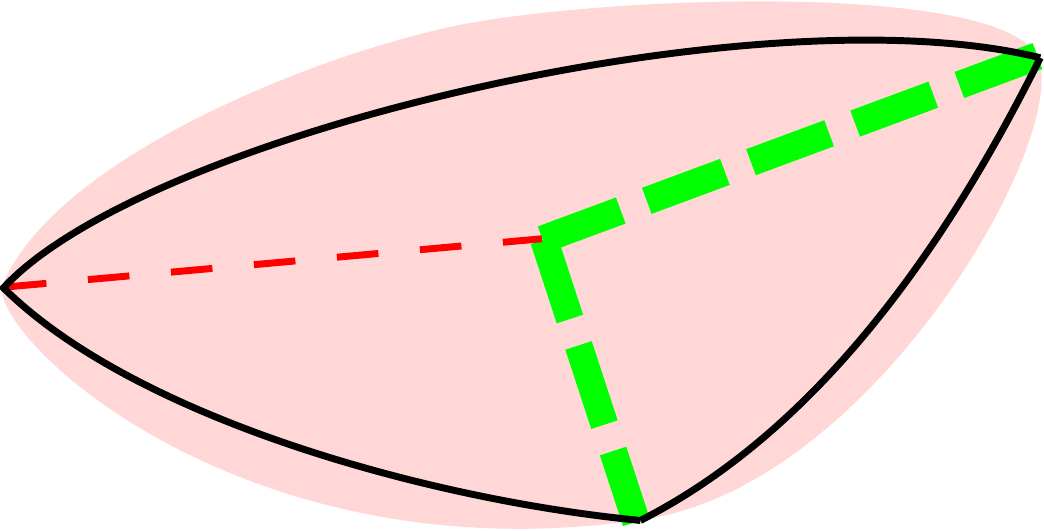}{30pt}
 \put(-23,0){\tiny $v_4$}
 $$
where the bold (green) edges belong to $\LL$.  The inverse move is the {\it
 negative {\bubble} move}.
The  {\it positive  {\Pachner} $2\leftrightarrow
 3$ move}   can be visualized as the transformation
$$\epsh{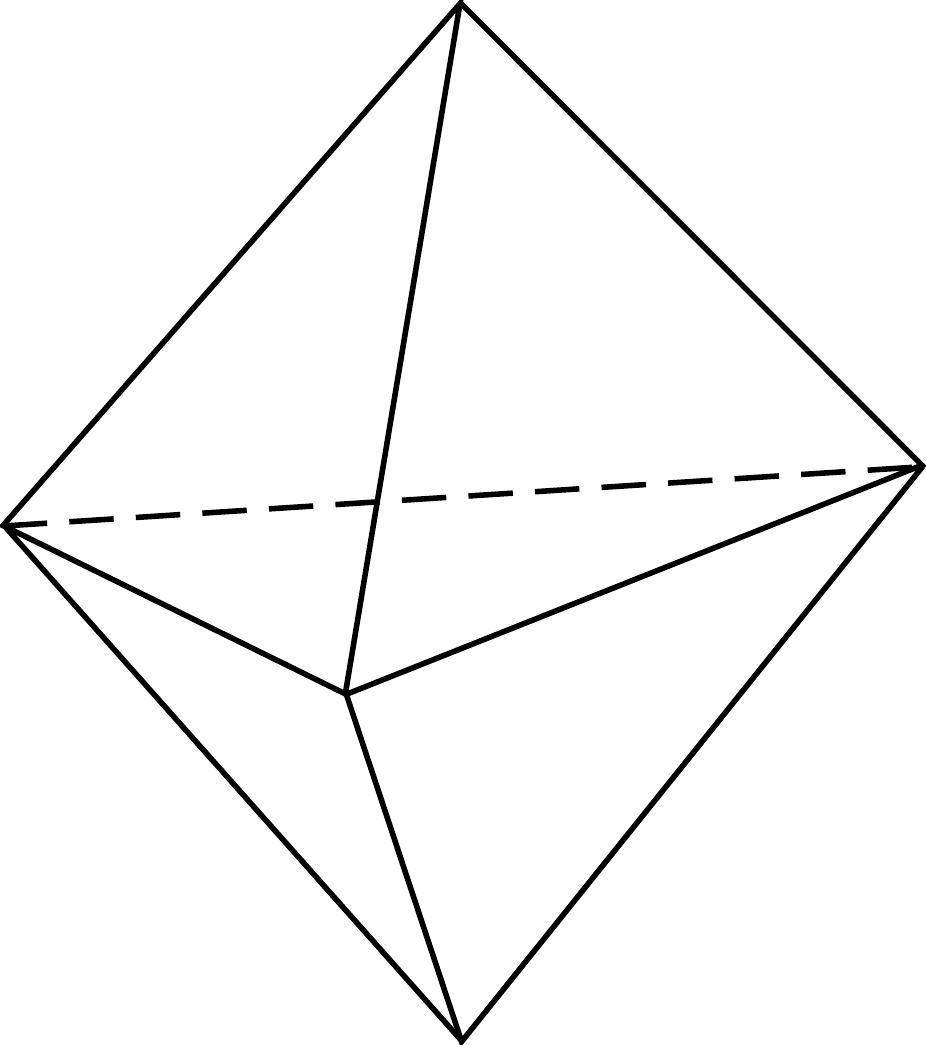}{50pt}\longleftrightarrow\epsh{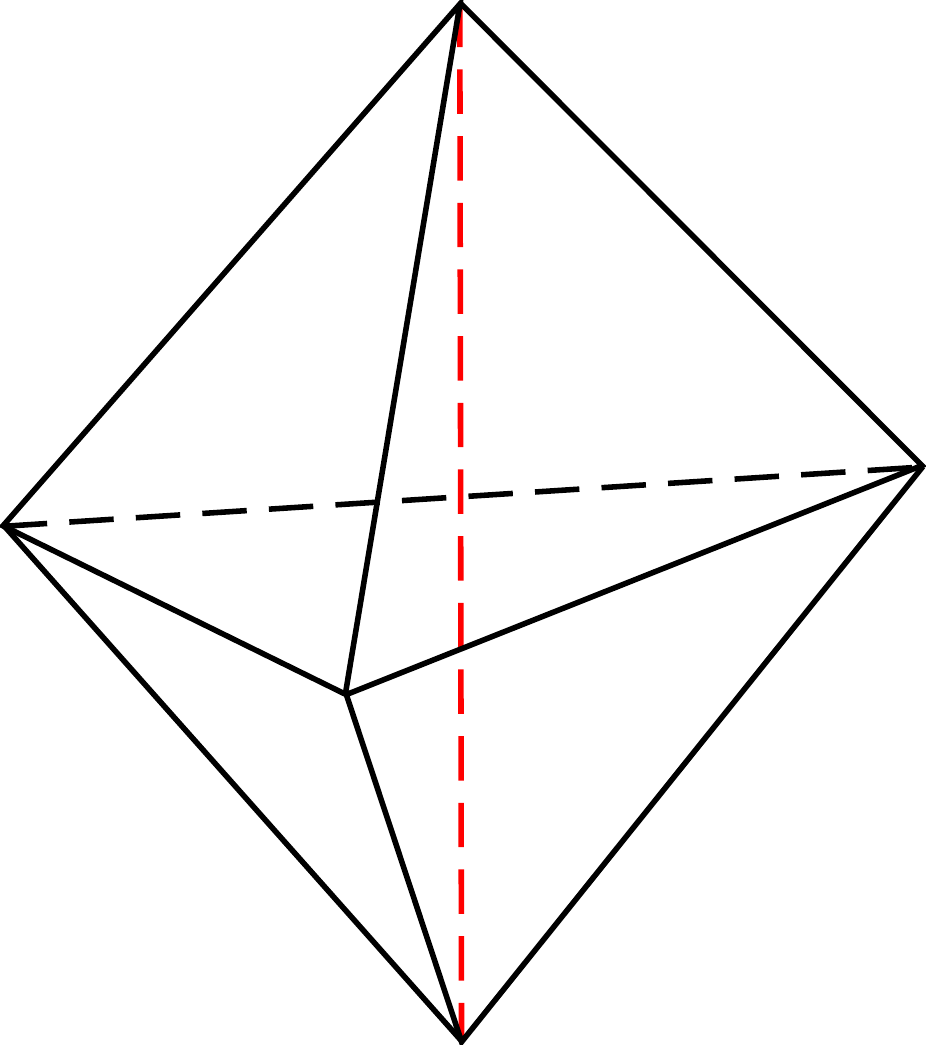}{50pt}.$$
This move  preserves the set  $\LL$.  The inverse move is the {\it negative  {\Pachner}  move}; it is allowed only when the edge common to the three tetrahedra on the right
is not in~$\LL$.

\begin{proposition}[\cite{BB}, Proposition 4.23]\label{L:Toplemma}
 Let $L$ be a non-empty link in
 $M$.  Any two $H$-triangulations of $(M,L)$ can be related by a finite
 sequence of {\bubble} moves and {\Pachner} moves in the class of
 $H$-triangulations of $(M,L)$.
\end{proposition}

Charges on $H$-triangulations first have been introduced in \cite{K1} and  the corresponding theory subsequently has been developed in \cite{BB}.  This theory is a natural extension of the theory of angle structures, see for example \cite{Neu90,Neu98}, into the framework of arbitrary triangulated three-manifolds.

 By a  \emph{charge} on a tetrahedron $T$, we mean a $\frac12\Z$-valued function $c$  on the set of edges of $T$ such that $c(e)=c(e')$ for any opposite edges $e, e'$ and $c(e_1)+c(e_2)+c(e_3)=1/2$
 for any edges $e_1, e_2, e_3$ of $T$ forming the boundary of a face.
Consider now an   $H$-triangulation $(\T,\LL)$ of $(M,L)$ as above. Let $E(\T)$ be the set of edges of $\T$ and let $\widehat{E}(\T)$ be the set of pairs (a tetrahedron $T$ of $\T$, an edge of $T$). Let $\epsilon_\T: \widehat{E}(\T) \rightarrow E(\T)$ be the obvious  projection.
For any edge $e$ of $\T$, the set  $\epsilon_\T^{-1}(e)$ has $n$ elements, where $n$ is the number of tetrahedra of $\T$ adjacent to $e$.

\begin{definition}\label{D:intcharge}
A  \emph{charge} on   $(\T,\LL)$  is a map $c: \widehat{E}(\T) \rightarrow \frac12\Z$ such that
\begin{enumerate}
\item the restriction of $c$ to any   tetrahedron $T$ of $\T$ is a  charge on $T$, \label{I:Defintcharge1}
\item  for each edge $e$ of $\T$ not belonging to $\LL$ we have $\sum_{e'\in \epsilon_\T^{-1}(e)}c(e')=1$,
\item  for each edge $e$ of $\T$ belonging to $\LL$ we have $\sum_{e'\in \epsilon_\T^{-1}(e)}c(e')=0$.
 \end{enumerate}
 \end{definition}

 Each charge $c$ on  $(\T,\LL)$ determines a cohomology class $[c]\in H^1(M; \Z/2\Z)$ as follows.
 Let $s$ be a simple closed curve in $M$ which lies  in general position with respect to $\T$ and such that $s$   never leaves a tetrahedron $T$ of  $\T$  through  the same 2-face by which it entered.
Thus each time $s$ passes through   $T$, it determines  a unique edge $e$ belonging to both  the entering and departing
faces. The sum of the residues $2 \,c \vert_T(e) (mod \, 2) \in \Z/2\Z$ over all passages of $s$ through  tetrahedra of $\T$ depends only on the homology class of $s$ and is the value of $[c]$ on $s$.

It is known that   each $H$-triangulation   $(\T,\LL)$ of $(M,L)$ has a  charge representing any given element of  $  H^1(M; \Z/2\Z)$.
We briefly outline a proof of this claim following \cite{Neu90},  \cite{Bas} and referring to these papers for the exact definitions and the details.  In this argument (and only here)  we shall use \lq\lq integral charges" that  are equal to two times our charges and take only integer values.  Let $\overline J$ be the abelian group generated by pairs (a tetrahedron $\Delta$ of $\mathcal T$, an edge of $\Delta$) modulo  the relations   $(\Delta, e)=(\Delta, \overline e)  $
where $\overline e$ is the edge opposite to $e$ in $\Delta$. An integral charge on $\mathcal T$ may be seen  as an element   of $\overline J$ satisfying certain additional properties. Recall   the Neumann chain complex associated with $\mathcal T$:
$$ \xymatrix
{
\tau=   (  C_0 \ar[r]^{\alpha } & C_1 \ar[r]^{\beta } & J \ar[r]^{\beta^* }&C_1  \ar[r]^{\alpha^* } &C_0  \ar[r] & 0).
 }
$$
Here  $C_i$ with $i=0, 1$ is the free abelian group freely generated by the $i$-dimensional (unoriented) simplices of  $\mathcal T$ and  $J$ is the quotient of $\overline J$ by the relations    $(\Delta, e_1)+(\Delta, e_2) +(\Delta, e_3)=0$
where $e_1, e_2, e_3$ are edges of a tetrahedron $\Delta \in \mathcal T$ forming a triangle.  The homomorphisms
$\alpha $, $\beta $ are defined by Neumann and the homomorphisms $\alpha^*$, $\beta^* $ are their   transposes  with respect to the obvious bases of $C_0, C_1$ and  a canonical non-degenerate bilinear form  on   $J$.  The relationship to the charges comes from the fact that  $\beta^*: J \to C_1$ splits canonically as a composition of  certain  homomorphisms $\beta_1:J \to \overline J$ and $\beta_2: \overline J \to C_1$. The rest of the argument is a homological chase. One starts with any   $x\in \overline J$  such that for every tetrahedron $\Delta$ of $\mathcal T$, the  coefficients of the edges of $\Delta$ in $x$ total to $1$.  One shows that then $\beta_2(x)- 2 \sigma \in \Ker (\alpha^*)$ where $\sigma\in C_1$ is the formal sum of the edges of $\mathcal T$ not belonging to $\LL$.  Using Neumann's computation of
$ H_2(\tau) $, one deduces  that    $\beta_2(x)- 2 \sigma=\beta^*(a)$ for some   $a\in J$. Then $x'=x-\beta_1(a)$  is an integral charge on $\mathcal T$.  Next,  using Neumann's formula  $H_3(\tau)= H^1(M; Z/2Z)$, one picks  a  cycle $b\in J$ of the chain complex $\tau$ such that $ x'-\beta_1(b)$ is  an integral charge on $\mathcal T$ representing the given element of
$H^1(M; Z/2Z)$.

\begin{lemma}\label{L:chargetransit}
Let $(\T,\LL)$ and $(\T',\LL')$ be $H$-triangulations of $(M,L)$ such that $(\T',\LL')$ is  obtained from $(\T,\LL)$ by an
 {\Pachner} move or an {\bubble} move.  Let $c$ be a charge on $(\T,\LL)$.
 Then there exists a charge $c'$ on $(\T',\LL')$ such that $c'$ equals $c$ on all pairs  (a tetrahedron $T$ of $\T$  not involved in the move, an edge  of  $T$) and for any  common edge  $e$ of $\T$ and $\T'$,
 \begin{equation}\label{E:intTrans22}
 \sum_{a\in \epsilon_\T^{-1}(e)}c(a)=  \sum_{a'\in \epsilon_{\T'}^{-1}(e)}c'(a').
 \end{equation} Moreover, $[c]=[c']$.
\end{lemma}
\begin{proof}
A straightforward calculation, cf. \cite{BB}, Lemma 4.10.
\end{proof}

The   charge $c'$ in this lemma  is unique if the move $(\T,\LL) \mapsto (\T',\LL')$ is negative. In this case we say that $c'$ is {\it induced} by $c$. If the move $(\T,\LL) \mapsto (\T',\LL')$ is
 positive, then $c'$ is not unique,  see Lemma 4.12 of \cite{BB}.

\subsection{The algebraic data}\label{SS:AlgPre}
We describe the algebraic data needed to define our  3-manifold invariant.
Let $\cat$ be a monoidal Ab-category  whose ground ring $\FK$
is a field.     Fix   a  $\spsi$-system  in $\cat$ with distinguished  simple objects $\{V_i\}_{i\in I}$.
Fix a family  $\{I_g\}_{g\in G}$ of  finite subsets of the set $I$ numerated by  elements of a group $G$ and satisfying the following conditions:
\begin{enumerate}
\item \label{I:comp1} for any $g\in G$,  if $i\in I_g$, then $i^*\in I_{g^{-1}}$;
\item   \label{I:comp2} for any   $i_1\in I_{g_1}, i_2\in I_{g_2}$, $k\in I\setminus I_{g_1g_2}$ with $g_1,g_2\in G$, we have  $H^{i_1 i_2}_k=0$;
 \item   \label{I:comp2+} if $i_1\in I_{g_1}, i_2\in I_{g_2}$  with $g_1,g_2\in G$, then either $I_{g_1g_2}= \emptyset$ or there is $k\in I_{g_1g_2}$ such that $H^{i_1 i_2}_k\neq 0$;
\item  \label{I:comp3} for any finite family $ \{g_r \in \Gr\}_r$, there is $g\in  \Gr$ such that $I_{gg_r}\neq\emptyset$ for all $r$;
\item  \label{I:comp4} we are given a map $\bb:I\to \FK$  such that  $\bb(i)=\bb(i^*)$  for all $i\in I$, and
 for any $g_1,g_2 \in \Gr$, $k\in I_{g_1g_2}$   such that $I_{g_1} \neq \emptyset$ and  $I_{g_2} \neq \emptyset$, \begin{equation*}
 \sum_{i_1 \in I_{g_1},\, i_2\in I_{g_2}}
 \bb({i_1})\bb({i_2})\dim(H^{i_1i_2}_{k}) =  \bb(k).
 \end{equation*}

\end{enumerate}

\subsection{$\Gr$-colorings and state sums}   Fix   algebraic data as in Section \ref{SS:AlgPre}.  Let  $M$ be
 a closed connected orientable
$3$-manifold and $\T$ a quasi-regular
triangulation  of~$M$
as in Section \ref{sect-MMMLLL}. A \emph{$\Gr$-coloring}
of $\T$ is a map $\wp$ from
the set of oriented edges of $\T$ to $\Gr$ such that
\begin{enumerate}
\item $\wp(-e)=\wp(e)^{-1}$ for any oriented edge $e$ of $\T$, where $-e$ is $e$
 with opposite orientation and
\item if $e_1, e_2, e_3$ are   ordered  edges of a  face of $\T$ endowed with orientation induced by the   order, then $\wp(e_1)\, \wp(e_2)\, \wp(e_3)=1$.
\end{enumerate}

A {\em $\Gr$-gauge of $\T$} is a map   from the set of  vertices of $\T$ to  $\Gr$.  The  $\Gr$-gauges of $\T$
form a multiplicative group which acts on the set of $\Gr$-colorings of $\T$ as follows.
 If $\delta$ is a $\Gr$-gauge of $\T$ and $\wp$ is a $\Gr$-coloring of $\T$,  then the $\Gr$-coloring $\delta \Phi$ is given by
$$(\delta\Phi)(e)=\delta(v^-_e)\,\Phi(e)\, \delta(v^+_e)^{-1},$$
where $v^-_e$ (resp.\@ $v^+_e$)  is  the  initial (resp.\@ terminal) vertex of an oriented edge $e$.

 Let $\MG M$ be the set of conjugacy classes of group homomorphisms from the fundamental group of $M$  to $ \Gr$. The elements of $\MG M$   bijectively correspond to the $\Gr$-colorings   of $\T$ considered up to gauge transformations.  Indeed, for a   vertex $x_0$ of $\T$,  each  $\Gr$-coloring $\Phi$ of $\T$ determines  a homomorphism $\pi_1(M,x_0)\to\Gr$. To compute  this homomorphism on an element of $\pi_1(M,x_0) $, one   represents this element by a loop based at $x_0$ and formed by  a sequence of  oriented edges of $\T$; then one takes the  product of the
values of $\Phi$ on these edges.     Let $[\Phi]\in \MG M$  be  the conjugacy class of this homomorphism. We say that $\Phi$ represents $[\Phi]$.  The assignment $\Phi \mapsto [\Phi]$ establishes the bijective correspondence mentioned above.

 A {\it state} of a  $\Gr$-coloring $\wp$  of $\T$  is a map $\p$ assigning to every
oriented edge $e$ of $\T$ an element $\p (e)$ of $I_{\wp(e)}$ such that
$\p(-e)=\p (e)^*$ for all $e$.  The set of all states of $\wp$ is denoted
$\states(\wp)$. The identity $\bb(
{{\p(e)}})=\bb( {{\p(-e)}})$ allows us to use the notation
$\bb(\p(e))$ for non-oriented edges.  It is easy to see that $\states(\wp) \neq \emptyset$  if and only if  $I_{\Phi(e)}\neq \emptyset$ for all oriented edges $e$ of $\T$. In this case we say that    $\Phi$  is  \emph{admissible}.

Let now $L$ be a non-empty link in $M$ and $(\T,\LL)$ be an $H$-triangulation of $(M,L)$  with  charge $c$.       From this data, we derive a certain partition function (state sum) as follows. Fix   a total order on the set of  vertices of $\T$.
 Consider a
tetrahedron $T$ of $\T$ with vertices $v_1, v_2, v_3, v_4$ in increasing order. We say that  $T$ is {\it right oriented} if the tangent vectors $v_1v_2, v_1 v_3, v_1 v_4$ form a positive basis in the tangent space of~$M$; otherwise   $T$ is {\it left oriented}.
 For an admissible $\Gr$-coloring $\wp$  of $\T$ and a state $\p \in \states(\wp)$, set
$$
 i=\p(\vect{v_1v_2}),
j=\p(\vect{v_2v_3}),
 k=\p(\vect{v_1v_3}),
 l=\p(\vect{v_3v_4}),
 m=\p(\vect{v_1v_4}), n=\p(\vect{v_2v_4}),
$$
 where $\vect{v_iv_j}$ is the oriented edge of $T$ going from $v_i$ to $v_j$.
Set
 $$
|T|_\p=
\left\{
\begin{array}{l}
\sj ijklmn(c({v_1v_2}),c({v_2v_3}))
 \, \text{if $T$ is right oriented}, \\
\sjn ijklmn(c({v_1v_2}),c({v_2v_3}))
 \, \text{if $T$ is left oriented.}
\end{array}\right.
$$
The $6j$-symbol $|T|_\p$ belongs to the tensor product of $4$ multiplicity modules
associated to the faces of $T$.  Specifically,
$$
|T|_\p\in
\left\{
\begin{array}{l}
H_m^{kl}\otimes H_k^{ij}\otimes H_{jl}^n \otimes
H_{in}^m
 \,\,\, \text{if $T$ is right oriented}, \\
H_m^{in}\otimes H_n^{jl}\otimes H_{ij}^k \otimes
H_{kl}^m
 \,\,\, \text{if $T$ is left oriented.}
\end{array}\right.
$$
 Note that any face of $\T$ belongs to exactly
two tetrahedra of $\T$, and the associated multiplicity modules are dual to
each other, see Lemma \ref{L:1}. These dualities allow us to contract     $\otimes_T |T|_\p$      into a scalar.  Denote by
$\cntr$ the tensor product of  all these tensor contractions determined by the faces of $\T$.    Set
$$
\Kas(\T,\LL,\wp,c)=\sum_{\p \in \states(\wp) }\, \left (\prod_{e\in \LL} \bb(\p(e))\right ) \, \cntr \left
 (\bigotimes_{T}|T|_{\p}\right ) \in \FK,
$$
where $T$ runs over all  tetrahedra of $\T$.  To compute $ \Kas(\T,\LL,\wp,c)$ we may need to order  the faces of $\T$, but the result does not depend on this order.

\begin{theorem}\label{T:MainTopInv}
Suppose that there exists a scalar $\wsqop\in \FK$ such that
$\sqop{}{8}$ is  $T$-equal to the operator
$$\wsqop\Id_{\hat H}\oplus \wsqop^{-1}\Id_{\check H}\in \End(H).$$
 Then, up to multiplication by integer powers of $\wsqop$, the state sum $\Kas(\T,\LL,\wp,c)$ depends only on the isotopy class of $L$ in $M$,  the conjugacy class  $[\wp]\in \MG M$, and the cohomology class $[c]\in H^1(M;\Z/2\Z)$ (and does not depend
 on the choice of
 $c$ in its cohomology class, the admissible representative $\wp$ of $[\wp]$, the   $H$-triangulation $\T$ of $(M,L)$, and  the ordering of the vertices of $\T$).
\end{theorem}

A proof of this theorem will be given in Section \ref{section11}.

\begin{lemma}\label{eee+}
 Any element of the set $  \MG M$ can be represented by an admissible $G$-coloring on an
 arbitrary quasi-regular triangulation $\T$ of $M$.
\end{lemma}
\begin{proof}
 Take any $G$-coloring $\Phi$ of $\T$ representing the given element of $  \MG M$.
  We say that a vertex
 of $\T$ is {\it bad} for $\Phi$ if there is an oriented edge $e$ in $\T$
 incident to this vertex and such that $I_{\Phi(e)}=\emptyset$.  It is clear that $\Phi$ is
 admissible if and only if $\Phi$ has no bad vertices. We show how to modify
 $\Phi$ in the class $[\Phi]$ to reduce the number of bad vertices.  Observe first that each pair (a vertex $v$ of $\T$, an element $g$ of $\Gr$) determines  a $\Gr$-gauge $\delta^{v,g} $   whose value on any vertex $u$ of $\T$  is defined by
\begin{equation}\label{E:DefDeltaGauge}
\delta^{v,g} (u)=\left\{\begin{array}{ll}g &  \text{if $u=v$,}\\ 1 & \text{else.}\end{array}\right.
\end{equation}

 Let $v
 $ be a bad vertex for $\Phi$.   Pick any $g\in \Gr$ such that $ I_{g \Phi(e)}\neq \emptyset$ for  all  edges
 $e$ of $\T$ adjacent to $v$ and oriented away from~$v$.
 The $G$-coloring  $\delta^{v,g}  \wp$ takes values in the set $\{h\in \Gr \,\vert\, I_h\neq \emptyset\}$ on all edges of $\T$ incident to $v$ and takes the same values as $\wp$ on all edges of $\T$ not incident
 to $v$.  Here we use the fact that the edges of $\T$ are not loops which
 is ensured by  the quasi-regularity of $\T$.  The transformation $\wp \mapsto
 \delta^{v,g}  \wp$ decreases the number of bad vertices.  Repeating this
 argument, we find a $\Gr$-coloring   without bad vertices  in the class $[\Phi]$.
\end{proof}

We represent any $h\in \MG M$ by an admissible $G$-coloring $\wp$ of $\T$ and any
$\zeta\in H^1(M; \Z/2\Z)$ by a charge $c$ and set
$$\Kas(M,L,h, \zeta)=\Kas(\T,\LL,\wp,c)\in \FK.$$
By Theorem \ref{T:MainTopInv}, $\Kas(M,L,h,\zeta)$ is a topological invariant of the tuple
$(M,L,h,c)$.

\section{Proof of Theorem \ref{T:MainTopInv}}\label{section11}

Throughout this section we keep the assumptions of Theorem \ref{T:MainTopInv}.

\begin{lemma}\label{L:indOrdering}
Up to multiplication  by  integer powers of $\wsqop$, $\Kas(\T,\LL,\wp,c)$ does not depend on the ordering of the vertices of $\T$.
\end{lemma}
\begin{proof}
Consider the natural action of the symmetric group  on the orderings of the vertices  of $\T$.   As the symmetric group is generated by simple transpositions $(r,r+1)$, it is enough to consider  the action of  one such transposition on an ordering. If the  vertices   labelled by $r$ and $r+1$ do not span an edge of $\T$, then the new state sum is identical  to the old one.  Suppose that   an edge,   $e$,  of $\T$ connects the    vertices   labelled by $r$ and $r+1$.
Let $P$ be the set of all labels $p$   such that the vertices   labeled by $r,r+1$, and $p$ form a face  of $\T$.
This face, denoted $f_p$, belongs to two adjacent tetrahedra of $\T$ containing  $e$ and  determines   two dual multiplicity spaces.

For a tetrahedron $T$ of $\T$, consider the   transformation of the $6j$-symbol  $|T|_{\p}$ under  the permutation $(r,r+1)$.  If $T  $ does not contain $e$, then    $|T|_{\p}$ does not change. We claim  that for $T \supset e$,     the   $6j$-symbol  $|T|_{\p}$  is  multiplied by an integer power of~$\wsqop$ independent of $\p$ and  composed with the tensor product      of   operators    acting on the  multiplicity spaces
 corresponding to $f_p$, where $p$ runs over the 2-element  set $\{p\in P\,\vert f_p\subset T\}$.  Here the operator corresponding to $p$ in this set  is  $\Aops$ if $p>r$ and  $\Bops$ if $p<r$.   Since $\Aops $ and $\Bops$ are involutive and self-dual,   the effect of this transformation after the tensor contraction $\cntr$  will be multiplication by an integer power of   $\wsqop$ independent of $\p$. This will imply the  lemma.

 The claim above  follows from Equations \eqref{E:6jsyms01}-\eqref{E:6jsyms23}.  Indeed, let $r,r+1, p, p'$ be the labels of the vertices of $T$. Suppose for concreteness that  $p<r$ and $r+1<p'$ (the other cases are similar).      Then the  left (resp.\@ right) hand side of  Formula \eqref{E:6jsyms12} with $a=c(T,v_pv_r), b=c(T, v_pv_{r+1}), c=c(T, v_rv_{r+1})$  computes $|T|_{\p}$ before (resp.\@ after) the permutation of $r$ and $r+1$.    The operators $\Aops_2$ and $\Bops_3$ in  \eqref{E:6jsyms12} act on the multiplicity spaces corresponding to the  faces $f_{p'} $ and $f_p $, respectively.
 Therefore Formula \eqref{E:6jsyms12} implies our claim.
\end{proof}

\begin{lemma}\label{L:one-movezz}
 Let  $(\T,\LL)$,
 $(\T',\LL')$ be   $H$-triangulations of $(M,L)$ such that $(\T',\LL')$ is obtained from  $(\T,\LL)$  by a negative
 {\Pachner} move or a negative {\bubble} move.  Then any admissible $\Gr$-coloring $\wp$ on $\T$ restricts to an
 admissible $\Gr$-coloring $\wp'$ of~$\T'$. For any  charge  $c$ on $(\T,\LL)$, we have   (up to multiplication by  powers of~$\widetilde q$)
\begin{equation}\label{tvt}
 \Kas(\T,\LL,\wp,c)=\Kas(\T',\LL',\wp',c') ,
 \end{equation}
 where  $c'$ is the   charge  on $(\T',\LL')$  induced by $c$.
\end{lemma}
\begin{proof}
The values of $\wp'$ form a subset of the set of values of
 $\wp$;  therefore the admissibility of $\wp$ implies the    admissibility of $\wp'$.

 The rest of the proof follows the lines of  \cite[Section VII.2.3]{Tu} via   translating the geometric  moves
 into algebraic identities. First, we prove \eqref{tvt} for a negative {\bubble} move.
 Let $ v_1, v_2, v_3, v_4$  be the vertices of $\T$ and $F=v_1v_2v_3$ the face  of $\T'$ as in the description of the bubble  move  in Section~\ref{sect-MMMLLL}  (see Figure \ref{F:col_bub}).
 Since our state sums do not depend on the ordering of the vertices, we assume that  $v_4$ is the last in  the order of
 the vertices of $\T$ and the order of the vertices of $\T'$  is induced by the order of the vertices of $\T$.
We can also assume that   $v_1v_4, v_3v_4 \in \LL$    and      $v_1v_3\in \LL'$.  Let $T_r$ (resp. $T_l$) be the right oriented  (resp. left oriented)   tetrahedron
 of $\T$ disappearing under the move.
 Let $a,  b \in \frac12\Z$ be the $c$-charges of the edges $v_1v_2$, $v_2v_3$  of $T_r$ respectively.  The  properties of a charge  imply that the charges of  the edges $v_1v_2$,  $v_2v_3$ of $T_l$ are $-a$ and $-b$, respectively.

  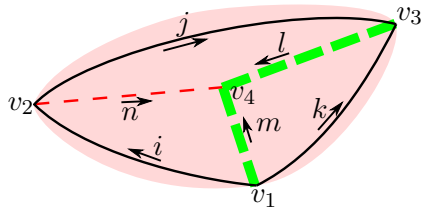
\begin{figure}[b]
 $\epsh{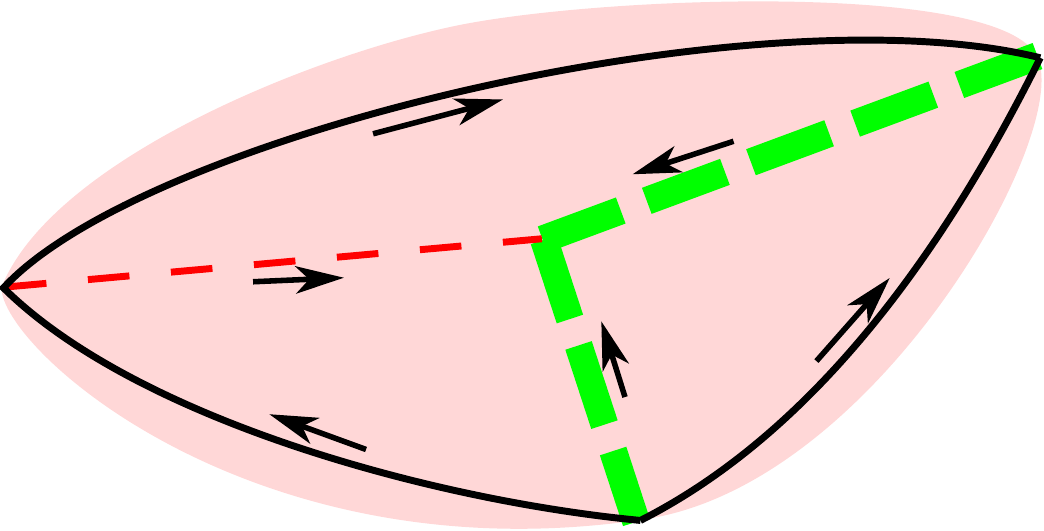}{70pt}$
 \put(-32,-5){\small $k$}\put(-53,-10){\small $m$}
 \put(-45,20){\small $l$}\put(-83,28){\small $j$}
 \put(-104,-6){\small $n$}\put(-92,-19){\small $i$}
 \put(-147,-2){\small $v_2$}\put(0,32){\small $v_3$}
 \put(-55,-37){\small $v_1$}
 \put(-63,2){\small $v_4$}
\caption{$T_l\cup T_r$ colored by $\p\in S$}
 \label{F:col_bub}
\end{figure}

 Fix a state  $\p'\in \states (\wp')$ and let $S \subset \states(\wp)$ be the set
 of all states   of $\wp$ extending~$\p'$.   It is enough to show that the
 term $\Kas_{\p'}$ of $\Kas(\T',\LL',\wp',c')$ associated to $\p'$ is equal to the sum
  $ \sum_{\p\in S} \Kas_{\p}$ of the
 terms of $\Kas(\T,\LL,\wp,c)$ associated to all $\p \in S$.
 Set $i=\p'(\vect{v_1v_2})$, $j=\p'(\vect{v_2v_3})$, and $k=\p'(\vect{v_1v_3})$. For any distinct indices $p,q\in \{1,2,3,4\}$,
 set  $I_{pq}=I_{\Phi (\vect{v_pv_q})}\subset I$. The admissibility of $\Phi$ implies that $I_{pq}\neq \emptyset$ for all $p, q$.
 Clearly,  $i\in I_{12}, j\in I_{23}, k\in I_{13}$. A  state  $\p\in S$ is determined by the    labels
 $$ l=\p(\vect{v_3v_4})\in I_{34}, \quad  m=\p(\vect{v_1v_4}) \in I_{14}, \quad {\text {and}} \quad  n=\p(\vect{v_2v_4})\in I_{24}.$$   We have
$$|T_r|_{\p}= \sj ijklmn(a,b)  \quad {\text {and}} \quad |T_l|_{\p}= \sjn ijklmn(-a,-b).$$
It is convenient to write $|T_r|^{ijk}_{lmn}$ for  $ |T_r|_{\p}$ and $|T_l|^{ijk}_{lmn}$ for $ |T_l|_{\p}$.


Denote by $*_f$ the tensor contraction determined by a
face $f$.  We have  $\Kas_{\p'}=*_{F}(\bb(k)X)$,  where $X$ is
the term  of the state sum   determined by  $\p'$ before   contraction   $*_{F}$ and multiplication by  $\bb(k)$. Let  $F_r$ and $F_l$  be the faces of $T_r$ and  $T_l$, respectively,  with vertices $v_1, v_2, v_3$.
We have
\begin{align*}
\sum_{\p\in S} \Kas_{\p}& = *_{F_r}*_{F_l}\left(X\otimes \sum_{l\in I_{34}, m \in I_{14}, n\in I_{24}} \bb(l)\bb(m)\ast^{kl}_m \ast^n_{jl} \ast^m_{in}\left(|T_r|^{ijk}_{lmn} \otimes |T_l|^{ijk}_{lmn}\right)\right)\\
 &=*_{F_r}*_{F_l}\left(X\otimes \sum_{l\in I_{34}, m\in I_{14}} \bb(l)\bb(m)\ast^{kl}_m \left(\sum_{n\in I_{24}}\ast^n_{jl} \ast^m_{in}|T_r|^{ijk}_{lmn}\otimes |T_l|^{ijk}_{lmn}\right)\right)\\
 &=  *_{F_r}*_{F_l}\left(X\otimes \sum_{l\in I_{34}, m\in I_{14}}\bb(l)\bb(m) \ast^{kl}_m \left(  g_{j,l} \,  (\delta^{kl}_{m}\otimes \delta^{ij}_k)\right)\right)\\
&=  *_{F_r}*_{F_l}\left(X\otimes \sum_{l\in I_{34}, m\in I_{14}}   \bb(l)\bb(m)\,  g_{j,l} \dim(H^{kl}_m)  \, \delta^{ij}_k\right), \end{align*}
where the third equality follows from the first inversion relation  and Condition~\eqref{I:comp2} in Section~\ref{SS:AlgPre}. The existence of the admissible coloring $\Phi$ and Condition~\eqref{I:comp2+} of Section~\ref{SS:AlgPre} imply that in the latter expression $g_{j,l}=1$ for all $l\in I_{34}$. Therefore this expression is equal to
\begin{align*}  &*_{F_r}*_{F_l}\left(X\otimes \sum_{l\in I_{34}, m\in I_{14}}     \bb(l)\bb(m^*)\dim(H^{lm^*}_{k^*})\,  \delta^{ij}_k\right)\\
&=  *_{F_r}*_{F_l}\left(X\otimes \sum_{l\in I_{34}, m'\in  I_{41}}   \bb(l)\bb(m')\dim(H^{lm'}_{k^*}) \, \delta^{ij}_k\right)\\
& = *_{F_r}*_{F_l}\left(X\otimes \bb(k^*) \,  \delta^{ij}_k\right) = *_{F_r}*_{F_l}\left(X\otimes \bb(k) \,  \delta^{ij}_k\right) =*_{F}(\bb(k)X)=\Kas_{\p'},
\end{align*}
where the second    equality is ensured by Condition~\eqref{I:comp4} in Section~\ref{SS:AlgPre}.  This  proves the lemma for  the {\bubble} moves. Similarly,  the {\Pachner} move translates into the charged pentagon identity  \eqref{E:ChPentId}, see Figure \ref{F:col_32move}.
\begin{figure}[b]
 $\epsh{tetra3}{110pt}$
 \put(-45,20){{\color{red}\small $j$ }}
  \put(-32,10){\small $j_0$}
 \put(-83,32){\small $j_1$}
   \put(-67,17){\small $j_2$}
    \put(-58,-27){\small $j_3$}
   \put(-25,-28){\small $j_4$}
 \put(-75,-6){\small $j_5$}
 \put(-88,-22){\small $j_6$}
 \put(-17,34){\small $j_7$}
 \put(-33,-13){\small $j_8$}
 \put(-108,0){\small $v_1$}
  \put(-45,54){\small $v_2$}
  \put(-70,-22){\small $v_3$}
 \put(-45,-55){\small $v_4$}
  \put(1,6){\small $v_5$}
\caption{Labeling for {\Pachner} move}
 \label{F:col_32move}
\end{figure}
\end{proof}

\begin{lemma}\label{L:Coboundary}
 Let $v$ be a vertex of $\T$ and let $g\in   \Gr $ be such that $I_g\neq \emptyset$.  If $\wp$ and $\delta^{v,g} \wp$ are
 admissible $\Gr$-colorings of $\T$, then
 $\Kas(\T,\LL,\wp,c)=\Kas(\T,\LL,\delta^{v,g} \wp, c)$, where $\delta^{v,g}$ is the $\Gr$-gauge of $\T$ defined   in \eqref{E:DefDeltaGauge}.
\end{lemma}
\begin{proof}
 A similar claim in a simpler setting (no  charges and   $\Gr$ is   abelian) was established in  \cite{GPT2}, Lemma 27.  The proof there relies on Lemma 26 of the same paper.  Replacing Lemma 26 by Lemma \ref{L:one-movezz} above and making the appropriate  adjustments, we easily adapt the argument  in \cite{GPT2} to the present setting.
\end{proof}

\begin{lemma}\label{C:ConstCobound}   If admissible $\Gr$-colorings $\wp$ and $\wp'$
  of $\T$ represent  the same element of $\MG M$, then $\Kas(\T,\LL,\wp,c)=\Kas(\T,\LL,\wp',c)$ for any charge $c$.
\end{lemma}
\begin{proof}
 Since $[\wp]=[\wp']$, there are pairs $(v_i,g_i)\in\{\text{vertices of $\T$}\}\times \Gr$ such that
 $$
\Phi'=\delta^{v_n,g_n} \delta^{v_{n-1},g_{n-1}} \cdots \delta^{v_1,g_1} \Phi.
$$
 Note that the gauges $\delta^{v,g}$  and   $\delta^{v',g'}$ commute for all $g, g'\in G$ provided the vertices $v, v'$ are distinct. Using this property and the identity
$\delta^{v,g}\delta^{v,g'}=\delta^{v,gg'}$, we can ensure that all the vertices $v_i$ in the expansion of $\Phi'$ are pairwise distinct.

We prove the  lemma  by induction on $n$.  If $n=0$,  then $\wp'=\wp$ and there is nothing to prove.  For $n\geq 1$, pick any $g\in \Gr$ such that the sets $I_g$, $I_{gg_1^{-1}}$, $I_{g\Phi(e)}$,  and $I_{gg_1^{-1} \Phi'(e)}$ are non-empty for all  oriented  edges $e$ of $\T$
outgoing from $v_1$.     Then the  colorings $\delta^{v_1,g}\wp$ and $\delta^{v_1,gg_1^{-1}}\wp' $ are admissible.  Clearly, $$\delta^{v_1,gg_1^{-1}}\wp'=\delta^{v_n,g_n}\delta^{v_{n-1},g_{n-1}}\cdots\delta^{v_2,g_2} \delta^{v_1,g}\wp.$$ Lemma
 \ref{L:Coboundary} and the induction assumption imply that
 \begin{align*}
 \Kas(\T,\LL,\wp,c)&=\Kas(\T,\LL,\delta^{v_1,g} \wp,c)\\
 &=\Kas(\T,\LL,\delta^{v_n,g_n}\delta^{v_{n-1},g_{n-1}}\cdots\delta^{v_2,g_2}\delta^{v_1,g}  \wp,c)\\
 &=\Kas(\T,\LL,\delta^{v_1,gg_1^{-1}}\wp' ,c) =\Kas(\T,\LL,\wp',c).
 \end{align*}
\end{proof}

Lemma~\ref{C:ConstCobound}  implies that  $\Kas(\T,\LL,\wp,c) $ depends only on the  element  of $\MG M$ represented by $\wp$.  We represent any $h\in \MG M $
 by an admissible $\Gr$-coloring $\wp$ of $\T$ and set  $\Kas(\T,\LL,h,c) =\Kas(\T,\LL,\wp,c) $.
 The scalar $\Kas(\T,\LL,h,c)$ is invariant under negative
 {\Pachner}/{\bubble} moves. More precisely,  Lemma~\ref{L:one-movezz} implies that under the assumptions of this lemma,
 for any $h\in \MG M$, we have (up to multiplication by  powers of~$\widetilde q$) \begin{equation}\label{neneb}\Kas(\T,\LL,h,c)=\Kas(\T',\LL',h,c'). \end{equation}


\begin{lemma}\label{L:IndCharge} For   any $h\in \MG M$, the scalar
$\Kas(\T,\LL,h,c)$ does not depend on the choice of the charge $c$ in its cohomology class.
\end{lemma}
\begin{proof}
The lemma is proved in two steps: first,  any two charges are connected by a finite sequence of local modifications and second,  the state sum is shown to be preserved under these  modifications.   Recall from Section \ref{sect-MMMLLL}  the set  $\widehat{E}(\T)$    of pairs (a tetrahedron $T$ of $\T$, an edge of $T$).
Let $\ast$ be the involution on $\widehat{E}(\T)$ carrying a pair $(T,e)$ to the pair $\ast (T, e)=(T, \ast e)$ where $\ast e$ is the edge of $T$ opposite to $e$. Fix from now on  an order on the set of vertices of $\T$.

For each edge $e$ of $\T$ we define a map $d(e): \widehat{E} (\T)\rightarrow \{-1/2,0,1/2\}$ as follows.   Let $v$ be the vertex of $e$ which is largest in the ordering of vertices.   Let $T_0, T_1, T_2,...,T_n=T_0$ be the cyclically ordered tetrahedra of $\T$ adjacent to $e$. We choose the cyclic order so that the induced orientation in the plane transversal to $e$ followed by the orientation of $e$ towards $v$ determines the given orientation of $M$.     For $i=1, \ldots , n$, denote   by $ e_i$ the only  edge of the 2-face $T_{i-1} \cap T_{i}$ which is distinct from $e$ and incident to  $v$.
For any $a \in \widehat{E}(\T)$, set
$$d(e)(a)=\left\{\begin{array}{ll}
1/2\, &\text{if  $a=(T_{i-1}, e_i)$ or $a=\ast (T_{i-1}, e_i)$  for some $i\in\{1,...,n\}$,} \\
-1/2 \, &\text{if  $a=(T_{i}, e_i)$ or $a=\ast (T_{i}, e_i)$ for some $i\in\{1,...,n\}$,}   \\
0 \, &\text{otherwise}.
 \end{array}\right.
 $$
 It is easy to see that for any family of  integers $\{\lambda_e \}_e$ numerated by  the edges $e$ of $\T$ the sum $c+\sum_{e }\lambda_e d(e)$ is a charge of $(\T,\LL)$ and $[c+\sum_{e }\lambda_e d(e)]=[c]$.
The following  is due to  Baseilhac   \cite{Bas}, see also     Neumann \cite{Neu90} and     \cite{BB}, Proposition 4.8:\\


\noindent
\textbf{Fact.}
For any pair of charges $c,c'$ of $(\T,\LL)$ with $[c]=[c']$, there is a family of  integers $\{\lambda_e \}_e$ numerated by  the edges $e$ of $\T$  such that  $c'=c+\sum_{e } \lambda_e d(e)$.\\

Therefore to prove the lemma, it is enough to show that for any edge $e$ of $\T$,
\begin{equation}\label{E:tvchangeofcharge}
\Kas(\T,\LL,h,c)=\Kas(\T,\LL,h,c+  d(e)).
\end{equation}
We  show how to reduce the case $e \in \LL$ to the case $e \notin \LL$. Suppose that $e \in \LL$. Pick a face $F$ of $\T$ containing $e$.  We apply the {\bubble} move to $F$ producing a new $H$-triangulation $(\T_b,\LL_b)$ such that    $e$ viewed as an edge of $\T_b$   does not belong to   $\LL_b$.   Pick a charge $c_b$ on $\T_b$ inducing the charge $c$ on $\T$.  A direct calculation shows that the charge $c_b+d(e)$ on $\T_b$ induces the  charge $c +d(e)$ on $\T$.   By \eqref{neneb}, $\Kas(\T,\LL,h,c)=\Kas(\T_b,\LL_b,h,c_b)$ and $\Kas(\T,\LL,h,c+d(e))=\Kas(\T_b,\LL_b,h,c_b+d(e))$.   Therefore, it is enough to prove  \eqref{E:tvchangeofcharge}    in the case $e \notin \LL$.

Suppose from now on that $e \notin \LL$. We first reduce the proof of \eqref{E:tvchangeofcharge}  to the case where $e$ is contained in at least 3 tetrahedra of $\T$ and $\T$ has a vertex  such that there is precisely one face of $\T$ containing   $e$ and this vertex.  Let $V$ be the set of vertices of $\T$. Pick a  face of $\T$ not containing  $e$ and having at least one side in $\LL$. We  apply to this face the {\bubble} move producing a new $H$-triangulation $(\T',\LL')$ whose set of vertices is the union of $V$ with  a 1-point set   $\{v_0\}$.  Pick a charge $c'$ on $\T'$ inducing the charge $c$ on $\T$.
We can find  a sequence $T_0, T_1,..., T_m$ of  distinct tetrahedra of $\T'$  with $m\geq 1 $ such that:   $v_0\in T_0$ and $T_0\cap T_1$ is the  face of $T_0$ opposite to $v_0$;   $T_i \cap T_{i+1}$ contains a common face of $T_i$ and $T_{i+1}$ for all $i$;
 $T_i$ does not contain $e$ for all $i<n$ and   $e\subset T_m$.  Generally speaking,  the intersection $T_i \cap T_{i+1}$ may contain more than  one face; we pick   any face in this intersection and denote it $T_i \cap' T_{i+1}$. We apply to $\T'$ a sequence of  $m$ positive {\Pachner} moves. The first move  replaces the pair $T_0, T_1$ by 3 tetrahedra and adds an edge connecting $v_0$ to the vertex of $ T_1$ opposite to $T_0\cap' T_1=T_0\cap T_1$. One of these new 3 tetrahedra, $t$,  contains the face $T_1\cap' T_2$. The second move replaces the pair $t, T_2$ by 3 tetrahedra and adds  an edge connecting $v_0$ to the vertex of $ T_2$ opposite to $T_1\cap' T_2$.  Continuing in this way $m$ times, we   transform $(\T',\LL')$   into a new $H$-triangulation $(\T'',\LL'')$ having a (unique) face that contains both $v_0$ and  $e$.  The triangulation $\T''$ and all the intermediate triangulations are   quasi-regular   because  the newly added edges always
 connect $v_0$ to another vertex (belonging to $V$).  Our transformations preserve the set $V\cup \{v_0\}$ of the vertices of the triangulation  and lift to  the  charges (though non-uniquely). If the charge $c'$ of $\T'$ is transformed into a charge $c_k$ at the $k$-th step, then $c_k+d(e)$ is a transformation of $c'+d(e)$ (this is obvious for   $k<m$   and is verified by a direct computation  for $k=m$).    Set  $c''=c_m$  and observe that as above,  $$\Kas(\T,\LL,h,c)= \Kas(\T',\LL',h,c')=\Kas(\T'',\LL'',h, c'')$$ and $$\Kas(\T,\LL,h,c+  d(e))= \Kas(\T',\LL',h,c'+  d(e))
 =\Kas(\T'',\LL'',h, c''+d(e)),$$ where  on the right hand side we view $e$ as an edge of $\T''$.  Note that $e$ is contained in at least 3 tetrahedra of $\T''$  because at the $(m-1)$-st step   the edge $e$ is contained in $T_m$ and in at least one other   tetrahedron of the triangulation, and
 the $m$-th  move above creates three tetrahedra of which  two   contain  $e$. Moreover,  there is precisely one face of $\T''$ containing   $e$ and the vertex $v_0$.

 Let $A_1, A_2$ be the vertices of $e$ and  $t_1, t_2,...,t_n$ with $n\geq 3$ be the cyclically ordered  tetrahedra of $\T''$ adjacent to $e$. Let
 $B_1=v_0, B_2, ..., B_n$ be the  vertices of $\T''$ (possibly with repetitions) such that $ A_1, A_2, B_i, B_{i+1}$ are the vertices of $t_i$ for all $i$ (here  $B_{n+1}=B_1$).    Clearly,  $B_i\neq B_1 $ for all $i\neq 1$. If $n>3$, then we apply to $\T''$  a   positive {\Pachner} move replacing $t_1, t_2$ by 3 tetrahedra and adding an edge connecting $B_1=v_0$ to $B_3$. This produces a quasi-regular triangulation $(\T''', \LL'')$ of $(M, L)$ with the same properties as $\T''$ but  having $n-1$ tetrahedra adjacent to $e$. As above,
$\Kas(\T'',\LL'',h, c'')=\Kas(\T''',\LL''',h, c''')$   and
$$\Kas(\T'',\LL'',h, c''+d(e))=\Kas(\T''',\LL''',h, c'''+d(e))$$ for a certain charge $c'''$  on $\T'''$. Proceeding by induction, we  reduce ourselves to the case $n=3$. In this case, the edge $e$ may be eliminated by a negative  {\Pachner} move so that the equality \eqref{E:tvchangeofcharge}
 follows   from \eqref{neneb}.
\end{proof}

\begin{proof}[Proof of Theorem \ref{T:MainTopInv}] Theorem \ref{T:MainTopInv}   follows from the results of this section and
 Proposition \ref{L:Toplemma}.
\end{proof}

\section{Cayley--Hamilton Hopf algebras}\label{section12}
In this section, we recall some of the results of \cite{DeCon} and
use them  to construct $\Psi$-systems in categories.  We assume that
the ground field $\FK$ is algebraically closed and is of
characteristic $0$.

\subsection{Cayley--Hamilton algebras}
\begin{definition}
An  \emph{algebra with trace} is
an (associative) algebra $\CH$ over $\FK$  with a $\FK$-linear map
$t\colon \CH\to \CH$ such that for all $a,b\in \CH$,
$$
t(a)b=bt(a), \quad t(ab) = t(ba),\quad {\text {and}} \quad  t(a)\, t(b)=t(t(a)b) .
$$
\end{definition}

The image    $t(\CH)$ of $t$  is a subalgebra of $\CH$ called the \emph{trace subalgebra}.  Note that   $t(\CH)$ is contained in the center $Z$ of $\CH$.

In the polynomial algebra $\FK[x_1,\ldots,x_n]$ define the elementary symmetric functions $e_i(x_1,\ldots,x_n)$:
\[
\prod_{i=1}^n(x-x_i)=x^n+\sum_{i=1}^ne_i(x_1,\ldots,x_n)x^{n-i},
\]
and the Newton functions
\[
\psi_k(x_1,\ldots,x_n)=\sum_{i=1}^nx_i^k,\quad 1\le k\le n,
\]
which are well known to be related
\[
e_i(x_1,\ldots,x_n)=P_i(\psi_1(x_1,\ldots,x_n),\ldots,\psi_i(x_1,\ldots,x_n))
\]
for certain uniquely defined polynomials $P_i(y_1,\ldots,y_i)$.
\begin{definition}
An algebra with trace $(\CH, t)$ is an \emph{$n^{\mathrm{th}}$-Cayley--Hamilton} algebra, if
$t(1)=n$ and
\[
a^n+\sum_{i=1}^n P_i(t(a),t(a^2),\ldots,t(a^i))a^{n-i}=0
\]
for any $a\in \CH$.
\end{definition}

A prototypical example of  an $n^{\mathrm{th}}$-Cayley--Hamilton algebra is the matrix algebra
$M_n(\FK)$  of $n\times n$ matrices over $\FK$ with the usual trace (with values in $\FK=\FK\Id\subset  M_n(\FK)$).

Let $(\CH,t)$ be a
finitely generated $n^{\mathrm{th}}$-Cayley--Hamilton algebra with trace subalgebra $A=t(\CH)$.  In the rest of this section we  assume that:
\begin{enumerate}
\item \label{I:Ass1} $\CH$ is prime (that is the product of any two non-zero ideals is non-zero),
\item  $\CH$ is a finite $A$-module,
\item the center $Z$ of $\CH$ is integrally closed,
\item \label{I:Ass4}  $A$ is a finitely generated algebra over $\FK$.
\end{enumerate}

The \emph{reduced trace} of $\CH$ is defined by the formula $t_{\CH/A}=\frac{m}{n}t$, where  $m\geq 1$ is the minimal  divisor of $n$ such that $\CH$ is an $m^{\mathrm{th}}$-Cayley--Hamilton algebra with trace  $\frac{m}{n}t$.  Then  $m=[\CH:A]=\operatorname{dim}_A\CH$ and there exists a reduced trace $t_{\CH/Z}:\CH \to Z$ such that
$
t_{\CH/A} = t_{Z/A}\circ t_{\CH/Z}$. Note that $  [\CH : A] = [\CH : Z][Z : A]$.

  By an $n$-dimensional representation of $(\CH, t)$ we mean an algebra homomorphism
$
\phi:\CH\rightarrow M_n(\FK)
$
which is compatible with traces in the sense that $t(a)= Tr(\phi(a))$ for all $a\in \CH$, where $Tr$ is the standard   trace on $M_n(\FK)$.
Let $V (A)$ be the affine algebraic variety associated to $A$, which can be identified
with the maximal spectrum of $A$ or with the set of homomorphisms $  A\to \FK$.
By Theorem~3.1 of \cite{DeCon}, the
(closed) points of $V (A)$ parametrize semi-simple representations of $(\CH, t_{\CH/A})$ of dimension $m = [\CH : A]$. We can similarly  use the points of  $V (A)$ to
parametrize semi-simple representations of $(Z, t_{Z/A})$
of dimension $p = [Z : A]$.
Since  $\CH$ is a finite $A$-module,  $Z$  also is a finite $A$-module. Then $Z$ is a finitely generated algebra over $\FK$ and its associated affine variety
$V (Z)$ parametrizes semi-simple representations  of $(\CH, t_{\CH/Z})$   of dimension $[\CH : Z]$. Observe finally that  the inclusion $A\subset Z$ defines a morphism of
algebraic varieties $\pi\colon V (Z) \to V (A)$ of degree $p$.

Given a point $x\in V (A)$, denote by $N_x$ the corresponding $m$-dimensional
semi-simple representation of $\CH$. Given a point $P\in V (Z)$, denote by $M_P$ the corresponding
$[\CH : Z]$-dimensional semi-simple representation of $\CH$.
 For $x\in V (A)$, the fiber $\pi^{-1}(x)$ (with multiplicities)  is a  cycle
$\sum h_iP_i$ of degree $\sum h_i=p$, where $P_i\in V(Z)$ and $h_i\geq 1$.  One has the following equality of $\CH$-modules:
\[
N_x = \oplus_i\, h_iM_{P_i}.
\]
The Zariski open subset of $V(A)$ consisting of  the points $x$ such that  $\pi^{-1}(x)$ consists of $p$ distinct points is called the \emph{unramified locus} of   $\CH$. For $x$ in the unramified locus, any   $P_i \in \pi^{-1}(x)$ corresponds to an irreducible representation~$M_{P_i}$.
\subsection{Cayley--Hamilton Hopf algebras}\label{SS:CHHopfA}
\begin{definition}
An \emph{$n^{\mathrm{th}}$-Cayley--Hamilton Hopf algebra} is a Hopf algebra which is also an $n^{\mathrm{th}}$-Cayley--Hamilton algebra  such that  the trace subalgebra is a Hopf subalgebra.
\end{definition}
Assume now that $\CH$ is a Cayley--Hamilton Hopf algebra satisfying  the assumptions \eqref{I:Ass1}-\eqref{I:Ass4} of  the last subsection.
 The co-multiplication on $A$ defines an associative binary
operation on  the variety $\Gr= V (A)$ while the antipode defines the inverse operation, so   $\Gr$ becomes  an algebraic group. One has the following decomposition formula for  $x, y\in \Gr$ (Proposition~5.15 of \cite{DeCon}):
\[
N_x\otimes N_y=mN_{xy}
\]
where $m=[\CH : A]$.
A pair   $x, y\in \Gr$ is \emph{generic} if $x, y$ and  $xy$ lie in the unramified locus.  Then for each point $P\in V (Z)$ lying  in the fiber of $x, y$ or $ xy$ the corresponding representation $M_P$ is irreducible. For $P\in\pi^{-1}(x)$ and $Q\in\pi^{-1}(y)$,  one has the following Clebsch--Gordan decomposition  (Theorem~5.16 of \cite{DeCon}):
\begin{equation}\label{E:ssMM}
M_P\otimes M_Q\simeq\bigoplus_{O\in\pi^{-1}(xy)}M_O^{\oplus h_O^{P,Q}}   \quad {\text {and}} \quad
\sum_{O}h_O^{P,Q}=[\CH : Z]
\end{equation}
for some non-negative integers $h_O^{P,Q}$. Also $\sum_{P,Q}h_O^{P,Q}=m$ for all ${O\in\pi^{-1}(xy)}$. Note that generic pairs $(x,y)$  form a Zariski
open subvariety in $\Gr\times \Gr$.

  For   $x\in G$ we define a set $I_x$ as follows:  if $x$ is in the  unramified locus, then $I_x=\pi^{-1}(x)$, otherwise $I_x=\emptyset$.  Set $I=\bigcup_{x\in G}I_x$ and consider the  family $\{M_P\}_{P\in I}$ of irreducible representations discussed above.

  \begin{theorem}\label{T:CHHopfPsi-system} Let $\cat$ be the monoidal  Ab-category   of $\CH$-modules of finite dimension over $\FK$.
Then $\cat $ has a $\Psi$-system with distinguished simple objects $\{M_P\}_{P\in I}$.
\end{theorem}
\begin{proof}      Any $P\in I$ belongs to $I_x$ for a unique point $x$ of the unramified locus.  Using the antipode $S:\CH\to \CH$, we   associate to the irreducible representation $M_P$ the dual representation $M_P^*$ together with the evaluation morphism of $\CH$-modules $M_P^*\otimes M_P\to \FK$ and the coevaluation morphism $\FK \rightarrow M_P\otimes M_P^*$ determined by $1\mapsto \sum v_i\otimes v_i^*$ where $\{v_i\}$ is a basis of $M_P$ and $\{v_i^*\}$ is the dual basis of $M_P^*$.  The representation $M_P^*$ is isomorphic to $M_{P^*}$ where $P^*=P\circ S\in\pi^{-1}(x^{-1})$ and  $x^{-1}=x\circ S$. Thus, we obtain an involution $I\rightarrow I$, $P\mapsto P^*$ and $\CH$-module morphisms $d_{P^*}: M_{P^*}\otimes M_P\to \FK$, $b_P:\FK \rightarrow M_P\otimes M_{P^*}$   satisfying Equation \eqref{E:dualitybd}.   Equation \eqref{E:isom} follows from \eqref{E:ssMM}.  Thus, $\{M_P\}_{P\in I}$ is a $\Psi$-system. \end{proof}

\begin{theorem}\label{T:CHHopfPsi-system+}
Let  $\bb: I \rightarrow \FK$  be the constant function taking the value $  \frac{1}{m}$.
 The triple  $(\Gr,I,\bb)$ satisfies Conditions \ref{I:comp1}-\ref{I:comp4} of Subsection \ref{SS:AlgPre} (where instead of $\widehat \Psi$-systems we should speak of $\Psi$-systems).
\end{theorem}
\begin{proof}
    Condition \ref{I:comp1} follows from the fact that if $x$ is in the
 unramified locus   then so is $x^{-1}=x\circ S$.  Conditions~\ref{I:comp2} and \ref{I:comp2+} follow from   \eqref{E:ssMM}.  Moreover, \eqref{E:ssMM} implies that if $x_1,x_2,x_1x_2$ are   in the
 unramified locus and $P\in I_{x_1x_2}$, then
\begin{equation*}
\sum_{P_1 \in I_{x_1},\, P_2\in I_{x_2}}
\bb({P_1})\bb({P_2})\dim(H^{P_1P_2}_{P})=\frac1{m^2}\sum_{P_1 \in I_{x_1},\, P_2\in I_{x_2}}
\dim(H^{P_1P_2}_{P})=\frac m{m^2} = \bb(P).
\end{equation*}
This implies  Condition \ref{I:comp4}. Condition \ref{I:comp3} holds since the
 unramified locus is a Zariski open subset of $G=V(A)$.
\end{proof}
\begin{remark}
A Cayley--Hamilton Hopf algebra $\CH$ is \emph{sovereign} if $\CH$ contains a group-like element $\phi$ such that $S^2(x)=\phi^{-1}x\phi$ for all $x\in \CH$.
If $\CH$ is sovereign then the results of \cite{Bi} imply that the category $\cat$ of Theorem \ref{T:CHHopfPsi-system} is sovereign (aka. pivotal).  In this case, the right duality comes from the sovereign structure on $\cat$.  Then the general theory of sovereign categories implies that the operator $C=(AB)^3\in \End(H)$ is the identity.
\end{remark}
\subsection{Examples from quantum groups at roots of unity}
Let $\mathfrak{g}$ be a simple Lie algebra of rank $n$ over $\FK=\C$ with the root system $ \Delta$.   Fix simple roots $\alpha_1,\ldots,\alpha_n\in\Delta^+$ and denote by $(a_{ij})_{i,j=1}^n$ the corresponding Cartan matrix. Denote by $d_i$ the
length of the $i$-th simple root.

For an odd positive integer $N$ denote by $\rofo$ a primitive
complex root of $1$ of order $N$ (in the case of type $G_2$ we
require that $N\notin 3\Z$).

Consider the quantized universal enveloping algebra
$\mathcal{U}_\rofo=U_\rofo(\mathfrak{g})$. It is an associative unital algebra  over $\mathbb{C}$ generated by $K_\mu$, where $\mu$ runs over the weight lattice of $\mathfrak{g}$ and
$E_i$, $F_i$, $i = 1,\ldots, n$ with the defining relations:
\[
K_\mu K_\nu=K_{\mu+\nu},\quad K_0=1,
\]
\[
K_\mu E_i=\rofo^{\alpha_i(\mu)}E_iK_\mu,\quad K_\mu F_i=\rofo^{-\alpha_i(\mu)}F_iK_\mu,
\]
\[
E_iF_j-F_jE_i=\delta_{ij}(K_{\alpha_i}-K_{\alpha_i}^{-1})/(\rofo_i-\rofo_i^{-1}),
\]
\[
\sum_{k=0}^{1-a_{ij}}(-1)^k\begin{bmatrix} 1-a_{ij}\\ k\end{bmatrix}_{\rofo_i}
E_i^{1-a_{ij}-k}E_jE_i^k=0,\quad i\ne j,
\]
\[
\sum_{k=0}^{1-a_{ij}}(-1)^k\begin{bmatrix} 1-a_{ij}\\ k\end{bmatrix}_{\rofo_i}
F_i^{1-a_{ij}-k}F_jF_i^k=0,\quad i\ne j,
\]
where $\rofo_i=\rofo^{d_i}$,
\[
\begin{bmatrix} m\\ k\end{bmatrix}_{\rofo}=\frac{[m]_\rofo!}{[m-k]_\rofo![k]_\rofo!},\quad
[m]_\rofo!=[m]_\rofo[m-1]_\rofo\cdots [2]_\rofo[1]_\rofo,\quad [m]_\rofo=\frac{\rofo^m-\rofo^{-m}}{\rofo-\rofo^{-1}}.
\]
The formulas
\[
\Delta(K_\mu)=K_\mu\otimes K_\mu,
\]
\[
\Delta(E_i)=E_i\otimes 1+K_{\alpha_i}\otimes E_i,
\]
\[
\Delta(F_i)=1\otimes F_i+F_i\otimes K_{-\alpha_i},
\]
define a homomorphism of algebras $\Delta\colon \mathcal{U}_\rofo\to \mathcal{U}_\rofo\otimes \mathcal{U}_\rofo$. There are unique counit and antipode turning   $\mathcal{U}_\rofo$ into a Hopf algebra with comultiplication $\Delta$. We denote by $\mathcal{U}_\rofo^{\pm}$ the subalgebras of $\mathcal{U}_\rofo$ generated by $\{E_i\}_i$ and $\{F_i\}_i$ respectively. The subalgebra generated by $\{K_\mu\}_\mu$ will be denoted by $\mathcal{U}_\rofo^0$. We also consider Hopf subalgebras $\mathcal{B}^{\pm}_\rofo=\mathcal{U}_\rofo^0\otimes\mathcal{U}_\rofo^{\pm}$. It is known that the subalgebras
\(
Z_0^{\pm}\subset \mathcal{B}^{\pm}_\rofo
\)
generated by $E_\alpha^N$, $K_{\alpha_i}^N$ (respectively $F_\alpha^N$, $K_{\alpha_i}^N$) and the subalgebra $Z_0\subset \mathcal{U}_\rofo$ generated by $E_\alpha^N$, $F_\alpha^N$, $K_{\alpha_i}^N$ are central Hopf subalgebras. Moreover,   $\mathcal{B}^{\pm}_\rofo$ and $\mathcal{U}_\rofo$ are Cayley--Hamilton Hopf algebras with trace subalgebras $Z_0^{\pm}$ and $Z_0$ respectively, see \cite{DeCon}. In all these three situations, Theorem \ref{T:CHHopfPsi-system}
produces a monoidal Ab-category $\cat$ with a $\Psi$-system.

\begin{conjecture}\label{conj:odd-roots}
This $\Psi$-system can be extended to a $\spsi$-system in $\cat$ such that there exists a scalar $\wsqop\in \C$ for which the operator $\sqop{}{8}\in \End(H)$ of Lemma \ref{nnn17} is   $T$-equal to   $\wsqop\Id_{\widehat H}\oplus \wsqop^{-1}\Id_{\check H} .$
\end{conjecture}
If this conjecture is true then Theorem \ref{T:MainTopInv} implies
that the state sum arising from $\cat$ with this $\spsi$-system and
the algebraic data of Theorem \ref{T:CHHopfPsi-system+}  is a
topological invariant of the triple (a closed connected oriented
3-manifold $M$, a non-empty  link in $M$, a conjugacy class  of
homomorphisms $\pi_1(M) \to \Gr$). In the next section, we verify
Conjecture~\ref{conj:odd-roots}   for the Borel subalgebra of
$U_\rofo(\sll_2)$   for  any  primitive complex  root  of unity
$\rofo$ of odd order $N$. The corresponding topological  invariant
generalizes the one constructed in \cite{K1} which in the case of links in
the 3-sphere coincides with the $N$-colored Jones polynomial
evaluated at $\rofo$.

\section{The Borel subalgebra of $U_\rofo(\sll_2)$}\label{section13}

\subsection{The $\Psi$-system} As above, fix a positive integer $N$ and a primitive   $N$-th root of unity~$\rofo$.    In what follows,   $\mathbb{Z}_N=\mathbb{Z}/N\mathbb{Z}$.
Consider the Hopf algebra $B_\rofo$ defined  by the following presentation
\[
 \mathbb{C}\langle a^{\pm 1},b\, \vert\ ab=\rofo ba,\ \Delta(a)=a\otimes a,\ \Delta(b)=a\otimes b+ b\otimes 1\rangle.
\]
Following \cite{K1,K2}, we consider the cyclic representations of $B_\rofo$, i.e. the representations  carrying  $b$ to an invertible operator.

Let $G=\mathbb{R}\times\mathbb{R}_{>0}$ be the upper half plane with the  group structure
given by   $(x,y)(u,v)=(x+yu,yv)$. As a topological space, the set
$I= G\setminus (\{0\}\times \mathbb{R}_{>0})$ has  two connected components    $I_\pm=\{(x,y)\in G \, \vert\ \pm x>0\}$.
We fix $\epsilon\in\mathbb{C}$ such that
$
\epsilon^N=-1$. In particular, in the case of odd $N$, we assume that $\epsilon=-1$. We define the $N$-th root function  $\sqrt[N]{x}$ on real numbers $x$ by the condition that it is positive real for positive real $x$ and $\sqrt[N]{x}=\epsilon \sqrt[N]{-x}$ for negative $x$.
Define two maps
\[
u\colon G\to \mathbb{R}_{>0},\quad v\colon G\to \mathbb{R}_{>0}\sqcup \epsilon\mathbb{R}_{>0}
\]
\[
u(g)= u_g=\sqrt[N]{y},\quad v(g)= v_g=\sqrt[N]{x},\quad g=(x,y)\in G.
\]
We have the following properties
\[
u_{gh}=u_gu_h,\quad u_{g^{-1}}=\frac{1}{u_g},
\]
and
\[
v_{g^{-1}}=\epsilon_g\frac{v_g}{u_g},\quad
\epsilon_g=\epsilon^{\pm1},\quad g\in I_{\pm}.
\]
To any $g\in I$, we associate a $B_\rofo$-module $V_g$ which is an $N$-dimensional vector space with a distinguished basis
$\{\bv_i\}_{i\in\mathbb{Z}_N}$, and the (left)  $B_\rofo$-module structure is given by the formulae:
\[
a\bv_i=u_g\rofo^i\bv_i,\quad b\bv_i=v_g\bv_{i+1},\quad i\in\mathbb{Z}_N.
\]
Note that the distinguished basis permits to identify   $V_g$ with    $\mathbb{C}^N$.

In what follows, we need  the following function
\[
\Phi_{g,m}=(-\epsilon_g)^m\rofo^{m(m-1)/2},\quad \overline\Phi_{g,m}=\frac1{\Phi_{g,m}},\quad m\in \mathbb{Z}_N.
\]
\begin{proposition}\label{prop:psi}
In the category of $B_\rofo$-modules, the set of objects $\{V_g\}_{g\in I}$ with the involution $g^*=g^{-1}$ is a $\Psi$-system, where the duality morphisms
\[
d_g\colon V_g\otimes V_{g^*}\to \mathbb{C},\quad b_g\colon \mathbb{C}\to V_g\otimes V_{g^*}
\]
are given by the formulae
\[
d_g(\bv_i\otimes \bv_j)=\left\{\begin{array}{cl}
\Phi_{g,i} &\mathrm{if}\ i+j=0;\\
0 &\mathrm{otherwise},
\end{array}
\right.\quad
b_g(1)=\sum_{i\in\mathbb{Z}_N} \overline\Phi_{g^*,-i}\bv_i\otimes \bv_{-i},
\]
and the multiplicity  spaces $H^{f,g}_h$ are such that  $\operatorname{dim}(H^{f,g}_h)$ is $N$ if $h=fg$ and is zero otherwise.
\end{proposition}
\begin{proof}
It is straightforward to verify that $d_g$ and $b_g$ are morphisms of the category of $B_\rofo$-modules. The dimensions of the multiplicity spaces were calculated in \cite{K1,K2}.
\end{proof}

\subsection{Calculation  of the operators $\Aop$, $\Bop$, $\Lop$, and $\Rop$.} Let us call a pair of elements $g,h\in I$ \emph{admissible} if  $gh\in I$. For an operator $E$ satisfying the equation $E^N=-1$ and  an admissible pair $(g,h)$,
 we associate an operator valued function
\(
\Psi_{g,h}(E)
\)
as a solution of the functional equation
\[
\frac{\Psi_{g,h}(\rofo E)}{\Psi_{g,h}(E)}=\frac{v_g-u_gv_hE}{v_{gh}}.
\]
More precisely, we choose numerical coefficients $\psi_{g,h,m}$, $m\in\mathbb{Z}_N$, such that
\[
\Psi_{g,h}(E)=\sum_{m\in\mathbb{Z}_N}\psi_{g,h,m}(\epsilon E)^m.
\]
The above functional equation translates to the following difference equation
\[
\frac{\psi_{g,h,m}}{\psi_{g,h,m-1}}=\frac{\epsilon^{-1}u_gv_h}{v_g-v_{gh}\rofo^m}.
\]
We fix a unique solution of the latter equation normalized so that $\psi_{g,h,0}=1$. Using the notation of \cite{KMS}, we have
\[
\psi_{g,h,m}=w(v_{gh},u_gv_h/\epsilon,v_g\vert m),\quad \Psi_{g,h}(E)=f\left(v_{gh}/v_g,0\left\vert Eu_gv_h/v_g\right.\right),
\]
where
\[
w(x,y,z\vert m)=\frac{(y/z)^m}{(\rofo x/z;\rofo)_m},\quad x^N+y^N=z^N,
\]
\[
f(x,y\vert z)= \sum_{m\in\mathbb{Z}_N}\frac{(\rofo y;\rofo)_m}{(\rofo x;\rofo)_m}z^m,\quad 1-x^N=(1-y^N)z^N,
\]
and
\[
(x;\rofo)_m=\prod_{j=0}^{m-1}(1-x\rofo^j).
\]

For two operators $U$ and $V$ satisfying the conditions $U^N=V^N=1$, $UV=VU$, we also define
\[
L(U,V)=\frac1N\sum_{i,j\in\mathbb{Z}_N}\rofo^{ij}U^iV^j.
\]
In what follows, the standard basis $\{e_i\}$ of $\mathbb{C}^N$ will be indexed by elements of $\mathbb{Z}_N$. Define operators $X,Y\in\operatorname{Aut}(\mathbb{C}^N)$,
\[
Xe_i=\rofo^ie_i,\quad Ye_i=e_{i+1},\quad i\in\mathbb{Z}_N.
\]
For $g\in I$, let $\pi_g\colon B_\rofo\to \operatorname{End}(V_g)$ be the algebra homomorphism corresponding to the $B_\rofo$-module structure of $V_g$.
\begin{lemma}\label{lem:sgh}
For any admissible pair $(g,h)$, the operator valued function
\[
S_{g,h}=\Psi_{g,h}(-Y^{-1}X\otimes Y)L(Y\otimes 1,1\otimes X)
\] takes its values in the set of invertible matrices and satisfies the equation
\[
(\pi_g\otimes\pi_h)(\Delta(x))=S_{g,h}(\pi_{gh}(x)\otimes \operatorname{id}_{\mathbb{C}^N})S_{g,h}^{-1},\quad \forall x\in B_\rofo.
\]
\end{lemma}
\begin{proof}
A straightforward computation.
\end{proof}
The operator $S_{g,h}$, considered as a linear map $S_{g,h}\colon V_{gh}\otimes \mathbb{C}^N\to V_g\otimes V_h$, permits to identify the multiplicity spaces $H^{g,h}_{gh}$ and $H_{g,h}^{gh}$ with $\mathbb{C}^N$ and $(\mathbb{C}^N)^*$, respectively, through the formulae
\[
fv=S_{g,h}(v\otimes f), \quad f\in H^{g,h}_{gh},\ v\in V_{gh},
\]
\[
f(v\otimes w)=(\operatorname{id}_{V_{gh}}\otimes f)(S_{g,h}^{-1}(v\otimes w)),\quad
v\in V_g,\ w\in V_h,\ f\in H_{g,h}^{gh}.
\]
In what follows, we shall use the notation $e_i$ and $e^*_i$ for the dual bases in $H^{g,h}_{gh}$ and $H_{g,h}^{gh}$, respectively, which correspond to the standard dual bases in $\mathbb{C}^N$ and $(\mathbb{C}^N)^*$.
\begin{lemma}\label{L:ExpAB}
The action of the operators $\Aop$, $\Aop^*$, $\Bop$, and $\Bop^*$ is given by
\begin{multline*}
\Aop e_i=\Psi_{g^*,gh}(\epsilon_g\rofo)\sum_{j\in\mathbb{Z}_N}\Phi_{g^*,i-j}e^*_j,\quad
\Aop^* e_i=\Psi_{g,h}(\rofo/\epsilon_g)\sum_{j\in\mathbb{Z}_N}\Phi_{g,j-i}e^*_j,\\
\quad e_i\in H^{g,h}_{gh},\quad e^*_j\in H_{g^*,gh}^{h},
\end{multline*}
\begin{multline*}
\Aop e^*_i=\frac{1}{N\Psi_{g,h}(\rofo/\epsilon_g)}\sum_{j\in\mathbb{Z}_N}
\overline\Phi_{g,j-i}e_j,\quad
\Aop^* e^*_i=\frac{1}{N\Psi_{g^*,gh}(\epsilon_g\rofo)}\sum_{j\in\mathbb{Z}_N}
\overline\Phi_{g^*,i-j}e_j,\\ e^*_i\in H_{g,h}^{gh},\quad
 e_j\in H^{g^*,gh}_{h},
\end{multline*}
\[
\Bop e_i=\frac{1}{\nu(v_{gh}/v_g)}\Phi_{h,i}e^*_{-i},\quad \Bop^* e_i=\frac{1}{\nu(v_{g}/v_{gh})}\Phi_{h^*,-i}e^*_{-i},\quad
e_i\in H^{g,h}_{gh},\quad e^*_{-i}\in H_{gh,h^*}^{g},
\]
\[
\Bop e^*_i=\nu(v_{g}/v_{gh})\overline\Phi_{h^*,-i}e_{-i},\quad \Bop^* e^*_i=\nu(v_{gh}/v_{g})\overline\Phi_{h,i}e_{-i},\quad e^*_i\in H_{g,h}^{gh},\quad e_{-i}\in H^{gh,h^*}_{g},
\]
where
\[
\nu(x)=\frac{1-x^N}{N(1-x)}.
\]
\end{lemma}
\begin{proof}
From Lemma~\ref{lem:sgh} it follows that for any
\(
x\in V_h,\ y\in H^{g^*,gh}_h,\ z\in H^{g,h}_{gh}
\),
\[
(d_{g^*}\otimes\operatorname{id}_{V_h})(\operatorname{id}_{V_{g^*}}\otimes S_{g,h})(S_{g^*,gh}\otimes \operatorname{id}_{H^{g,h}_{gh}})(x\otimes y\otimes z)=x\langle y,Az\rangle.
\] Starting from this identity, and using the results of  Appendices A and C of \cite{KMS}, a  straightforward   calculation   yields  the formulas above for the action of   $\Aop$.

Similarly, for any
\(
x\in V_g,\ y\in H^{gh,h^*}_g,\ z\in H^{g,h}_{gh}
\) the identity
\[
(\operatorname{id}_{V_g}\otimes d_{h})(S_{g,h}P\otimes \operatorname{id}_{V_{h^*}})(\operatorname{id}_{H^{g,h}_{gh}}\otimes S_{gh,h^*})(z\otimes x\otimes y)=x\langle y,Bz\rangle.
\] gives rise to the action of the operator $\Bop$.
\end{proof}
\begin{lemma}\label{L:ExpLR}
The action of the operators $\Lop=\Aop^*\Aop$ and $\Rop=\Bop^*\Bop$ is given by
\[
\Lop e_i=\left(\frac{u_gv_h}{v_{gh}}\right)^{N-1}e_{i-1},\quad e_{i}\in H^{g,h}_{gh},
\]
\[
\Lop e^*_i=\left(\frac{u_gv_h}{v_{gh}}\right)^{N-1}e^*_{i+1},\quad e^*_{i}\in H_{g,h}^{gh},
\]
\[
\Rop e_i=\rofo^{-i}\left(\frac{v_g}{v_{gh}}\right)^{N-1}e_{i},\quad e_i\in H^{g,h}_{gh},
\]
\[
\Rop e^*_i=\rofo^{-i}\left(\frac{v_g}{v_{gh}}\right)^{N-1}e^*_{i},\quad e^*_i\in H_{g,h}^{gh}.
\]
\end{lemma}
\begin{proof}
The case of $\Rop$ is straightforward. In the case of $\Lop$ we use the   equality
$z^{N-1}f(x,0\vert z \rofo)=x^{N-1}f(z,0\vert x \rofo)$ for the function $f(x,y\vert z)$ of \cite{KMS}.\end{proof}

\subsection{The $\widehat{\Psi}$-system for odd $N$}  Assume from now on that $N$ is odd.
\begin{proposition}\label{P:ExtendsPsiHat}
The $\Psi$-system of Proposition~\ref{prop:psi} extends to a $\widehat{\Psi}$-system
where $\scop{}=\Id_H$ and  $\srop{}$ is given by
\[
\srop{} e_i=\rofo^{-\left(\frac{N+1}{2}\right)i}\left(\frac{v_g}{v_{gh}}\right)^{\frac{N-1}2}e_{i},\quad e_i\in H^{g,h}_{gh},
\]
\[
\srop{} e_i^*=\rofo^{-\left(\frac{N+1}{2}\right)i}\left(\frac{v_g}{v_{gh}}\right)^{\frac{N-1}2}e_{i}^*,\quad e_i^*\in H^{gh}_{g,h}.
\]
\end{proposition}
\begin{proof}
Since $N$ is odd, we can set  $\epsilon=-1$, and for any $g\in I$, the coordinates $u_g$ and $v_g$ are real numbers. In particular, as the operator $r=\Rop^N$ has a positive spectrum, we define $\sqrt{r}$ as the unique  positive operator such that $(\sqrt{r})^{2}=r$.  Define
\[
\srop{}=(\sqrt{r})^{-1}\Rop^{\frac{N+1}2}.
\]
Let us calculate $\srop{}e_i$ for $e_i\in H^{gh}_{g,h}$.  Write  $\Rop=\Rop_0\Rop_1$ as the product of commuting operators $\Rop_0$ and $\Rop_1$ where $\Rop_0=\left(v_g/v_{gh}\right)^{N-1}$ and $\Rop_1 e_j =\rofo^{-j}e_j$ for $e_j\in H^{gh}_{g,h}$.  Then we have $\Rop_1^N=\Id_H$ and $r=\Rop^N=\Rop_0^N\Rop_1^N=\Rop_0^N$ so
$$(\sqrt{r})^{-1}\Rop^{\frac{N+1}2}e_i=(\Rop_0^N)^{-1/2}\Rop_0^{(N+1)/2}\Rop_1^{(N+1)/2}e_i=\left({v_g}/{v_{gh}}\right)^{\frac{N-1}2}\rofo^{-\left(\frac{N+1}{2}\right)i}e_{i}.$$
The computation of  $\srop{}e_i^*$ is similar.

Next consider the operator
$$\Cop=(\Aop\Bop)^3=\Aop(\Bop\Aop\Bop)\Aop\Bop=\Aop\Aop^*\Bop^*\Aop^*\Aop\Bop=\Lop^{-1}\Bop^*\Lop\Bop.$$
From Lemmas \ref{L:ExpAB} and \ref{L:ExpLR} it is easy to see that $\Lop^{-1}\Bop^*\Lop\Bop=\Id_H$.
Thus, we can define $\scop{}=\Id_H$.  Then it is easy to see that $\scop{}, \srop{}$ satisfy Equations \eqref{E:srts} and \eqref{E:sqrtCC} since $\Cop$ and $\Rop$ satisfy analogous formulas without square roots.
\end{proof}

\begin{lemma}\label{L:slopComp}
The operator $\slop{}=\Bop\Aop \srop{-}  \Aop\Bop$ is given by
\[
\slop{} e_i=\left(\frac{u_gv_h}{v_{gh}}\right)^{\frac{N-1}{2}}e_{i-{(N+1)}/{2}},\quad e_{i}\in H^{g,h}_{gh},
\]
\[
\slop{} e^*_i=\left(\frac{u_gv_h}{v_{gh}}\right)^{\frac{N-1}{2}}e^*_{i+{(N+1)}/{2}},\quad e^*_{i}\in H_{g,h}^{gh}.
\]
\end{lemma}
\begin{proof}
Write $\Lop=\Lop_0\Lop_1$ as the product of commuting operators $\Lop_0$ and $\Lop_1$ where $\Lop_0\vert_{H^{g,h}_{gh}\oplus H_{g,h}^{gh}}=\left(\frac{u_gv_h}{v_{gh}}\right)^{N-1}$ and $\Lop_1$ is a translation operator such that $\Lop_1^N=\Id_H$.  We will show that $\slop{}=\Lop_0^{1/2}\Lop_1^{(N+1)/2}$.

Let $r'=\Lop^N=\Bop\Aop r^{-1} \Aop\Bop$ where $r=\Rop^N$.  Then
\begin{align*}
\slop{}&=\Bop\Aop \sqrt{r}\Rop^{-\frac{N+1}{2}}\Aop\Bop
=(\sqrt{r'})^{-1}\Bop\Aop \Rop^{-\frac{N+1}{2}}\Aop\Bop
= (\sqrt{r'})^{-1}L^{\frac{N+1}{2}}.
\end{align*}
Since $r'=L^N=\Lop_0^N$ we have
$$\Lop^\frac12=\sqrt{r'}^{-1}\Lop^{\frac{N+1}2}=(\Lop_0^N)^{-\frac12}\Lop_0^{\frac{N+1}2}\Lop_1^{\frac{N+1}2}=\Lop_0^{\frac12}\Lop_1^{\frac{N+1}2}$$
and the formulas of the lemma follow.
\end{proof}

\begin{lemma}
 The operator $\sqop{}{8}=\srop{}B\slop{} B \slop{-}\scop{-}$ has the following form
  $$\sqop{}{}= (-1)^{(N-1)/2}(\rofo^{-a}\Id_{\hat H}\oplus \rofo^{a}\Id_{\check H})\in \End(H)$$
   where $a=\frac{N^2-1}{8}$.
\end{lemma}
\begin{proof}
 A direct computation using Proposition~\ref{P:ExtendsPsiHat} and Lemma~\ref{L:slopComp} shows that
\begin{align*}
\sqop{}{8}e_i
&=(-1)^{(N-1)/2}\rofo^{\left(\frac{N^2-1}{8}\right)}e_i
\end{align*}
for $e_i\in H^{g,h}_{gh}$.
Similarly, $\sqop{}{8}e_i^*=(-1)^{(N-1)/2}\rofo^{-\left(\frac{N^2-1}{8}\right)}e_i^*$ for $e_i^*\in H_{g,h}^{gh}$.
\end{proof}

 For $g\in \Gr$ set $I_g=\{g\}\subset I$, if $g\in I$  and  $I_g=\emptyset$, otherwise.  Let  $\bb:I\to \C$  be the constant function taking the value $\frac1N$.  Then  the triple  $(\Gr,I,\bb)$    satisfies  Conditions \ref{I:comp1}-\ref{I:comp4} of Subsection \ref{SS:AlgPre}.   We summarize the results of  this section in  the following theorem.
 \begin{theorem}
For any odd $N$, the category of $B_\rofo$-modules has a $\widehat{\Psi}$-system such that the algebraic data $(\Gr,I,\bb)$ satisfies   Conditions \ref{I:comp1}-\ref{I:comp4} of Subsection \ref{SS:AlgPre} and $\sqop{}{}=(-1)^{(N-1)/2}(\rofo^{-a}\Id_{\hat H}\oplus \rofo^{a}\Id_{\check H})$ where $a=\frac{N^2-1}{8}$.
 \end{theorem}

 Thus, the category of $B_\rofo$-modules gives rise to a topological invariant as in Theorem \ref{T:MainTopInv} and as was mentioned above this invariant generalizes that of \cite{K1}.


\section*{Appendix}

The   relations in the Fundamental lemma (Lemma \ref{Fundamental lemma}) express  the action of the  standard generators  of the symmetric group $\mathbb{S}_4$ on the tensors $T$, $\bar T$ or equivalently on the tensors  $S$, $\bar S$ defined at the end of Section \ref{TheformsTand}.  We give  a geometric interpretation of this action   in the case where the operators $A$ and $B$ are symmetric, i.e.,  $A=A^*$, $B=B^*$. This interpretation involves a combinatorial 3-dimensional TQFT  which we now define.

Consider a compact oriented surface  (possibly with boundary)  endowed with  oriented  cellular structure $\Sigma$ such that all 2-cells are either bigons or triangles. For example, for any $\varepsilon, \mu\in \{-1,+1\}$, the unit disk $D$ in $\mathbb{C}$ has such a structure  consisting of  a single bigon with two 0-cells $\{ \pm1\}$ and  two 1-cells $e^1_\pm\colon [0,1]\to D$ given by
$
e^1_+(t)= \varepsilon e^{\varepsilon i\pi t}$ and $
e^1_-(t)=\mu e^{-\mu i\pi t}$ where $t\in [0,1]$. 

A bigon of $\Sigma$ is  \emph{inessential} if its  edges are co-oriented. In the example above, the   cellular structures with $\varepsilon=\mu$ are inessential.  
A triangular 2-cell of $\Sigma$ is positive (resp.\ negative) if the orientation of precisely two (resp.\  one) of its edges is compatible with that of the cell itself.  We shall consider only cellular structures $\Sigma$ without inessential bigons and such that all the triangular cells are either positive or negative.  We   associate the 1-dimensional vector space $\mathbb{C}$ to all bigons,   the vector space  $\hat H$ to all positive triangles  and $\check H$ to all negative triangles. Finally, we associate with $\Sigma$   the  tensor product over $\mathbb{C}$  of these vector spaces numerated by the 2-cells of $\Sigma$. It is isomorphic to   $\hat H^{\otimes m_+}\otimes \check H^{\otimes m_-}$, where $m_+$ (resp.\ $m_-$) is   the number  of positive  (resp.\  negative) triangles of $\Sigma$.

Next, we define  elementary 3-cobordisms. An oriented tetrahedron in $\mathbb{R}^3$ with ordered vertices has a natural cell structure, where the orientation on the edges is induced from the order. Such a tetrahedron is {\it positive} if the oriented
edges (12,13,14)  form a positive basis in $\mathbb{R}^3$ and {\it negative} otherwise. We associate with a positive (resp.\ negative) tetrahedron  the tensor $S\in \hat H\otimes \check H\otimes\hat H\otimes \check H$ (resp.\  $\bar S\in  \check H\otimes\hat H\otimes \check H\otimes\hat H$). Here the   face opposite to the $i$-vertex corresponds to the $i$-th tensor factor for $i\in\{1,2,3,4\}$.

Next, we consider cones over essential bigons with induced cellular structure. The orientation condition on the triangular cells leaves four isotopy classes of such cones. We describe them for  the cone   over the unit disk   $D\subset \mathbb{C}$ with the cone point  $(0,1)\in\mathbb{C}\times\mathbb{R}$. The four possible cellular structures have  three 0-cells $\{ (\pm 1,0),  (0,1)\}$  and four 1-cells 
\[
\{e^1_{0\pm}(t)=\pm e^{i\pi t},\ e^1_{1\pm}(t)=(\pm(1-t),t)\}
\] 
or
\[
\{e^1_{0\pm}(t)=\mp e^{-i\pi t},\ e^1_{1\pm}(t)=(\pm(1-t),t)\}
\] 
or
\[
\{e^1_{0\pm}(t)=\pm e^{i\pi t},\ e^1_{1\pm}(t)=(\pm t,1-t)\}
\]
or else
\[
\{e^1_{0\pm}(t)=\mp e^{-i\pi t},\ e^1_{1\pm}(t)=(\pm t,1-t)\}.
\]
Let us call them cones of type $a_+$, $a_-$, $b_+$, and $b_-$, respectively. We   associate to these cones respectively the operators $A\vert_{\check H}$, $A\vert_{\hat H}$, $B\vert_{\check H}$, and $B\vert_{\hat H}$ viewed as vectors in $\hat H^{\otimes 2}$ or $\check H^{\otimes 2}$. Note that our cones are invariant under rotation by the angle $\pi$ around the vertical coordinate axis, and this invariance leads here to the condition $A=A^*$, $B=B^*$. 

Now we give a TQFT interpretation of the formulae \eqref{E:sym01S}--\eqref{E:sym23S}. Let us take, for example the right hand side of \eqref{E:sym01S}. The form $\bar S$ corresponds to a negative tetrahedron. The edge joining the first and the second vertices is incident to two faces opposite to the third and  the forth vertices. We can glue two cones, one of type $a_+$ and another one of type $a_-$, to these two faces in the way that one of the edges of the base bigons are glued to the initial edge $12$ and of course by respecting all orientations. Namely, we glue the cone of type $a_+$ to the face opposite to the third vertex of the tetrahedron so that the tip of the cone is glued to the forth vertex, and we glue the cone of type $a_-$ to the face opposite to the forth vertex of the tetrahedron so that the tip of the cone is glued to the third vertex. Finally, we can glue naturally the two bigons with each other by pushing continuously the initial edge inside the ball and eventually closing the gap like a book. The result of all these operations is that we obtain a positive tetrahedron, where the only difference with respect to the initial tetrahedron is that the orientation of the initial edge $12$ has changed and this corresponds to changing the order of its vertices. Notice that as these vertices are neighbors, their exchange does not affect the orientation of all other edges. On the other hand, as the order of the tensor components in our TQFT rules for the tetrahedra is matched with the order of vertices, we have to exchange also the first two tensor components. This explains the left hand side of \eqref{E:sym01S}. The other two relations are interpreted in a similar manner.

\linespread{1}

\vfill

\end{document}